\documentclass[12pt]{article}
\usepackage{amsmath}
\usepackage{graphicx,psfrag,epsf}
\usepackage{enumerate}
\usepackage{url} 
\usepackage{float}
\usepackage[table,xcdraw,dvipsnames]{xcolor}

\usepackage{cite}  

\usepackage{multirow}
\usepackage{lipsum}
\usepackage{amsfonts}
\usepackage{cite}      
                  
\usepackage{graphicx}   

\usepackage{amsthm}
\RequirePackage[colorlinks,citecolor=blue,urlcolor=blue]{hyperref}
\usepackage{amssymb,calc}
\usepackage{thmtools}
\usepackage{mathrsfs}
\usepackage{mathtools}
\usepackage{bm}
\usepackage{bbm}
\usepackage{enumerate}
\usepackage{rotating}
\RequirePackage{cleveref}
\usepackage{hyphenat}
 \usepackage[numbers]{natbib}
\usepackage{caption,subcaption}
\usepackage{tikz}
\usepackage{pgfplots}

\usepackage[T1]{fontenc}
\usepackage{textcomp} 
\usepackage{lmodern}

\usepackage[noend]{algpseudocode}
\usepackage{algorithm}
\makeatletter
\renewcommand{\ALG@name}{Table}
\makeatother

\pgfplotsset{width=10cm}

\tikzset{declare function={gamma(\x)=sqrt(2*pi)*\x^(\x-0.5)*exp(-\x)*exp(1/(12*\x));}}
\tikzset{declare function={tpdf(\x,\nu)=gamma(0.5*(\nu+1))/(sqrt(pi*\nu)*gamma(\nu/2))*(1+\x^2/\nu)^(-(\nu+1)/2);}}
\tikzset{declare function={invgampdf(\x,\a,\b)=(\b/\x)^\a/\x/gamma(\a)*exp(-\b/\x);}}

\newcommand{\nhphantom}[1]{\ifmmode\settowidth{\dimen0}{$#1$}\else\settowidth{\dimen0}{#1}\fi\hspace*{-\dimen0}}

\usetikzlibrary{patterns}
\tikzset{
	hatch distance/.store in=\hatchdistance,
	hatch distance=5pt,
	hatch thickness/.store in=\hatchthickness,
	hatch thickness=0.5pt,
}

\definecolor{pink}{rgb}{0.9, 0.17, 0.31}

\def\C {\,|\:}

\newcommand\E{\mathbb E}

\renewcommand\d{\mathrm d}
\newcommand\iid {\overset{\mathrm{iid}}{\sim}}

\renewcommand\P{\mathbb P}

\newcommand\R{\mathbb R}
\renewcommand\b{\bm{\beta}}
\newcommand\e{\mathrm e}

\usepackage[toc,page,header]{appendix}
\usepackage{minitoc}

\newcommand{\wt}[1]{\widetilde{#1}}

\newtheorem{lemma}{Lemma}
\newtheorem{theorem}{Theorem}
\newtheorem{remark}{Remark}

\renewcommand{\nhphantom}[1]{\ifmmode\settowidth{\dimen0}{$#1$}\else\settowidth{\dimen0}{#1}\fi\hspace*{-\dimen0}}

\numberwithin{equation}{section}

\newtheorem{definition}[theorem]{Definition}
\newtheorem{ass}[remark]{Assumption}

\crefname{thm}{Theorem}{Theorems}
\crefname{prop}{Proposition}{Propositions}
\crefname{lem}{Lemma}{Lemmas}
\crefname{coro}{Corollary}{Corollaries}
\crefname{add}{Addendum}{Addendums}
\crefname{asm}{Assumption}{Assumptions}
\crefname{alg}{Algorithm}{Algorithms}
\crefname{proc}{Procedure}{Procedures}
\crefname{exe}{Exercise}{Exercises}
\crefname{exa}{Example}{Examples}
\crefname{prob}{Problem}{Problems}
\crefname{section}{Section}{Sections}
\crefname{subsection}{Section}{Sections}
\crefname{appendix}{Appendix}{Appendices}

\allowdisplaybreaks[3]

\begin{document}

\doparttoc 
\faketableofcontents 


\def\spacingset#1{\renewcommand{\baselinestretch}%
{#1}\small\normalsize} \spacingset{1}


	\title{\sf  Adaptive Bayesian Predictive Inference for Sparse High-dimensional Regression}
	\author{Veronika Ro\v{c}kov\'{a}\footnote{Veronika Ro\v{c}kov\'{a} is   professor of Econometrics and Statistics at the Booth School of Business of the University of Chicago. The author gratefully acknowledges support from the James S. Kemper Faculty Fund and the National Science  Foundation (DMS:1944740). The author would like to thank Johannes Schmidt-Hieber and Edward George for useful conversations and encouraging remarks.}
			} 
	\maketitle

\bigskip
\begin{abstract}
Bayesian predictive inference provides a coherent description of  entire predictive uncertainty through predictive distributions.
We examine several widely used sparsity priors from the predictive (as opposed to estimation) inference viewpoint.  To start, we investigate predictive distributions in the context of a high-dimensional Gaussian observation with a known variance but an  unknown sparse mean   under the Kullback-Leibler loss. First, we show that LASSO (Laplace) priors  are incapable of achieving  rate-optimal predictive distributions.
However, deploying the Laplace prior inside the Spike-and-Slab framework (e.g. with the Spike-and-Slab LASSO prior),  rate-minimax performance  can be attained with properly tuned parameters (depending on the sparsity level $s_n$).  We highlight the discrepancy between  prior calibration for the purpose of prediction  and estimation.  
 Going further, we investigate popular hierarchical priors which are known to attain {\em adaptive} rate-minimax performance for estimation. 
  Whether or not they are rate-minimax also for predictive inference has, until now,  been unclear. We answer affirmatively by showing that hierarchical Spike-and-Slab priors are adaptive and   attain the minimax rate {\em without} the knowledge of $s_n$. This is the first rate-adaptive result in the literature on predictive density estimation in sparse setups. Building on the sparse normal-means model, we extend our adaptive rate findings to the case of sparse high-dimensional regression with Spike-and-Slab priors. All of these results underscore benefits of fully Bayesian predictive inference.

\end{abstract}

\noindent%
\vfill

\newpage
\spacingset{1.45} 

\section{Introduction}
A fundamental goal in predictive inference is using observed data to estimate an {\em entire} predictive distribution of a future observation.
The  Bayesian approach offers a complete description of  predictive uncertainty   through posterior predictive distributions. 
Distributional predictions, as opposed to mere point predictions, are a valuable decision-making device for  practitioners.
 
This paper focuses on the problem of predicting a Gaussian vector whose mean depends on an unknown sparse high-dimensional parameter. First, we investigate predictive inference for the quintessential sparse normal-means model. Later, we extend our findings to  high-dimensional sparse linear regression which is widely used in practice.  
There is a wealth of  literature on  {\em estimation theory} in terms of rate-minimax posterior concentration (under the $\ell_2$ loss) for various priors, both for the $\ell_0$-sparse Gaussian sequence model 
 (point-mass Spike-and-Slab priors \cite{castillo_vdv},   Spike-and-Slab LASSO priors \cite{rockova} or continuous shrinkage priors \cite{dirichlet_laplace,van_der_pas,carvalho})   and sparse high-dimensional regression \cite{castillo_vdv2,RG14b}. Lower bound results have also been established \cite{castillo_vdv2},  showing that the Laplace prior \cite{casella} that   yields a rate-optimal posterior mode in sparse regression, at the same time, yields  a rate-suboptimal posterior distribution.  Rate-minimaxity  has become a useful criterion for prior calibration in order to acquire a frequentist license for   Bayesian estimation procedures. This work focuses on rate-minimax certification in the context of Bayesian {\em predictive inference}.
 While there are  decision-theoretic parallels between predictive density estimation under the Kullback-Leibler (KL) loss and point estimation under the quadratic loss   \cite{george_liang_xiu,george_liang_xu2,muk1},  the prediction problem is intrinsically different and may require different prior calibrations. This invites the question: ``{\em How to calibrate priors in order to obtain rate-minimax optimal predictive density estimates?}". To answer this question, this paper examines  popular shrinkage priors but from a predictive point of view. We build on   \cite{muk1} and \cite{muk2} but our work is different in at least three fundamental ways. First,  we focus on (a) popular priors that are widely used in practice (such as the Bayesian LASSO \cite{casella} and  the Spike-and-Slab LASSO  \cite{rockova}),  and (b) we study {\em adaptation} to sparsity levels through hierarchical priors (none of the proposed estimators  in \cite{muk1} and \cite{muk2} are adaptive). Lastly, (c) in addition to the Gaussian sequence model, we study the high-dimensional sparse regression model which is indispensable to practitioners. 

\citet{muk1}  were first to study predictive density estimation for a multivariate normal vector with $\ell_0$ sparse means (with up to $s_n$ signals) and found several fascinating parallels between sparse normal means estimation and predictive inference under the KL loss. In addition, \cite{muk1} constructed a predictive density estimator (inspired by hard thresholding) by pasting together two predictive density estimators (under the uniform prior and under a discrete ``cluster" prior) depending on the magnitude of the observed data.  While  minimax optimal (up to a constant), this estimator is not entirely Bayesian since it is not a posterior predictive distribution under any prior. This estimator also relies on discrete priors (which may not be as natural to implement in practice) and is not smooth  with respect to the input data.  Finally, these results are non-adaptive, i.e. the knowledge of sparsity level $s_n$ is required to construct the optimal estimator. In a followup paper, \cite{muk2} proposed proper Bayesian predictive densities under discrete Spike-and-Slab priors (i.e. sparse univariate symmetric  priors with {\em discrete} slab distributions) and showed that they are minimax-rate optimal. Again, these results are unfortunately non-adaptive, where the knowledge of $s_n$ is required to tune/construct the prior. Moreover, \cite{muk2} again mainly focused on discrete slab priors which may not be as practical. Uniform continuous slab distributions were shown to be minimax-rate optimal if the parameter space is suitably constrained. All these existing results have been obtained only under the Gaussian sequence model.

Our work  provides new insights into predictive performance of shrinkage priors that are widely used in practice. Our goal is providing guidelines for calibrating popular priors in the context of prediction. 
Our contributions can be summarized as follows: (1) we study Bayesian LASSO priors and show that the predictive distributions are incapable of achieving rate-minimax performance, (2) we study Laplace-induced Spike-and-Slab priors (including the popular Spike-and-Slab LASSO prior \cite{rockova}) which have a {\em continuous} slab (versus a discrete slab considered in \cite{muk1}) and show that, if   calibrated by an oracle,  predictive densities are rate-minimax, (3) we investigate hierarchical variants of the Spike-and-Slab prior and show {\em adaptive} rate-minimaxity. In conclusion, no knowledge of $s_n$ is required for hierarchical Spike-and-Slab priors to be rate-minimax optimal. This self-adaptation property is new but is in line with previous findings in the context of estimation \cite{castillo_vdv, rockova}. Results (1)-(3) are obtained under the sparse Gaussian sequence model. Our final contribution is obtaining adaptive rate results in the context of high-dimensional sparse regression under Spike-and-Slab priors. Minimax predictive inference in low-dimensional regression was previously studied by \cite{george_xu}. We investigate the high-dimensional regime where $p>n$ and where the sparsity level $s_n$ is allowed to increase  with the sample size $n$. Focusing on both the total variation distance as well as the (typical) KL divergence, we establish a rate of estimation of  Spike-and-Slab  predictive distributions which is adaptive in $s_n$ and which mirrors the $\ell_2$ estimation rate. Our proof relies on some techniques developed in \cite{castillo_vdv2} but has required several new steps including a novel upper bound on the marginal likelihood under a Laplace prior.

The paper has the following structure. Section \ref{sec:normal_means} is dedicated to the sparse normal means model. Section \ref{sec:BLASSO} presents a lower-bound result showing that Bayesian LASSO is incapable of yielding posterior predictive densities with good properties in sparse setups.  Section  \ref{sec:separable} focuses on Spike-and-Slab priors with a Dirac spike and a Laplace slab (a direct extension of the Bayesian LASSO prior) as well as the Spike-and-Slab LASSO prior \cite{rockova}.   Section \ref{sec:hierarchical} then shows adaptive rate-minimax performance of hierarchical Spike-and-Slab priors. Section \ref{sec:sparse_regression} is dedicated to the sparse high-dimensional regression model. We conclude with a discussion in Section \ref{sec:discuss} and  the proof of Theorem \ref{thm:dirac_laplace} in Section   \ref{sec:proof_thm:dirac_laplace}.

{\bf Notation}
We define $a\lesssim b$ and $a\gtrsim b$ if, for some universal constant $C$, $a\leq C b$ and  $a\geq C b$, respectively. We write $a\sim b$ when $a\lesssim b$ and $a\gtrsim b$. We denote with $\phi(\cdot)$ and $\Phi(\cdot)$ the density and cdf of a standard normal distribution.  The Gaussian Mills ratio will be denoted by $R(x)=[1-\Phi(x)]/\phi(x)$. We will denote with $\E$ expectation with respect to a data generating process and with $E$ expectation over a latent variable.

\section{Predictive Inference in the Sparse Normal Means Model}\label{sec:normal_means}
Within the context of the Gaussian sequence model, we aim to predict   $\wt Y\sim N_n(\theta,r\times I)$ from an independent observation $Y\sim N_n(\theta,I)$ as $n\rightarrow\infty$, where the true underlying parameter $\theta$ is sparse in the sense that
$\theta\in\Theta_n(s_n)$ where $\Theta_n(s_n)=\{\theta\in\R^n:\|\theta\|_0\leq s_n\}$. We study the problem of  obtaining an entire {\em predictive density} $\hat p(\wt Y\C Y)$ for $\wt Y$ that is close to $\pi(\wt Y\C  \theta)$ in terms of the Kullback-Leibler   loss
\begin{equation}
L(\theta,\hat p(\cdot\C Y))=\int \pi( \wt Y\C\theta)\log\frac{\pi( \wt Y\C \theta)}{\hat p( \wt Y\C Y)}\d  \wt Y\label{eq:KL_loss},
\end{equation}
assuming that $r\in (0,\infty)$ is known. A more classical version of this problem (without sparsity) was examined in the foundational paper  \cite{george_liang_xu2}. This paper  studied predictive distributions and assessed  the quality of the density estimator $\hat p(\cdot\C Y)$  by its risk
$$
\rho_n(\theta,\hat p)=\int \pi( Y\C\theta)L(\theta,\hat p(\cdot\C  Y))\d  Y.
$$
For any prior distribution $\pi(\cdot)$, the average (Bayes) risk $r(\pi,\hat p)=\int \rho_n(\theta,\hat p)\pi(\theta)\d\theta$ is known to be minimized by the Bayes (posterior) predictive density
\begin{equation}
\hat p ( \wt Y\C Y)=\int \pi( \wt Y\C \theta)\pi(\theta\C Y)\d \theta.\label{eq:bpd}
\end{equation}
We review some perhaps  known, yet interesting, facts about the subtleties of predictive inference. See \cite{george_liang_xiu} for a complete compendium on knowledge in low-dimensional (non-sparse) situations.
A tempting, but deceiving, strategy is to use a plug-in estimator (e.g. the maximum likelihood estimator $\hat{\theta}_{MLE}$)  to obtain a predictive density estimate $\pi( \wt Y\C\hat{\theta}_{MLE})$. This malpractice was denounced by \citet{aitchison} who  showed that $\hat p_U( \wt Y\C  Y)$ defined as \eqref{eq:bpd} under the uniform prior dominates the plug-in predictive density $\pi( \wt Y\C \hat{ \theta}_{MLE})$. 
 When $p=1$, $\hat p_U( \wt Y\C  Y)$ is admissible  under KL loss \cite{liang_barron} but when $p\geq 3$, $\hat p_U( \wt Y\C  Y)$ is inadmissible and dominated by Bayesian predictive density under the harmonic prior \cite{komaki}.  \citet{george_liang_xu2} established general sufficient conditions under which a Bayes predictive density will be minimax and will dominate $\hat p_U( \wt Y\C  Y)$.  \citet{george_xu} extended some of these results to a regression setup (known variance, fixed dimensionality). 
These developments testify that the Bayesian approach through integration (as opposed to a plug-in approach) is a far more suitable predictive framework. Our work focuses on {\em high-dimensional} scenarios where $n\rightarrow\infty$ and when $\theta$ is sparse, focusing on {\em rate-minimaxity}.

In their pathbreaking paper, \citet{muk1}  were the first to study predictive density estimation for a multivariate normal vector with $\ell_0$ sparse means. These authors quantified the minimax risk under the KL loss which equals the minimax risk of estimating sparse normal means up to a constant which depends on the ratio of variances of future and observed data, i.e. with $s_n/n\rightarrow 0$
$$
\inf_{\hat p}\sup_{\theta\in\Theta_n(s_n)}\rho_n(\theta,\hat p)\sim\frac{1}{1+r}s_n\log(n/s_n)
$$
where $r=r/1$ is the variance ratio (future over observed data) and where the minimum is taken over all predictive density estimators. We will be targeting this minimax rate using popular priors.

\subsection{The Calibration Conflict of Bayesian LASSO}\label{sec:BLASSO}
The LASSO method \cite{tibshirani_lasso} is a staple in sparse signal recovery. 
The LASSO estimator is, in fact, Bayesian \cite{casella} as it corresponds to the posterior mode under the Laplace prior 
  \begin{equation}\label{eq:lasso_prior}
\pi(\theta\C\lambda) =\prod_{i=1}^n\pi_1(\theta_i\C\lambda)\quad\text{where}\quad \pi_1(\theta\C\lambda) =\frac{\lambda}{2}\e^{-\lambda|\theta|}\quad\text{for some}\quad \lambda>0.
\end{equation}
In our sparse normal-means setting, the LASSO estimator is known to attain the (near) minimax
rate $s_n\log n$  for the square Euclidean loss if the regularity parameter $\lambda$ is chosen of the order $\sqrt{2\log n}$.  Since the LASSO estimator is a posterior mode under the Laplace prior, it is tempting to utilize the entire posterior distribution as an inferential object.  This inclination was soon corrected  by \citet{castillo_vdv2} who showed that the entire posterior distribution (known as the Bayesian LASSO posterior) for such $\lambda$ puts no mass on balls centered around the sparse truth whose radius is of much larger order than the minimax rate. The discrepancy between the performance of a posterior mode and the entire posterior distribution is a revealing cautionary tale.  Since the posterior predictive distribution is a functional of the posterior distribution, we should be skeptical about Bayesian LASSO in the context of prediction inference as well.

For the  Bayesian LASSO  independent product prior \eqref{eq:lasso_prior} \cite{hans_biometrika,casella}, the Bayesian predictive density has a product form
$
\hat p( \wt Y\C  Y)=\prod_{i=1}^n\hat p( \wt Y_i\C  Y_i)
$
and
$$
L( \theta,\hat p(\cdot\C  Y))=\sum_{i=1}^n \int \pi( \wt Y_i\C\theta_{i})\log\frac{\pi( \wt Y_i\C \theta_{i})}{\hat p( \wt Y_i\C  Y_i)}\d y= 
\sum_{i=1}^n L(\theta_{i},\hat p(\cdot\C  Y_i)).
$$
The predictive risk of a product rule over the $\Theta_n(s_n)$ is additive and satisfies \cite{muk1}
\begin{equation}
(n-s_n)\rho(0,\hat p)< \rho_n(\theta,\hat p)=\sum_{i=1}^n\rho(\theta_{i},\hat p)\leq (n-s_n)\rho(0,\hat p)+s_n\sup_{\theta\in \R}\rho(\theta,\hat p),\label{eq:basic_inequality}
\end{equation}
where (for any $1\leq i\leq n$) and $\theta\in\R$
$$
\rho(\theta,\hat p)=\int \pi(Y_i\C\theta)\int \pi(\wt Y_i\C\theta )\log [\pi(\wt Y_i\C\theta)/\hat p(\wt Y_i\C Y)]\d \wt Y_i\d Y_i
$$
is the univariate risk. The following Lemma precisely characterizes the univariate risk $\rho(\theta,\hat p)$ for $\theta\in\R$ under the prior \eqref{eq:lasso_prior}. This Lemma, in fact, applies to {\em any} 
prior $\pi_1(\cdot)$.
\begin{lemma}\label{lemma:lasso_risk_decompose}
The univariate prediction risk under the Bayesian LASSO prior \eqref{eq:lasso_prior} satisfies
\begin{equation}
\rho(\theta,\hat p)=\theta^2/(2r)+ E \log N_{\theta,1}^{LASSO}(Z)-E \log N_{\theta,v}^{LASSO}(Z)\label{eq:risk_lasso}
\end{equation}
where $v=1/(1+1/r)$ and
$
N_{\theta,v}^{LASSO}(Z)=\int \exp \left\{\mu\left[Z/\sqrt{v}+\theta/v\right]-\frac{\mu^2}{2v}\right\}\pi_1(\mu\C\lambda)\d\mu
$
and where the expectation is taken over $Z\sim N(0,1)$.
\end{lemma}
\proof Section \ref{sec:lemma:lasso_risk_decompose}.

\smallskip
The inability of the entire posterior to concentrate around the truth at an optimal rate \cite{castillo_vdv2} stems from a tuning conflict. For noise coordinates $\theta_i=0$,   $\lambda$ needs to be large in order to push the coefficient to zero and for signals $\theta_i\neq 0$,  $\lambda$ needs to be small to avoid squashing large effects. We expect a similar conundrum for the prediction problem. The inequality \eqref{eq:basic_inequality} shows that,
in order for  Bayesian LASSO prediction distributions to be rate-minimax optimal, we would need 
\begin{equation}
\sup_{\theta\in\R}\rho(\theta,\hat p)\lesssim \frac{\log(n/s_n)}{1+r}\quad\text{and, at the same time,}\quad \rho(0,\hat p)\lesssim 
 \frac{s_n\log(n/s_n)}{(n-s_n)(1+r)}.\label{eq:risks}
\end{equation}
It will follow from  Theorem \ref{lemma:risk_lasso_upper} and Theorem \ref{corollary:blasso} that these two goals are simultaneously unattainable with the same $\lambda$.
\begin{theorem}\label{lemma:risk_lasso_upper}
For $v=1/(1+1/r)$ and the Laplace prior  $\pi_1(\theta\C\lambda)$  with $\lambda>0$ we obtain 
$$
\rho(0,\hat p)\leq    \log \left(1+\frac{\sqrt{2}}{\lambda\sqrt{\pi v}}\right)+ \frac{4}{\lambda^2 v}\quad\text{and}\quad 
\sup_{\theta\neq 0}\rho(\theta,\hat p)\leq \log\left(\sqrt{\frac{32\lambda^2\pi}{v}} \right)+\frac{\lambda^2}{2}+\lambda\sqrt{\frac{2}{\pi}}+\frac{4}{\lambda^2}.
$$
\end{theorem}
\proof Section \ref{sec:proof_lemma:risk_lasso_upper}

\smallskip

Theorem \ref{lemma:risk_lasso_upper} and the inequality \eqref{eq:risks} suggest the following calibrations.
For the signal-less scenarios with $\lambda\rightarrow\infty$, the dominant term in the risk is $1/(\lambda\sqrt{v})$ which means that $\lambda\sqrt{v}$ should increase to infinity at least as fast as $\frac{(n-s_n)(1+r)}{s_n\log(n/s_n)}$. For the signal scenarios, the calibration may need to depend on $r$. The dominant term is $\lambda^2$ which suggests that   (a) $\lambda$ should not increase faster than $[(1-v)\log (n/s_n)]^{1/2}$ (when $r>1$) and  $[v\log (n/s_n)]^{1/2}$ when $0<r<1$ and (b) $\lambda$ should not decay slower than $[(1-v)\log (n/s_n)]^{-1/2}$. The calibration upper bound mirrors the oracle  threshold for signal recovery (up to the multiplication factor which depends on $r$). Unfortunately, the two calibration goals for signal and noise are not attainable simultaneously. 
We can ascknowledge the calibration dilemma from a plot of the Bayesian LASSO univariate prediction risk $\rho(\theta,\hat p)$ for various choices of $\lambda$ when $r=2$ (Figure \ref{fig:risk} on the left). In order for the risk at zero to be small, we need large $\lambda$ which will unfortunately inflate the risk for larger $|\theta|$. On the other hand, to verify that small $\lambda$ will inflate the risk at zero, we have the following lower bound result implying that the Bayesian LASSO prediction risk is suboptimal for the calibration $\lambda\propto\sqrt{\log (n/s_n)}$ which mirrors optimal tuning for  estimation.

\begin{figure}
\scalebox{0.45}{\includegraphics{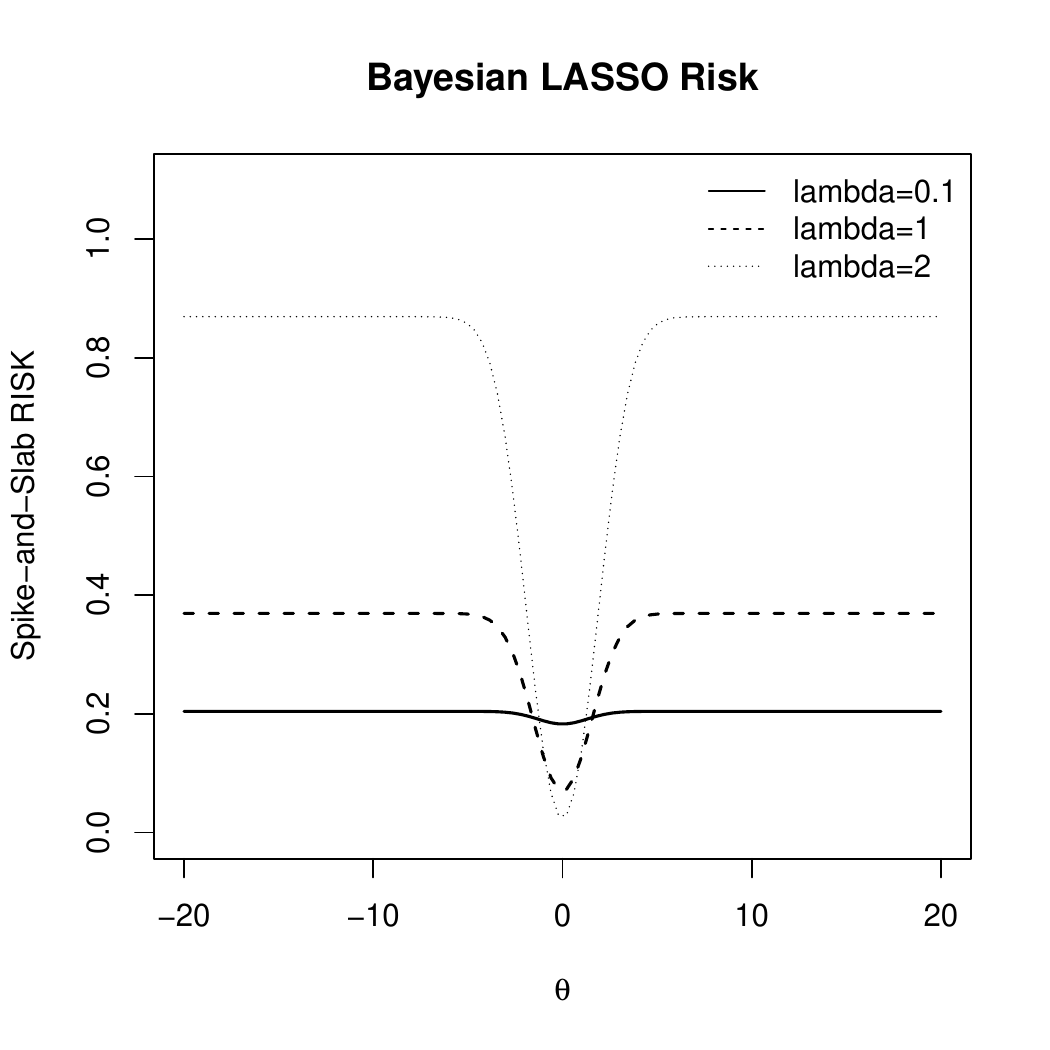}}
\scalebox{0.45}{\includegraphics{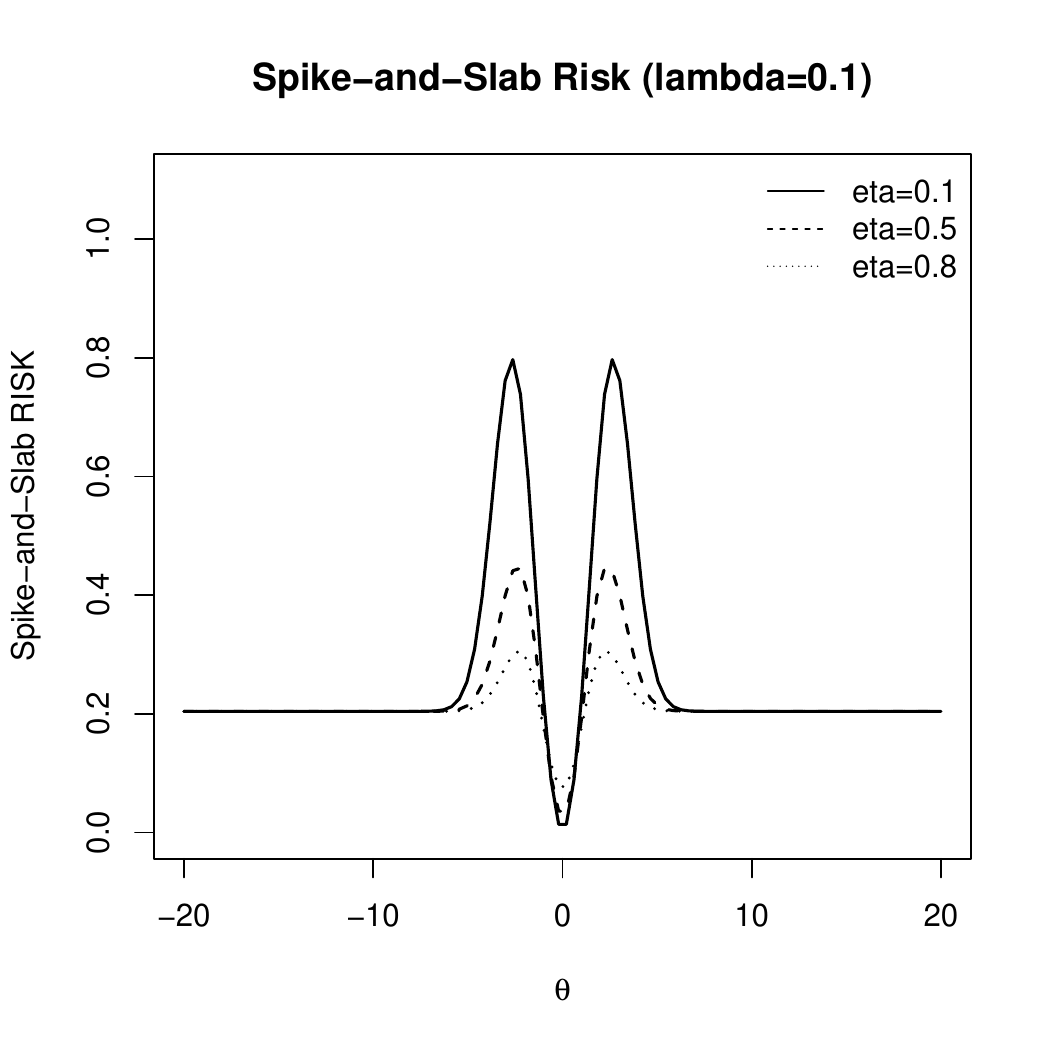}}
\caption{\small (Left) Bayesian LASSO prediction risk $\rho(\theta,\hat p)$ for $\lambda\in\{0.1,1,2\}$; (Right) Spike-and-Slab risk (Dirac spike and Laplace slab  with a penalty $\lambda=0.1$ with $\eta\in\{0.1,0.5,0.8\}$); Both plots correspond to $r=2.$}\label{fig:risk}
\end{figure}

\begin{theorem} (Bayesian LASSO is Suboptimal)\label{corollary:blasso}
Consider the Bayesian LASSO prior in \eqref{eq:lasso_prior} with $\lambda=\lambda_n\rightarrow \infty$ as $n\rightarrow\infty$. For fixed $v=1/(1+1/r)$  and a suitable $a>0$ such  that $\left(\frac{1}{\sqrt{v}}-1\right)\frac{a}{\lambda_n+a}<1$ we have
$$
\inf\limits_{\theta\in\Theta(s_n)}\rho_n(\theta,\hat p)> (n-s_n)\left[\left(\frac{1}{\sqrt{v}}-1\right)\frac{a\left(1-\Phi(a)\right)}{2(\lambda_n+a)}-\frac{1}{\lambda_n^2}\left(4+\frac{3}{v}+\frac{2}{\lambda_n\sqrt{v}}\right)\right].
$$
\end{theorem}
\proof It follows from Lemma \ref{lemma:lb_lasso} after noting that $\e^{-x^2/2}\log(x)\leq 1/x^2$ and that $\log(1+x)>x/2$ for $x\in (0,1)$.
\smallskip

The important takeaway message of Theorem \ref{corollary:blasso} is that the lower bound on the Bayesian LASSO prediction risk increases to infinity at a {\em suboptimal rate}  for calibrations $\lambda_n$ that increase to infinity at a slower pace  than $\frac{n-s_n}{s_n}\frac{1}{\log (n/s_n)}$.

\subsection{Separable Spike-and-Slab Priors}\label{sec:separable}
Section \ref{sec:BLASSO} unveils the sad truth about the  single Laplace prior that it cannot provide rate-optimal posterior predictive distributions in sparse situations.  
Precisely for the dual purpose of estimation and selection, the Spike-and-Slab  prior framework \cite{mitchell_beauchamp}  adaptively 
borrows from  two prior components depending on the magnitude of observed data.
Spike-and-Slab priors are known to have ideal adaptation properties in sparse signal recovery  such as rate-minimax adaptation in the supremum-norm sense \cite{hoffmann} (without any log factors) or  rate-minimax spatial adaptation \cite{rockova_rousseau}. This section highlights  benefits of prior mixing for prediction and demonstrates why the spike distribution is an indispensable enhancement of the simple Laplace prior from Section \ref{sec:BLASSO}.

We consider a separable (independent product) Spike-and-Slab prior for $\theta\in\R^n$ 
\begin{equation}
\pi(\theta\C\lambda,\eta)=\prod_{i=1}^n\pi(\theta_i\C\lambda,\eta),\,\,\,\text{where}\,\,\,
\pi(\theta_i\C\lambda,\eta)=\eta\pi_1(\theta_i)+(1-\eta)\pi_0(\theta_i)\,\,\,\text{and} \,\,\,  \eta\in (0,1).\label{eq:prior_spike_slab}
\end{equation}
The spike density $\pi_0(\theta_i)$ serves as a model for the noise coefficients $\theta_i=0$ while the slab density $\pi_1(\theta_i)$ serves as a  model  for the signals $\theta_i\neq 0$. The prior mixing weight $\eta\in(0,1)$ determines the amount of 
exchange between the two priors. We have, for now, silenced the dependence of $\pi_1$ on  $\lambda$ which will appear as a tuning parameter for the slab distribution in the next section.
In order to appreciate the benefits of prior mixing for the purpose of prediction, we first point out that the predictive distribution is  also a mixture. This simple fact is appreciated even more  after noting that this mixture is conveniently mixed by a weight that depends on the Bayes factor between the slab/spike models. Notice that because \eqref{eq:prior_spike_slab} is an independent product prior, the predictive rule is again a product rule. Therefore we focus on the {\em univariate} predictive distributions below.

\begin{lemma}\label{lemma:risk_mixture}
Denote by $m_j(Y_i)=\int \pi(Y_i\C\mu)\pi_j(\mu)\d\mu$ for $j=0,1$  the marginal likelihoods for $Y_i$ under the spike/slab priors $\pi_0$ and $\pi_1$. For $\eta\in (0,1)$, we define a mixing weight
\begin{equation}
\Delta_\eta(Y_i)=\frac{\eta m_1(Y_i)}{\eta m_1(Y_i)+(1-\eta) m_0(Y_i)}.\label{eq:thetax}
\end{equation}
Then the Bayesian predictive density under the prior \eqref{eq:prior_spike_slab} is a mixture, i.e. 
\begin{equation}
\hat p(\wt Y_i\C Y_i)=
\Delta_\eta(Y_i)\hat p_1(\wt Y_i\C Y_i) + [1-\Delta_\eta(Y_i)]\hat p_0(\wt Y_i\C Y_i)\label{eq:elegant_decompose}
\end{equation}
where $\hat p_j(\wt Y_i\C Y_i)=\frac{\int\pi(\wt Y_i\C\mu)\pi(Y_i\C \mu)\pi_j(\mu)\d\mu}{m_j(x)}$ for $j=0,1$ are posterior predictive densities under the  spike/slab priors (respectively).
\end{lemma}
\proof 
This follows readily from the definition 
$\hat p(\wt Y_i\C Y_i)=\frac{\int\pi(\wt Y_i\C\mu)\pi(\mu\C Y_i)\d\mu}{\int \pi(\mu\C Y_i)\d\mu}$.

\smallskip

This elegant decomposition provides useful insights into the workings of the mixture prior. The predictive density will be dominated by the slab predictive density when $\Delta_\eta(Y_i)$ is close to one, i.e. the observation $Y_i$ is more likely to have arisen from the marginal distribution $m_1$ as opposed to $m_0$. This is very intuitive. The opposite is true when the data $Y_i$ supports the spike model, in which case the predictive density is taken over by the spike predictive density. The amount of support for the spike/slab predictive densities is determined by the ratio of marginal likelihoods evaluated at the observed data $Y_i$.
In other words, we can rewrite the mixing weight  as a functional of the Bayes factor
$$
\Delta_\eta(Y_i)=\left[ 1+\frac{1-\eta}{\eta} BF(Y_i;0,1)\right]^{-1},
$$ 
where $BF(Y_i;0,1)=\frac{m_0(Y_i)}{m_1(Y_i)}$ is the Bayes factor for the spike model versus the slab model. 
Being able to switch between two regimes has been exploited for the purpose of constructing  predictive density estimators by \cite{muk1}. These authors proposed a class of univariate predictive density estimators that are analogs of hard thresholding in that they glue together two densities depending on the magnitude (signal detectability) of observed data $|Y_i|$. These  estimators, however, do  not have a fully Bayesian motivation and are not smooth  with respect to  $Y_i$. The discussion above highlights that regime switching is achieved {\em automatically and smoothly} within the Spike-and-Slab framework.   The predictive density decomposition  in Lemma \ref{lemma:risk_mixture} invites the possibility of upper-bounding the risk in terms of separate spike/slab risks.

\begin{lemma}\label{lemma:simple_bound_risk}
Denoting the average mixing weight
$
\Lambda(\theta_i)=\int\Delta_\eta(Y_i) \pi(Y_i\C\theta_i)\d x
$
and using the notation in Lemma \ref{lemma:risk_mixture}, the prediction risk under the prior \eqref{eq:prior_spike_slab} satisfies
\begin{align}
\rho(\theta_i,\hat p)
&<\Lambda(\theta_i) \rho(\theta_i,\hat p_1)+(1-\Lambda(\theta_i))\rho(\theta_i,\hat p_0).\label{eq:simple_upper_bound}
\end{align}

\end{lemma}
\proof It follows from a simple application of Jensen's inequality $E \log  X<\log E X$ which yields (with the expectation taken over a binary random variable $\gamma\in \{0,1\}$ with $P(\gamma=1)=\Delta_\eta(Y_i)$)
\begin{align*}
\rho(\theta_i,\hat p)
&<\int\pi( Y_i\C\theta_i)\left\{\int\pi( \wt Y_i\C\theta_i)E \left[\log \frac{\pi( \wt Y_i\C\theta_i)}{\gamma \hat p_1( \wt Y_i\C  Y_i)+(1-\gamma)\hat p_0( \wt Y_i\C  Y_i)}  \right]\d  \wt Y_i\right\} \d  Y_i\\
&=\Lambda(\theta_i) \rho(\theta_i,\hat p_1)+(1-\Lambda(\theta_i))\rho(\theta_i,\hat p_0).
\end{align*}

This elegant upper bound shows how the risk is dominated either by the spike predictive density (for parameter values $\theta_i$ such that $\Lambda(\theta_i)$ is small) or the slab predictive density  (for parameter values $\theta_i$ such that $\Lambda(\theta_i)$ is large).  This bound is perhaps more intuitive than useful, however. Our analysis rests on a more precise characterization of the risk. In Lemma \ref{lemma:decompose} below, we obtain a risk decomposition for $\rho(\theta_i,\hat p)$ but for  {\em general} mixtures  \eqref{eq:prior_spike_slab}.

\begin{lemma}\label{lemma:decompose}
The univariate risk for the Spike-and-Slab prior \eqref{eq:prior_spike_slab}   satisfies for $\theta\in\R$
\begin{equation}
\rho(\theta,\hat p)= \rho(\theta,\hat p_1)+ E \log    N_{\theta,1}^{SS}(Z)
- E \log   N_{\theta,v}^{SS}(Z),\label{eq:SS_slab_risk}
\end{equation}
where
$$
  N_{\theta,v}^{SS}(z)=\left[1+\frac{1-\eta}{\eta}\frac{\int \exp\left(\mu(\frac{\theta}{v}+\frac{z}{\sqrt{v}})-\frac{\mu^2}{2v}\right)\pi_0(\mu)\d\mu}{\int\exp\left(\mu(\frac{\theta}{v}+\frac{z}{\sqrt{v}})-\frac{\mu^2}{2v}\right)\pi_1(\mu)\d\mu}\right].
$$
and $v=1/(1+1/r)$ and where the expectation is taken over $Z\sim N(0,1)$.
\end{lemma}
\proof Section \ref{sec:proof_lemma:decompose}

\begin{remark}\label{rem:decompose}
Alternatively, we could write
$
\rho(\theta,\hat p)= \rho(\theta,\hat p_0)+ E \log  N_{\theta,1}(z)
- E \log   N_{\theta,v}(z)
$
for 
$N_{\theta,v}(z)=\left[1+\frac{\eta}{1-\eta}\frac{\int \exp\left(\mu(\frac{\theta}{v}+\frac{z}{\sqrt{v}})-\frac{\mu^2}{2v}\right)\pi_1(\mu)\d\mu}{\int\exp\left(\mu(\frac{\theta}{v}+\frac{z}{\sqrt{v}})-\frac{\mu^2}{2v}\right)\pi_0(\mu)\d\mu}\right].$
For the point-mass spike $\pi_0=\delta_0$ we would then obtain the same expression as in Theorem 2.1 of \cite{muk1}.
\end{remark}

Lemma \ref{lemma:decompose} is a simple generalization of  Theorem 2.1 in \cite{muk1}, who showed risk decomposition for {\em point-mass} spike-and-slab priors. 
We present it here because it is useful for other Spike-and-Slab priors such as the Spike-and-Slab LASSO  \cite{rockova} studied in Section \ref{sec:ssl}.
In comparison with Lemma \ref{lemma:lasso_risk_decompose}, we have a different expression $N_{\theta,v}^{SS}(z)$ which depends on the prior odds $(1-\eta)/\eta$ whose calibration will be crucial.

\subsubsection{Dirac Spike and Laplace Slab}\label{sec:dirac_laplace}
In Section \ref{sec:BLASSO}, we convinced ourselves that a single Laplace prior will not be able to yield rate-optimal predictive distributions.
We now show that  the Laplace slab $\pi_1(\theta\C\lambda)$ {\em in concert with} a Dirac spike $\pi_0(\theta)$ within \eqref{eq:prior_spike_slab} does.

\begin{theorem}\label{thm:dirac_laplace}
Assume the Spike-and-Slab prior  \eqref{eq:prior_spike_slab} with a Dirac spike $\pi_0(\theta_i)=\delta_0(\theta_i)$ and a Laplace slab $\pi_1(\theta_i\C\lambda)$ in \eqref{eq:lasso_prior}. Denote $v=1/(1+1/r)$ and set $(1-\eta)/\eta= n/s_n$. 
Choose $\lambda^2=vC_r$ for $C_r>2/[v(1/2+4)]$ when $0<r<1$ and  $\lambda^2=(1-v)C^*_r$  for $C^*_r>2/[5(1-v)]$ when $r\geq1$  
then with $s_n/n\rightarrow 0$ we have for any fixed $ r\in (0,\infty)$
\begin{equation}
\sup_{\theta\in \Theta(s_n)}\rho_n(\theta,\hat p)\leq \frac{5}{1+r}s_n\log (n/s_n)+\wt C(r) \label{eq:rate_optimal}
\end{equation}
where $\wt C(r)$ is either \eqref{eq:remainder1} or \eqref{eq:remainder2}.
\end{theorem}
\proof Section \ref{sec:proof_thm:dirac_laplace}.

\smallskip
Theorem \ref{thm:dirac_laplace} shows that the minimax-rate predictive performance is attainable by Spike-and-Slab priors with a simple Laplace slab distribution. This prior is widely used in practice \cite{ray_szabo}. 
Of course, our result in Theorem \ref{thm:dirac_laplace} is only rate-minimax, where the multiplication constant may not optimal.
Up until now, minimax (rate-optimality) property was established for discrete (grid) slab priors which may be simpler to treat analytically  \cite{muk2}.
The only other continuous slab result in the literature so far is in \cite{muk2} who analyzed {\em uniform} slabs over a finite parameter domain depending on $s_n$ (which are perhaps not as widely used).   
\begin{remark}(Calibration)
What is particularly noteworthy in Theorem \ref{thm:dirac_laplace} is the prior calibration. The oracle tuning for Spike-and-Slab priors in the context of estimation would be $\eta/(1-\eta)=s_n/n$. 
This is  the oracle tuning for $\eta$ in the prediction problem as well. However, in terms of calibrating $\lambda$, we distinguish between two regimes. When $r>1$,  the future observation is noisier than the observed data which makes the prediction problem simpler. In this case, we can choose $\lambda$ to be a small fixed constant which does not need to depend on $n$.  This corresponds to the usual Spike-and-Slab calibration and small $\lambda$ regime discussed earlier in \cite{rockova}. 
Here we tried to optimize $\lambda$  as a function of $r$. For the more difficult case when $0<r<1$,  i.e.
when the observed data is noisier, we can also choose  $\lambda$ proportional to the usual oracle detection threshold $\sqrt{2\log (n/s_n)}$ suitably rescaled by  a constant multiple of $\sqrt{v}$.
\end{remark}
\begin{remark}
Inside the proof of Theorem \ref{thm:dirac_laplace} we distinguish between two regimes: (1) when $|Y_i|>\lambda+\sqrt{2v\log (n/s_n)}$, with i.e. the observed data is above the usual detection threshold rescaled by $v$ and shifted by $\lambda$, and (2) when $|Y_i|\leq \lambda+\sqrt{2v\log (n/s_n)}$. In the first case,   the slab predictive density takes over and the reverse is true for case (2).  \cite{muk1} used a related regime switching idea to construct an estimator by pasting two predictive densities depending on the size of $|Y_i|$. We leverage these two regimes only inside a proof, not for a construction of the estimator.
\end{remark}
  It is interesting to compare the risk performance of the single Laplace prior and the Spike-and-Slab prior with a Laplace slab in Figure \ref{fig:risk}. The right plot corresponds to the calibration $\lambda=0.1$, which benefits values $\theta_i\neq 0$ . In order to diminish risk at zero, we need to decrease $\eta$. Compared with the left plot, Bayesian LASSO risk asymptotes towards the same value as $|\theta_i|\rightarrow\infty$, but has an elevated risk at zero. This confirms our intuition that the proper calibration for the Spike-and-Slab prior is a small (fixed constant) $\lambda$ and a small $\eta$ which depends on $s_n/n$. The difference between the risks for the Bayesian LASSO (left plot) with $\lambda=0.1$ and Spike-and-Slab with a Laplace slab  (right plot) with $\lambda=0.1$ is striking. The mixture prior can suppress risk at zero by decreasing $\eta$, while keeping the risk small for larger $|\theta_i|$. Notably, there are certain values of $\theta_i$ for which the risk is inflated due to uncertainty whether or not the underlying parameter $\theta_i$ arrived more likely from the slab/spike.

\subsubsection{The Spike-and-Slab LASSO}\label{sec:ssl}
The Spike-and-Slab LASSO prior \cite{rockova} has become   popular  for sparsity recovery  due to its self-adaptive shrinkage properties. It generalizes the Laplace prior \eqref{eq:lasso_prior} by deplying a two-point Laplace mixture \eqref{eq:prior_spike_slab} where
\begin{equation}
\pi_0(\theta_i)=\lambda_0/2\e^{-\lambda_0|\theta_i|}\quad\text{and}\quad  \pi_1(\theta_i)=\lambda_1/2\e^{-\lambda_1|\theta_i|} \quad\text{with $\lambda_1<\lambda_0$}.\label{eq:laplaces}
\end{equation}
This prior has been successfully implemented in high-dimensional regression \cite{RG14b}, graphical models \cite{deshpande}, factor analysis \cite{RG15}, biclustering \cite{moran_biclus}, among others  \cite{bai}.  The perception had been the  Spike-and-Slab theory holds only for point-mass spikes $\pi_0(\cdot)=\delta_0(\cdot)$. \citet{rockova}, however, showed that asymptotic minimaxity is attainable also for these {\em continuous} spike distributions in the context of estimating sparse mean under the Euclidean loss. Here, we examine the ability of Spike-and-Slab LASSO priors to yield optimal predictive distributions. While the majority of implementations of the Spike-and-Slab LASSO focus on posterior mode detection, posterior simulation is available through traditional Gibbs sampling  or Bayesian bootstrap techniques \cite{nie}. The ability to simulate from the posterior  
makes the posterior predictive distribution readily available. Here, we show that it is in fact rate-optimal when properly calibrated.

\begin{theorem}\label{thm:ssl}
Assume the Spike-and-Slab LASSO prior  \eqref{eq:prior_spike_slab} with \eqref{eq:laplaces} and $(1-\eta)/\eta=c$ for some fixed constant $c>0$. Choose $\lambda_0 =n/s_n$ and $\lambda_1$ as in Theorem \ref{thm:dirac_laplace}. 
When $s_n/n\rightarrow 0$, then we have 
\begin{equation}
\sup_{\theta\in \Theta(s_n)}\rho(\theta,\hat p)\sim \frac{s_n}{1+r} \log (n/s_n)\label{eq:rate_optimal}
\end{equation}
for any fixed $r\in(0,1)$. The same conclusion holds for $r\in  [1,\infty)$ for parameters $\theta\in\Theta_n(s_n)\cap\{\theta\in\R^n:\min_{1\leq i\leq n}|\theta_i|>c\sqrt{\log (n/s_n)}\} $ for $c>2\sqrt{v}$.
\end{theorem}
\proof Section \ref{sec:proof_thm:ssl}.

\begin{figure}
\scalebox{0.3}{\includegraphics{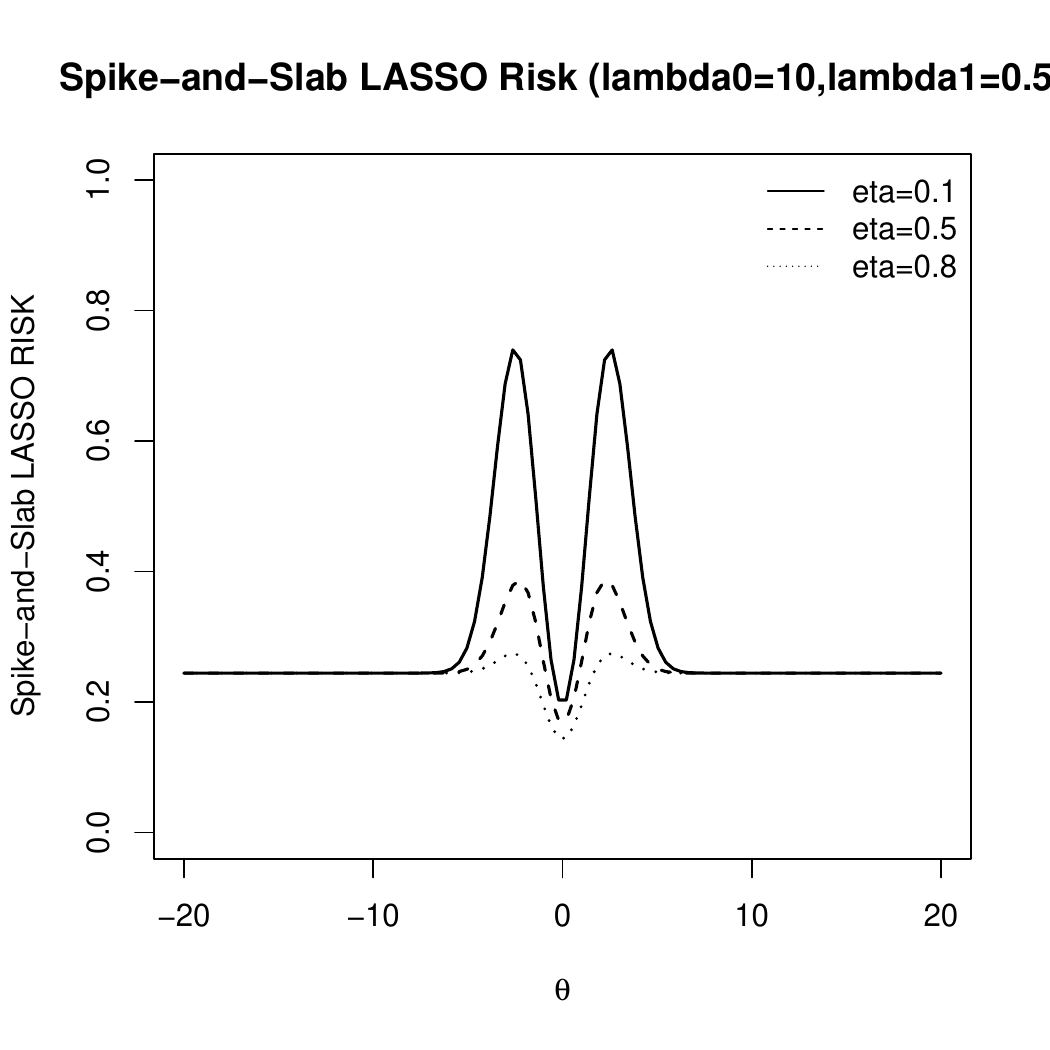}}
\scalebox{0.3}{\includegraphics{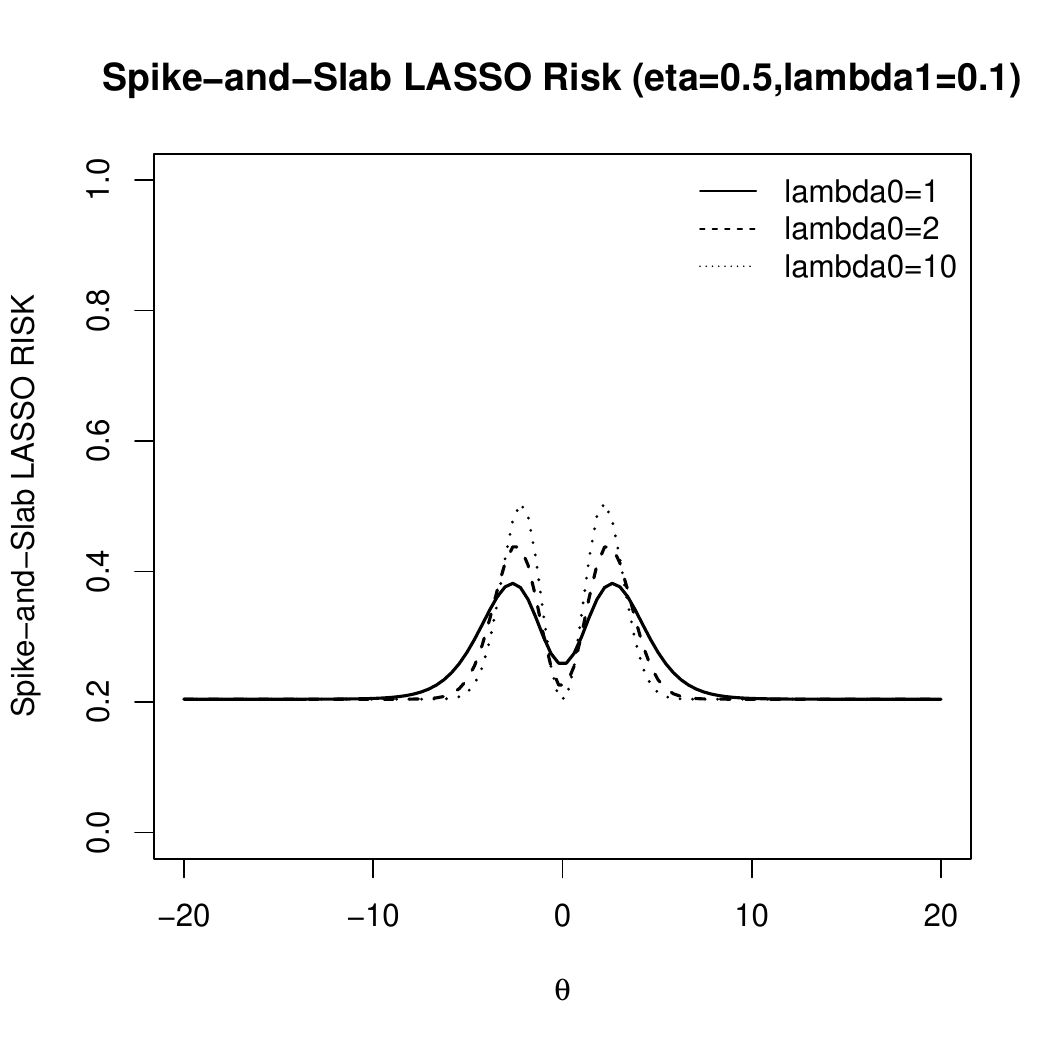}}
\scalebox{0.3}{\includegraphics{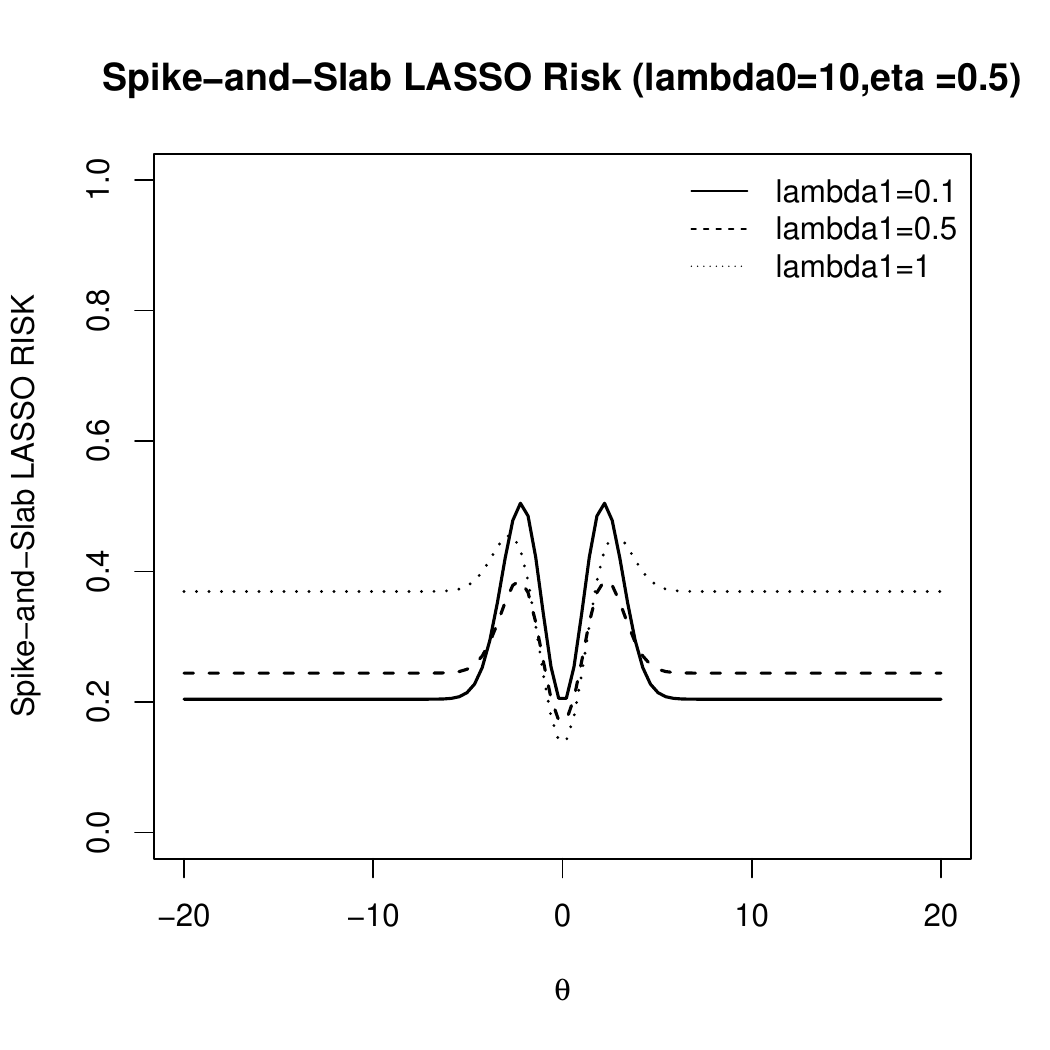}}
\caption{\small Spike-and-Slab LASSO prediction risk for various calibrations. (Left) Varying $\eta$ for fixed $\lambda_0=10,\lambda_1=0.5$; (Middle) Varying $\lambda_0$ for fixed $\eta=0.1,\lambda_1=0.1$; (Right) Varying $\lambda_1$ for fixed $\eta=0.1,\lambda_0=10$}\label{fig:ssl}
\end{figure}

\begin{remark}(Calibration)
\citet{rockova} concluded that $\lambda_0/\lambda_1\times (1-\eta)/\eta$ should behave as $(n/s_n)^c$ for some $c\geq 1$ to yield rate-minimaxity in estimation.  
Figure \ref{fig:ssl} (middle plot) shows that the risk at zero will be small when $\lambda_0$ is large, not necessarily only when $\eta$ is small (left plot). This observation is confirmed by the upper bound on $\rho(0,\hat p)$ in \eqref{eq:upper_bound_SSL_zerp}. Interestingly, we {\em do not} require $(1-\eta)/\eta$ to be of the order $n/s_n$ as long as  $\lambda_0$ is increasing at a rate $n/s_n$. If we assumed $(1-\eta)/\eta=n/s_n$, the proof in  Section \ref{sec:proof_thm:ssl} shows that  the calibration $\lambda_0=n/s_n$ would yield the optimal rate only for $r\in (0,1)$. For general $r\in(0,\infty)$ we need to fix $(1-\eta)/\eta$ and let $\lambda_0$ be of the order $n/s_n$. In terms of $\lambda_1$, we have the same assumption as in Theorem \ref{thm:dirac_laplace}, where $\lambda_1$ is set proportional to either $\sqrt{v}$ or $\sqrt{1-v}$. The incentive is to keep it small so that the risk for large effects is small, but not too small so that the supremum over $\theta\neq 0$ is not too large. For $r>1$, we require a signal-strength assumption that is similar to the one in \cite{rockova}.
\end{remark}

\subsection{Hierarchical Spike-and-Slab Priors}\label{sec:hierarchical}
 All results in \cite{muk1,muk2} are {\em non-adaptive}, i.e. the knowledge of $s_n$ is required to construct or tune the estimator in order to obtain minimax performance.  While we consider only rate-minimax performance, there are reasons to be hopeful that hierarchical priors (which do achieve automatic adaptation in parameter estimation) will yield an adaptive rate for predictive density estimation. One recent remarkable example is supremum-norm adaptation (without any log factors) of hierarchical spike-and-slab priors  in non-parametric wavelet regression  \cite{hoffmann}. We inquire into adaptability of hierarchical priors   in the context of predictive inference.

We continue our investigation of the Dirac spike and Laplace slab prior from Section \ref{sec:dirac_laplace}.  However, now we assume a {\em hierarchical} version  (not an independent product), i.e.
\begin{equation}\label{eq:hierarchical}
\pi(\theta)=\int_{\eta}\prod_{i=1}^n[(1-\eta)\delta_0+\eta \pi_1(\theta_i)]\pi(\eta)\d\eta\quad\text{and}\quad \pi(\eta)\sim Beta(a,b)
\end{equation}
for some $a,b>0$.
In Section \ref{sec:separable}, we saw how independent product Spike-and-Slab mixtures yield {\em within-coordinate} mixing, adaptively borrowing from both the spike and the slab.
The hierarchical prior \eqref{eq:hierarchical} performs an additional {\em across-coordinate} mixing, borrowing strength and transmitting sparsity information. This point was highlighted in the  estimation context   under the squared error loss \cite{RG14b,rockova}.  The posterior predictive distribution under the hierarchical prior is no longer an independent product but a scale mixture of two-point mixtures. Indeed, using Lemma \ref{lemma:risk_mixture}, we have
\begin{align}
\hat p( \wt Y\C  Y)
&=\int_\eta\prod_{i=1}^n\left[\Delta_\eta( Y_i)\hat p_1( \wt Y_i\C  Y_i)+(1-\Delta_\eta( Y_i))\hat p_0( \wt Y_i\C  Y_i)\right]\d\pi(\eta),\label{eq:posterior_predictive_mixture}
\end{align}
where $\hat p_1$ and $\hat p_0$ are the slab and spike univariate posterior predictive densities, 
$\Delta_\eta( Y_i)$
is the mixing weight \eqref{eq:thetax} and where $m_1$ and $m_0$ are the marginal likelihoods defined in Lemma \ref{lemma:risk_mixture}.
It is also useful to rewrite the predictive distribution as an average predictive density $\hat p( \wt Y\C Y,\eta)$ where the average is taken over the posterior $\pi(\eta\C  Y)$, i.e.
\begin{equation}
\hat p( \wt Y\C  Y)= E_{\eta\C Y}\hat p( \wt Y\C Y,\eta)\label{eq:mixture_rep}.
\end{equation}
This key property is utilized in the following lemma which allows us to bound the KL loss.   
 
\begin{lemma}\label{lemma:KL_average}
The Kullback-Leibler loss of the predictive distribution under the hierarchical prior \eqref{eq:hierarchical} satisfies
$
L(\theta,\hat p(\cdot\C  Y))\leq E_{\eta\C  Y} L(\theta,\hat p(\cdot\C  Y,\eta)).
$
\end{lemma}
\proof
This follows from rewriting  the posterior predictive distribution \eqref{eq:posterior_predictive_mixture}  as \eqref{eq:mixture_rep} and by applying  Jensen's inequality $E\log X<\log EX$ to obtain
\begin{align}
L(\theta,\hat p(\cdot\C  Y))
&=\int \pi( \wt Y\C\theta)\log \pi( \wt Y\C \theta) - \int \pi( \wt Y\C\theta)\log E_{\eta\C  Y} \hat p( \wt Y\C  Y,\eta) \d Y\nonumber\\
&\leq E_{\eta\C  Y} L(\theta,\hat p(\cdot\C  Y,\eta)).\label{eq:bound_KL_sep}
\end{align}

Lemma \ref{lemma:KL_average} shows that it is the conditional distribution $\pi(\eta\C Y)$ which drives the prediction performance. This posterior is expected to be concentrated around zero for sparse situations such as ours. In fact, we would expect that the posterior   $\pi(\eta\C Y)$ will carry important information regarding the sparsity level $s_n$.  To fortify this intuition, the following Lemma shows that the risk of the hierarchical prior is determined by the typical posterior mean of the log-odds $(1-\eta)/\eta$ of a spike versus a slab and its reciprocal.
\begin{lemma}\label{lemma:risk_hierarchical}
The prediction risk under the   hierarchical prior \eqref{eq:hierarchical} satisfies  for $\lambda>2$
 \begin{align*}
\rho(\theta,\hat p)\leq& s_n\left\{C( \lambda,v)+ (1-v)\left[E_{ Y\C\theta} E \log\left(\frac{1-\eta}{\eta}\C Y\right)\right]\right\}\\
&+D(n-s_n)\sup_{i:\theta_i\neq 0} E_{ Y_{\backslash i}\C\theta}  E\left(\frac{\eta}{1-\eta}\C Y_{\backslash i}\right). 
\end{align*}
for a suitable constant $C( \lambda,v)>0$ defined in \eqref{eq:Clambda} and $D=1+2/(a-1)$, where $ Y_{\backslash i}$ denotes the vector $ Y$ without the $i^{th}$ coordinate.
\end{lemma}
\proof Section \ref{sec:proof_lemma:risk_hierarchical}

\smallskip
This Lemma shows how hierarchical priors achieve improved rates compared to the default tuning $(1-\eta)/\eta=n$ for when $s_n$ is {\em not known}.
Lemma \ref{lemma:posterior_odds} below characterizes the behavior of the  posterior mean of prior model (spike/slab) odds.

\begin{lemma}\label{lemma:posterior_odds}
Assume the hierarchical Spike-and-Slab prior $\pi(\theta)$ in \eqref{eq:hierarchical} with $a,b>0$. Under the Gaussian model $Y\sim N_n(\theta, I)$, the posterior distribution $\pi(\eta\C  Y)$   satisfies  
$$
E\left(\frac{\eta}{1-\eta}\C Y\right)\leq \frac{a+E[s_n(\theta)\C  Y]+1}{b-1}\quad\text{and}\quad
E\left(\frac{1-\eta}{\eta}\C Y\right)\leq E\left(\frac{b+n}{s_n(\theta)+a-1 }\C  Y\right)
$$
where  $s_n(\theta)=\|\theta\|_0$. 
\end{lemma}
\proof Section \ref{sec:proof_lemma:posterior_odds}.
\begin{remark}
This Lemma shows that the usual calibration \cite{castillo_vdv} with $a=1$ and $b=n+1$ implies
$
E\left(\frac{\eta}{1-\eta}\C Y\right)\lesssim E[ s_n(\theta)/n\C Y]$ and $ 
E\left(\frac{1-\eta}{\eta}\C Y\right)\lesssim E[n/s_n(\theta)\C  Y].
$
While we focused on upper bounds in Lemma \ref{lemma:posterior_odds}, one can easily show lower bounds as well implying that the order of these expectations is $s_n(\theta)/n$ and $n/s_n(\theta)$, respectively. In particular, for $\lambda>2$ we obtain   from \eqref{eq:sandwich} that  $E\left(\frac{\eta}{1-\eta}\C Y\right)> \frac{a+E[s_n(\theta\C  Y]}{b-1}$.
\end{remark}

Going back to Lemma \ref{lemma:risk_hierarchical},  it is crucial to understand the {\em typical} behavior of the posterior mean of the odds $(1-\eta)/\eta$ and $\eta/(1-\eta)$ under the model $ Y\sim N_n(\theta, I).$  The following Lemma utilizes the known  property of Spike-and-Slab priors that the posterior does not {\em overshoot} the true dimensionality $s_n$ by too much (Theorem 2.1 in \cite{castillo_vdv}).
\begin{lemma}\label{lem:overshoot}
Assume $ Y\sim N_n(\theta, I)$ and the hierarchical Spike-and-Slab prior \eqref{eq:hierarchical} with $a=1$ and $b=n+1$.
 Then for some suitable $M>0$ we have
$$
\sup_{\theta\in \Theta_n(s_n)}E_{ Y\C \theta}E\left(\frac{\eta}{1-\eta}\C Y\right)\leq  
Ms_n/n+o(1)\quad\text{as $n\rightarrow\infty$}.
$$
\end{lemma}
\proof Section \ref{sec:proof_lem:overshoot}
 
\smallskip
Lemma  \ref{lem:overshoot} takes care of the ``noise" part of the risk bound in Lemma \ref{lemma:risk_hierarchical} where $(n-s_n)E_{ Y_{\backslash i}\C\theta}E[\eta/(1-\eta)\C Y_{\backslash i}]\lesssim s_n$.
In order to bound the ``signal" part of the risk bound, we need to show that $s_nE_{ Y\C\theta}E[ (1-\eta)/\eta\C Y]\lesssim s_n\log(n/s_n)$. 
From Lemma \ref{lemma:posterior_odds}, we need to make sure that the posterior $\pi(\eta\C Y)$ does not {\em undershoot} $s_n$ by too much. 
In   Lemma \ref{lem:no_miss} below, we show that when the true underlying signal is strong enough, the posterior does not miss {\em any} signal.
First, we define 
\begin{equation}
\Theta_n(s_n,\wt M)=\Theta_n(s_n)\cap\left\{\theta\in\R^n: \min_{i:\theta_i\neq 0} |\theta_i|>\wt M\sqrt{\log n}\right\}\label{eq:beta_min}.
\end{equation}
A similar (but stronger) signal strength condition was used in \cite{castillo_vdv2} in the context of high-dimensional regression with Spike-and-Slab priors. 
\begin{lemma}\label{lem:no_miss}
Assume $ Y\sim N_n(\theta, I)$ and the hierarchical Spike-and-Slab prior \eqref{eq:hierarchical} with $a=2$ and $b=n+1$. Denote with $S$ an index of all subsets of $\{1,\dots, n\}$ and define $c=(\wt M^2-2)/4$.  We have
$$
\sup_{\theta\in \Theta_n(s_n,\wt M)} P\big( \exists j\,\,\,\text{such that}\,\,\,  \theta_j\neq 0\quad\text{and}\quad j\notin S\C  Y\big)\leq \frac{C\e^{\lambda^2/2} s_n}{n^{c-1}}
$$
with probability at least $1-2/n$. Assume $\lambda>0$ such that $\lambda^2\leq 2 d\log n$ for some $d>0$. Then  for $c>2+d$ we  have
$$
\sup_{\theta\in \Theta_n(s_n,\wt M)} E_{ Y\C\theta}E\left(\frac{1-\eta}{\eta}\C Y\right)\lesssim n/s_n.
$$
\end{lemma}

\proof Section \ref{sec:proof_lem:no_miss}

\smallskip
Combining all the pieces, Theorem  \ref{thm:adaptive} below  characterizes rate-minimax performance of  the hierarchical Spike-and-Slab prior when the signals are large enough (i.e. over the parameter space \eqref{eq:beta_min}). The near-minimax performance is guaranteed over the entire parameter space $\Theta_n(s_n)$.
\begin{theorem}\label{thm:adaptive}
Assume the hierarchical prior \eqref{eq:hierarchical} with a Laplace slab \eqref{eq:lasso_prior} and with $a=2$ and $b=n+1$. 
Choose $\lambda^2=vC_r$ for $C_r>2/[v(1/2+4)]$ such that $\lambda >2$ when $0<r<1$ and  $\lambda^2=(1-v)C^*_r$  for $C^*_r>2/[5(1-v)]$ such that $\lambda>2$ when $r\geq 1$. 
Denote $c=(\wt M^2-2)/4$  where $\wt M$ is the signal-strength constant in \eqref{eq:beta_min} then we have for   $c>2$
$$
\sup_{\theta\in \Theta_n(s_n,\wt M)} \rho(\theta,\hat p)\lesssim \frac{s_n}{r+1}\log (n/s_n)\quad\text{and}\quad \sup_{\theta\in \Theta_n(s_n )} \rho(\theta,\hat p)\lesssim \frac{s_n}{r+1}\log (n).
$$
\end{theorem}
 \proof 
The proof follows directly from  Lemma \ref{lemma:risk_hierarchical}, \ref{lemma:posterior_odds}, \ref{lem:overshoot} and \ref{lem:no_miss}. The proposed calibrations ascertain that 
the term $C( \lambda,v)$ defined in \eqref{eq:Clambda} is suitably small.

\smallskip

Theorem \ref{thm:adaptive} establishes that  rate-minimax predictive performance is achievable with hierarchical mixture priors that {\em do not} require the knowledge of $s_n$. There is no other rate-adaptive predictive density result in the literature so far. While we considered a fixed $\lambda$ regime (as $n\rightarrow\infty$), for near-minimaxity (with a log factor $\log n$ instead of $\log (n/s_n)$) we could allow having $\lambda$ increase at a rate $\sqrt{\log n}$ (rescaled by $r$). Another possible  strategy to achieve adaptation would be via  empirical Bayes  or via a two-step procedure, plugging in an estimator of $s_n/n$ in place of $\eta$. We have seen from Lemma \ref{lemma:risk_hierarchical} that hierarchical priors perform this plug-in automatically. Theorem \ref{thm:adaptive} nicely complements adaptive rates of estimation results under the squared error loss for Spike-and-Slab priors obtained earlier by  \citet{castillo_vdv} (Theorem 2.2) or \citet{rockova}.  The next section extends some of the findings to the more practical case of high-dimensional regression.

\section{High-dimensional Sparse Regression} \label{sec:sparse_regression}
Predictive inference in a low-dimensional (sparse) regression was studied by \cite{george_xu} who established sufficient conditions for minimaxity and dominance of a Bayesian estimator under the uniform prior over a plug-in estimator. Here, we consider a high-dimensional scenario $p>n$ with subset selection uncertainty and an {\em unknown} and possibly diverging sparsity level $s_n$. Rather than targeting the exact minimax risk, we will establish risk rates for the Spike-and-Slab prior with either a uniform slab (mirroring the setup in \cite{george_xu}) as well as a Laplace slab.
We observe $Y=(Y_1,\dots, Y_n)'$ from 
$$
Y_i=X_i'\beta_0+\varepsilon_i\quad\varepsilon_i\iid N(0,1)
$$
where $X_i$ is a vector of $(p\times 1)$ covariate values and where $\beta_0\in\Theta_p(s_n)$, i.e. $\|\beta_0\|_0\leq s_n$, and  $p>n$.
The goal is to predict $\wt Y=(\wt Y_1,\dots, \wt Y_m)'$ where
$$
\wt Y_i=\wt X_i'\beta_0+\wt \varepsilon_i\quad\wt\varepsilon_i\iid N(0,1),
$$
where $\wt X_i$ is a vector of the same  $p$ covariates  but with possibly different values and  where $Y$ and $\wt Y$ are conditionally independent given $\beta_0\in\R^p$. Throughout this section, we assume that the columns of $X$ and $\wt X$ have been suitably normalized so that $\|X^j\|_2=\sqrt{n}$ and  $\|\wt X^j\|_2=\sqrt{m}$. We will also denote $\|X\|=\sup_{1\leq j\leq p}\|X^j\|_2$.
The Spike-and-Slab prior is underpinned by a subset indicator $S\subset\{1,\dots,p\}$ and can be written (for some $0\leq R\leq p$) as
\begin{equation}
 \pi(\beta)= \sum_{s=0}^R\pi(s)\sum_{S:|S|=s}\frac{1}{{p\choose s}}\times \pi_{S,\lambda}(\beta_S)\otimes_{i\notin S}\delta_0(\beta_i)\label{eq:SS_reg}
\end{equation}
where $\pi_{S,\lambda}(\beta_S)$ is the slab portion for the subset $\beta_S$ of active coefficients $\beta$ inside $S$ and where (as in Assumption 1 in \cite{castillo_vdv2}) for some $A_1,A_2,A_3,A_4>0$
 \begin{equation}
 A_1p^{-A_3}\pi(s-1)\leq \pi(s)\leq A_2p^{-A_4}\pi(s-1),\quad s=1,\dots, p\label{eq:prior_dim}.
 \end{equation}
 The predictive density is a mixture of conditional predictive densities, given $S$, mixed by the posterior $\pi(S\C Y)$ and can be written as (see e.g. (37) in \cite{george_xu})
\begin{equation}
\hat p(\wt Y\C Y )= \sum_{S}\pi(S\C Y) \hat p(\wt Y\C Y,S).\label{eq:multiple_shrinkage}
\end{equation}
We denote with $S_\beta=\{i:\beta_i\neq 0\}$ the support of a vector $\beta\in\R^p$. When the posterior concentrates on $S_0=S_{\beta_0}$, the posterior predictive density will be dominated by $\hat p(\wt Y\C Y,S_0)$ for which the uniform prior dominates the plug-in density estimator $\pi(\wt Y\C\hat \beta_{MLE})$ and for which the scaled harmonic prior will be minimax when $|S_0|\geq 3$ \cite{george_xu}. We leverage model selection consistency to obtain predictive risk bounds in high-dimensional scenarios $p>n$ where both $S_0$ and $s_n$ are unknown. First, we find a convenient characterization  of the KL risk.
Denote 
\begin{equation}
\Delta_{n,\beta,\beta_0}(Y,X)\equiv \frac{\pi(Y\C \beta)}{\pi(Y\C \beta_0)}=\e^{-1/2\|  X (\beta-\beta_0)\|^2_2+( Y-  X\beta_0)'  X(\beta-\beta_0)}\label{eq:delta}
\end{equation}
the normalized likelihood  and its marginal version by
$$
\Lambda_{n,\beta_0,\pi}(Y,X)=\int_{\beta}\Delta_{n,\beta,\beta_0}(Y,X)\pi(\beta)\d\beta.
$$
This notation  helps us compactly characterize the KL prediction risk.

\begin{lemma}\label{lemma:reg_risk}
 Writing $\bar X=(X',\wt X')'$ and $Z=(Y',\wt Y)'$, the KL predictive risk $\rho_{n,m}(\beta_0, \hat p)\equiv\E_{Y\C\beta_0}  KL(\pi(\wt Y\C \beta_0), \hat p(\wt Y\C Y ))$ for the regression prediction problem satisfies
 $$
\rho_{n,m}(\beta_0, \hat p) =\E_{Y\C\beta_0} \log \Lambda_{n,\beta_0,\pi}(Y,X)-  \E_{Y\C\beta_0}\E_{\wt Y\C\beta_0}\log \Lambda_{n+m,\beta_0,\pi}(Z,\bar X).
 $$
 \end{lemma}
\proof 
We have
 \begin{align*}
 \hat p(\wt Y\C Y )&=\pi(\wt Y\C \beta_0)  \frac{\int_{\beta} \Delta_{m,\beta,\beta_0}(\wt Y,\wt X)\times \Delta_{n,\beta,\beta_0}(Y,X)\pi(\beta)\d\beta }{ \int_{\beta}\Delta_{n,\beta,\beta_0}(Y,X)\pi(\beta)\d\beta}\\
 &=\pi(\wt Y\C \beta_0)  \frac{\int_{\beta} \Delta_{m+n,\beta,\beta_0}(Z,\bar X)\pi(\beta) \d\beta }{ \int_{\beta}\Delta_{n,\beta,\beta_0}(Y,X)\pi(\beta)\d\beta}.
\end{align*}
The rest follows from $\rho_{n,m}(\beta_0, \hat p)=\E_{Y\C\beta_0} KL(\pi(\wt Y\C \beta_0), \hat p(\wt Y\C Y )).$

It is useful to elaborate on the KL divergence  $KL(\pi(\wt Y\C \beta_0),\hat p(\wt Y\C Y))$ for the mixture estimator $\hat p(\wt Y\C Y)$ from \eqref{eq:multiple_shrinkage}. The following Lemma helps us understand the connection between the concentration of the posterior predictive distributions and the concentration of the parameter posteriors.
\begin{lemma}\label{eq:lemma_concentration}
For a given set of  observations $Y$ we have for $\hat p(\wt Y\C Y)$ from \eqref{eq:multiple_shrinkage}
\begin{equation}
KL(\pi(\wt Y\C \beta_0),\hat p(\wt Y\C Y)) \leq 1/2\,  E_{\beta\C  Y}    \|\wt X (\beta-\beta_0)\|^2_2 +  \sqrt{m}E_{\beta\C  Y} \|\wt X(\beta-\beta_0)\|_2.\label{eq:magic}
\end{equation}
\end{lemma}
\proof For $ \Delta_{n,\beta,\beta_0}(Y,X)$ defined in \eqref{eq:delta} we have
$$
KL(\pi(\wt Y\C \beta_0),\hat p(\wt Y\C Y))
=-\int \pi(\wt Y\C \beta_0)\log E_{S\C Y}E_{\beta\C S,Y} \Delta_{m,\beta,\beta_0}(\wt Y,\wt X)
$$ 
which implies an upper bound from Jensen's inequality $E\log X\leq \log EX$
\begin{align*}
\E_{\wt Y\C \beta_0}E_{S\C Y}E_{\beta\C S,Y}  \left[ 1/2 \|\wt X (\beta-\beta_0)\|^2_2+(\wt Y-\wt X\beta_0)'\wt X(\beta-\beta_0)\right].
\end{align*}
Applying the Cauchy-Schwartz inequality we have $\E_{\wt Y\C \beta_0} E_{\beta\C  Y} (\wt Y-\wt X\beta_0)'\wt X(\beta-\beta_0)\leq E_{\beta\C  Y}  \|\wt X(\beta-\beta_0)\|_2\E_{\wt Y\C \beta_0}\|\wt\varepsilon\|_2$.
Because $\|\wt\varepsilon\|_2^2$ has a $\chi^2_m$ distribution with $m$ degrees of freedom, the square root of this variable has an expectation $\sqrt{2}\Gamma((m+1)/2)/\Gamma(m/2)\leq\sqrt{m}$.

\smallskip

\begin{remark}
Lemma \ref{eq:lemma_concentration} invites the possibility of quickly bounding the predictive risk using concentration rate results for $\|\wt X(\beta-\beta_0)\|_2$. In the normal means model from Section \ref{sec:normal_means}, one can use Theorem 2.5 in \cite{castillo_vdv} which shows that the typical posterior probability that $\|\beta-\beta_0\|_2$ exceeds the multiple of the minimax rate $r_n=\sqrt{s_n\log (n/s_n)}$ has a sub-Gaussian tail to conclude that $\E_{Y\C\beta_0}E_{\beta\C Y}\|\beta-\beta_0\|_2^l \lesssim r_n^l$ for $l\in \mathbb N$. Together with the fact that $\|\wt\varepsilon\|_2$ has an expectation bounded by $\sqrt{r\times n}$ in the setup from Section \ref{sec:normal_means}, we obtain a loose upper bound $s_n\log(n/s_n)+\sqrt{r ns_n\log(n/s_n)}$ which does not match the minimax rate $1/(1+r)s_n\log (n/s_n).$ This justifies the somewhat more delicate treatment of the predictive risk  in Section \ref{sec:normal_means}.
\end{remark}
Lemma \ref{eq:lemma_concentration} applies to {\em any} Spike-and-Slab prior and indicates, for example, that if one can bound the {\em average} posterior mean of $\|\beta-\beta_0\|_1$  
one can directly obtain an upper bound on the prediction risk (despite perhaps quite loose).  \cite{castillo_vdv2}  showed that average posterior probability that the distance $ \|\beta-\beta_0\|_1$ exceeds the near-minimax rate $s_n\sqrt{1/n\log p}$  (with multiplication constants depending on suitable compatibility numbers) goes to zero. Using the inequality
  $\|\wt X(\beta-\beta_0)\|_2\leq \|\wt X\|\|\beta-\beta_0\|_1$  one can quickly deduce a rate $m/n\times s_n^2\log p$ from \eqref{eq:magic} when $\sqrt{n}\leq s_n\sqrt{\log p}$. This rate, however, has an unjustifiable extra factor $s_n$ and may be improved by transporting the $\ell_2$ concentration rate under some suitable of the largest eigenvalue  of $\wt X'\wt X$. One way or another, however, Theorem 2 in \cite{castillo_vdv2} {\em does not} immediately imply a bound on the posterior mean of $\|\beta-\beta_0\|_1$ or $\|\beta-\beta_0\|_2$, even for most typical $Y$. This prevents us from quickly leveraging existing posterior concentration results in regression to bound the predictive risk. We will pursue a different direction to obtain a more concrete upper bound on the KL risk, total-variation (TV) risk and a typical KL distance.
 We now focus on two particular choices of $\pi_{S,\lambda}(\cdot)$ in \eqref{eq:SS_reg}.

\subsection{The Uniform Slab}
Rather than establishing exact minimaxity, we want to understand the rate at which the KL risk $\rho_{n,m}(\beta_0, \hat p)$ increases allowing for both $p$ and $s_n$ to grow with $n$.
Since the uniform prior was shown to dominate the plug-in estimator \cite{george_xu}, it is natural to consider a version of the Spike-and-Slab prior \eqref{eq:SS_reg} with a uniform slab.

\begin{lemma}\label{lemma:uniform_slab}
Consider a Spike-and-Slab prior \eqref{eq:SS_reg} with $\pi_{S,\lambda}(\beta_S)\propto  (\lambda/2)^s$ where $s=|S|$ and with $R=\min\{n,p\}$. Assume that $\pi(s)\propto c^{-s}p^{-as}$ for some $a,c>0$. Assume that $\|X\|=\sqrt{n}$ and $\|\wt X\|=\sqrt{m}$. Then  
$$
\sup_{\beta_0\in\Theta_p(s_n)}\rho_{n,m}(\beta_0, \hat p)\leq s_n\log (\e\, c\, p^{a+1} )+s_n/2\log[2/\pi(n+m)/\lambda^2].
$$
\end{lemma}
\proof
We use Lemma \ref{lemma:reg_risk}.
Jensen's inequality $E\log X\leq \log EX$  and the fact that  $E\e^{\mu+\sigma X}=\e^{\mu^2+\sigma^2/2}$  for a standard normal r.v. $X$ implies
$$
 \E_{Y\C\beta_0}\log \Lambda_{n,\beta_0,\pi}(Y,X)\leq \log  \int_\beta \e^{-1/2\|  X (\beta-\beta_0)\|^2_2}\E_{Y\C\beta_0}\e^{(Y-  X\beta_0)'  X(\beta-\beta_0)} \pi(\beta)\d\beta =0. 
$$
Next, changing variables $b(S)\equiv \beta_{S}-\beta_{0,S}$ where $\beta_S=\{\beta_i:i\in S\}$ we find that (using similar arguments as in the proof of Theorem 6 in \cite{castillo_vdv2}) for $X_S$ a submatrix of columns of $X$ inside $S$ and $s_0=|S_0|$
\begin{align*}
\Lambda_{n+m,\beta_0,\pi}(Z,\bar X)&
\geq \frac{\pi(s_0)}{{p\choose s_0}} (\lambda/2)^{s_0}\int  \e^{-1/2\|\bar Xb(S_0) \|_2^2}\d b(S_0)\geq \frac{\pi(s_0)(\lambda/2)^{s_0}}{{p\choose s_0}}\frac{(2\pi)^{s_0/2}}{|\bar X_{S_0}'\bar X_{S_0}|^{1/2}}\\
&\gtrsim  \frac{c^{-s_0}p^{-as_0}s_0^{s_0}}{(\e p)^{s_0}}\frac{(\pi\lambda^2/2)^{s_0/2}}{\|\bar X\|^{s_0}}\geq \left(\frac{s_0}{\e\, c\, p^{a+1} }\right)^{s_0}\frac{(\pi\lambda^2/2)^{s_0/2}}{(n+m)^{s_0/2}}
\end{align*}
which yields the desired conclusion.
\begin{remark}
While the scaling of the uniform prior $\pi_{S,\lambda}(\beta)\propto  (\lambda/2)^s$ is immaterial for coefficient estimation, it does influence  posterior model probabilities. \cite{castillo_vdv2} showed that such a prior yields a posterior that is asymptotically indistinguishable  from the posterior obtained under the Laplace prior when $\lambda$ is small (e.g. $\lambda=\sqrt{n}/p$).
This choice would yield an upper bound of the order $s_n[\log(p)\vee\log (1+m/n)].$ For a fixed $\lambda$, we obtain an upper bound on the risk of the rate $s_n[\log(p)\vee\log (n+m)]$. 
\end{remark}

\subsection{The Laplace Slab}
In their Theorem 6, \cite{castillo_vdv2} show that posteriors under the uniform slab and the Laplace slab   are asymptotically indistinguishable in terms of TV distance under suitable design compatibility conditions when $\lambda$ is small. Since the posterior predictive distribution is a functional of the posterior, we can expect the posterior predictive distributions  to be indistinguishable as well. We obtain risk bounds for the total-variation  predictive risk
$$
\rho^{TV}_{n,m}(\beta_0, \hat p)=\E_{Y\C\beta_0}\| \pi(\wt Y\C\beta_0)-\hat p(\wt Y\C Y)\|_{TV}
$$
as well as upper bounds on the typical KL distance between $\hat p(\wt Y\C Y)$ and $\pi(\wt Y\C \beta_0)$.  The rate that we are targeting will be somewhat compatible with the rate of the 
typical  point prediction error $\|\wt X\beta-\wt X\beta_0\|_2^2$.
   Theorem 2 in \cite{castillo_vdv2} established the concentration rate for $\|X\beta-X\beta_0\|_2^2$ as proportional to $c(S_0){s_n\log p}$  where $c(S_0)$ depends on the properties of the design sub-matrix $X_{S_0}$ of the true model $S_0$ (such as  the compatibility number $\phi(S_0)$ defined in Definition \ref{def:compatible} and the  smallest scaled singular value $\wt \phi(s)$ defined in Definition \ref{def:smallest_scaled})
for a penalty parameter $\lambda$ satisfying $\frac{\|X\|}{p}\leq \lambda\leq 2\bar\lambda$ where $\bar\lambda=2\|X\|\sqrt{\log p}.$
Similarly as in Lemma \ref{lemma:uniform_slab}, our target rate will be $s_n[\log(p)\vee\log (1+m/n)]$. Our predictive inference risk analysis will focus on sparse $\beta_0$ vectors which allow for consistent estimation. To this end, we will impose certain identifiability conditions.

\begin{definition}(Beta-min Condition)\label{ass:beta_min}
We define  $s_n$-sparse vectors $\beta_0$ with support $S_0=S_{\beta_0}$ and strong-enough signals  as (for some $M>0$)
 \begin{equation}
\wt\Theta(s_n, M)=\left\{\beta_0\in\R^p: \|\beta_0\|_0\leq s_n\,\,\text{and}\,\,\inf_{i\in S_0}|\beta_{0i}|\geq \beta_{min}\equiv \frac{M}{\wt\psi(S_0)^2}\frac{\sqrt{|S_0|\log p}}{\|X\|\phi(S_0)}\right\}.\label{eq:betamin}
\end{equation}
where $\wt\psi(S_0)$ is the smallest scaled sparse eigenvalue of "small models" defined as
\begin{equation}
\wt\psi(S)=\wt \phi\left((2+3/A_4+33/\phi(S)^2\lambda/\bar\lambda)|S|\right).
\end{equation}
\end{definition}

Corollary 1 in \cite{castillo_vdv2} (rephrased as Theorem \ref{thm:consist} in the Supplement) concluded consistent estimation of $S_0$ uniformly for all $\beta_0\in \wt\Theta(s_n, M)$ for which $\wt\psi(S_0)$ and $\phi(S_0)$ are bounded away from zero.  In addition, for our risk analysis we impose a mild design assumption.
\begin{ass}(Design Assumption)\label{ass:design}
For $S_0=\{i:\beta_{0i}\neq 0\}$ with $s_0=|S_0|$, denote with $X_0=X_{S_0}$ and assume $\|X_0\|=\sqrt{n}$ and that for some $0<b<1$ and $d>0$
$$
\frac{4\times \chi^2_{s_0,1/2}}{(1-b)^2\beta_{min}^2}<\lambda_{min}(X_0'X_0)<\lambda_{max}(X_0'X_0)\lesssim n (\log p)^{d}.
$$
where $\chi^2_{s_0,1/2}$ is the median of the Chi-squared distribution with $s_0$ degrees of freedom and  $\beta_{min}$ is the signal threshold from \eqref{eq:betamin}.
\end{ass}
\begin{remark}
Since $s_0-1<\chi^2_{s_0,1/2}<s_0$ for $s_0\geq 2$, the smallest eigenvalue part of Assumption \ref{ass:design} essentially requires that $\lambda_{min}(X_0'X_0)$ is at least of the order $n/\log p$. The upper bound allows another logarithmic deviation from the orthogonal design. From an upper bound    $\lambda_{max}(X_0'X_0)\leq \max\limits_{j=1,\dots, s_0} \|(X_0'X_0)^j\|_1 \leq ns_0$ (which holds due to the Cauchy-Schwartz inequality  and because $\|X_0\|=\sqrt{n}$) we see that the upper bound requirement is satisfied when $s_0\lesssim (\log p)^d$.
 \end{remark}
 The next Theorem characterizes the TV predictive risk as well as the concentration rate for the KL distance for high-dimensional sparse regression with Spike-and-Slab priors with Laplace slab. 
\begin{theorem}\label{thm:final}
Consider a Spike-and-Slab prior \eqref{eq:SS_reg} with $\pi_{S,\lambda}(\beta)= (\lambda/2)^s\e^{-\lambda\|\beta_S\|_1}$ where $\lambda=\sqrt{n}/p$ and with $R=p$ where $\pi(s)$ satisfies \eqref{eq:prior_dim} with $A_4>2$.   Under Assumption  \ref{ass:design} we have for $s_n\leq n=o(p/\sqrt{\log p})$, some suitable $M>0$, $d>0$ and for any $c_0>0$
$$
\sup_{\beta\in\wt\Theta(s_n): \phi(S_0)\wedge\wt\psi(S_0)\geq c_0} \left[\rho^{TV}_{n,m}(\beta_0, \hat p)\right]^2 \lesssim \mu_n
$$
and
$$
\sup_{ \beta\in\wt\Theta(s_n):  {\phi(S_0)\wedge\wt\psi(S_0)\geq c_0 }} \mathbb P_{Y\C\beta_0}\left( KL(\pi(\wt Y\C \beta_0),\hat p(\wt Y\C Y))\gtrsim \mu_n\right)=o(1). 
$$
for $\mu_n=s_n[(\log p)^{d-1\vee 1}\vee \log(1+m/n)]$.
\end{theorem}
\proof Section \ref{sec:proof_thm:final}
 
 \smallskip
 The proof of Theorem \ref{thm:final} relies on subset selection consistency and utilizes a novel upper bound on the marginal likelihood under the Laplace prior.
Similarly as in Section \ref{sec:hierarchical}, the rate in Theorem \ref{thm:final} is  again {\em adaptive}, i.e. the prior does not depend on $s_n$, and (for $m/n\lesssim \log p$) is proportional to the minimax estimation rate $s_n\log (p/s_n)$ up to a logarithmic factor. Such adaptation 
is a collateral benefit of the fully Bayesian treatment of sparse regression.
The results from Section \ref{sec:normal_means} cannot be quickly concluded as a special case with $X=\wt X=I$. Our analysis of the sparse normal-means model in Section \ref{sec:normal_means} was sharper, where under the beta-min condition we were able to obtain the actual minimax rate $1/(1+r)s_n\log(n/s_n)$ without any additional logarithmic factors.

\section{Discussion}\label{sec:discuss}
This paper investigates several widely used priors from a predictive inference point of view. We establish a negative result for the Bayesian LASSO, showing that posterior predictive densities under this prior cannot converge to the  true density of future data at an optimal rate for the usual tuning that would yield an optimal posterior mode in estimation. Next, we study the popular Dirac Spike and Laplace Slab mixture prior and the Spike-and-Slab LASSO prior and show that proper calibrations (that depend on $s_n$ and $r$) yield rate-minimaxity.
By considering a hierarchical extension, we show that adaptation to $s_n$ is possible with the usual beta-binomial prior on the sparsity pattern. Finally, we obtain an adaptive rate for predictive density estimation in high-dimensional regression which mirrors the estimation rate of sparse regression coefficients.
 

\section{Proof of Theorem \ref{thm:dirac_laplace}} \label{sec:proof_thm:dirac_laplace}
We recall the risk decomposition in Remark \ref{rem:decompose} (or equivalently in Lemma 3 in \cite{muk1})
\begin{equation}
\rho(\theta,\hat p)=E  g(Z,\theta,v),\quad\text{where}\quad  g(z,\theta,v)=\theta^2/(2r)+\log N_{\theta,1}(z)-\log N_{\theta,v}(z),\label{eq:risk_g}
\end{equation}
where the expectation is taken over $Z\sim N(0,1)$ and where 
$$
N_{\theta,v}(z)=1+\frac{\eta}{1-\eta}\int \exp \left\{\mu\left[z/\sqrt{v}+\theta/v\right]-\frac{\mu^2}{2v}\right\}\pi_1(\mu\C\lambda)\d\mu
$$
with $\pi_1(\mu\C\lambda)=\lambda/2\e^{-\lambda|\mu|}$ for  $\lambda>0.$ We find that
$$
\log N_{\theta,v}(z)=\log \left[1+\frac{\eta}{1-\eta}\frac{\lambda}{2}(I_1^v+I_2^v)\right]
$$
where
\begin{equation}\label{eq:Is} 
I_1^v=\sqrt{v}\frac{\Phi(\mu_1/\sqrt{v})}{\phi(\mu_1/\sqrt{v})} \quad\text{and}\quad I_2^v=\sqrt{v}\frac{\Phi(-\mu_2/\sqrt{v})}{\phi(-\mu_2/\sqrt{v})}
\end{equation}
and
\begin{equation}
\mu_1=z\sqrt{v}+\theta-v\lambda\quad\text{and}\quad  \mu_2=z\sqrt{v}+\theta+v\lambda.\label{eq:mus}
\end{equation}
The reversed risk characterization in Lemma \ref{lemma:decompose} will appear later in the section below.
 \subsection{The case when $\theta\neq 0$.}
Define $t_v=\lambda\sqrt{v}  -|z+\theta/\sqrt{v}|$. When $z+\theta/\sqrt{v}>0$ we have $I_1^v>I_2^v$ and the reverse is true when $z+\theta/\sqrt{v}\leq0$.
This observation yields the following bounds
$$
\log N_{\theta,1}(z)\leq \log\left(1+\frac{\eta\lambda }{1-\eta} R_1\right)\quad\text{and}\quad
\log N_{\theta,v}(z)> \log\left(1+\frac{\eta\lambda\sqrt{v} }{(1-\eta)2} R_v\right) 
$$
where (using the Mills ratio notation $R(x)=(1-\Phi(x))/\phi(x)$)
$$
R_v =R(t_v)= \frac{1-\Phi(t_v)}{\phi(t_v)}.
$$
We consider two scenarios for bounding the function $g(z,\theta,v)$ in \eqref{eq:risk_g} depending on the magnitude of   $|z+\theta|$ (which has the same distribution as $|Y|$).  In Section \ref{sec:separable} we emphasized that depending on $|Y|$,
the entire predictive density is dominated by either the slab or the spike predictive density. We exploit this idea but only for the proof, not for the construction of an actual estimator \cite{muk1}.    
\subsubsection{ Case (i): $1/R_1\leq [\eta /(1-\eta)]^{v}$}   \label{sec:casei}
 As will be seen later, this event is equivalent to $|z+\theta|$  being large enough. Since $ z+\theta$ is a proxy for $Y$, we would then expect the risk to be dominated by the slab 
risk.  We can write
\begin{equation}
\log \frac{N_{\theta,1}(z)}{N_{\theta,v}(z)}\leq\log\left(
\frac{R_1}{R_v}\right)+\log\left(\frac{\frac{1}{R_1}+\frac{\eta}{1-\eta}\lambda}{\frac{1}{R_v}+\frac{\eta}{1-\eta}\frac{ \sqrt{v}}{2}\lambda}\right).\label{eq:two_terms}
\end{equation}
This expression, in fact, brings us back to the risk decomposition  in Lemma \ref{eq:SS_slab_risk}  written in terms of the  risk of the slab Bayesian LASSO predictive density $\hat p_1(\cdot)$ plus an extra component (the second summand in \eqref{eq:two_terms}).
Indeed, $\theta^2/(2r)+\log(R_1/R_v)$ bounds the LASSO risk in Lemma \ref{lemma:lasso_risk_decompose} and the second summand in \eqref{eq:two_terms} corresponds to $E \log N_{\theta,1}^{SS}(z)-
E\log N_{\theta,v}^{SS}(z)$ in Lemma \ref{eq:SS_slab_risk}. The goal is to show that the second summand in \eqref{eq:two_terms} is small. Indeed, we have
\begin{equation}
\log\left( {\frac{1}{R_1}+\frac{\eta}{1-\eta}\lambda}\right)-\log\left({\frac{1}{R_v}+\frac{\eta}{1-\eta}\frac{ \sqrt{v}}{2}\lambda}\right)\leq \log\left(\frac{2}{\sqrt{v}}\right)+ \log \left[1+\frac{1}{\lambda}\left(\frac{1-\eta}{\eta }\right)^{1-v}\right].\label{eq:term2}
\end{equation}
Next, we have 
\begin{equation}
\frac{R_1}{R_v}<\frac{1}{\Phi(-t_v)}\exp\left\{\frac{\lambda^2(1-v)}{2}+\lambda|z|(1-\sqrt{v})+\frac{\theta^2}{2}\left(1-\frac{1}{v}\right)+z\theta\left(1-\frac{1}{\sqrt{v}}\right)\right\}\label{eq:term1}.
\end{equation}
Now we focus on the term $1/\Phi(-t_v)$ in \eqref{eq:term1}.  We have $-t_v>-\lambda\sqrt{v}$ and (because $\lambda\sqrt{v}>0$ we can apply Lemma \ref{lem:mills})
$$
\frac{1}{\Phi(-t_v)}<\frac{1}{1-\Phi(\lambda\sqrt{v})}\leq   \frac{\sqrt{\lambda^2v+4}+\lambda\sqrt{v}}{2\phi(\lambda\sqrt{v})}\leq   \frac{\sqrt{\lambda^2 +4}+\lambda }{2\phi(\lambda\sqrt{v})}. 
$$
This yields (using Lemma \ref{lemma:aux_function})
\begin{equation}
\log[1/\Phi(-t_v)]\leq \log \sqrt{2\pi} +\frac{\lambda^2v}{2}+\frac{4}{\lambda^2 }+\log(\lambda ). \label{eq:term3}
\end{equation}
Now, we  understand when the event $\{1/R_1<[\eta\lambda/(1-\eta)]^{v}\}$ actually occurs.
Because the Mills ratio function $R(x)$ is monotone decreasing and  satisfies  
$$
0.5/\phi(x)  \leq R(x)\leq 1/\phi(x)\quad\text{when}\quad x<0
$$
and because $(1-\eta)/\eta=n/s_n$ goes to infinity, we have 
\begin{equation}
R(x_{\eta })=\left[\frac{1-\eta}{ \eta }\right]^{ v}\,\,\,\text{for some}\,\,\,   h^L_{\eta }<x_{\eta }<h^U_{\eta }<0\label{eq:x_threshold}
\end{equation}
where
\begin{equation}
h^U_{\eta }=-\sqrt{2v\log\left[( \pi 2)^{-v/2}(1-\eta)/\eta\right]}\quad\text{and}\quad h^L_{\eta }=-\sqrt{2v\log\left[(\pi/2)^{-v/2} (1-\eta)/ \eta\right]}.\label{eq:hs}
\end{equation}
Then we have
$$
\mathcal A_1\equiv \left\{\frac{1}{R_1}\leq\left(\frac{\eta}{1-\eta}\right)^v\right\}=\{ t_1\leq x_{\eta }\}\subset\{ t_1\leq h^U_{\eta }\}=\{|z+\theta|\geq \lambda-h^U_{\eta }\}.
$$
Interestingly, this corresponds to the regime when the magnitude of  $ z+\theta$ is above $-h^U_{\eta,\lambda}$, the usual detection threshold $\sqrt{2\log(n/s_n)}$ suitably rescaled by $v$. 
This implies  that the slab is active when the size of the observed data $|x|$ is large, yielding  an upper bound 
 $$
  g(z,\theta,v)<   g_1(z,\theta,v)
 $$
 where 
 \begin{align}
 g_1(z,\theta,v)=& -\log\Phi(-t_v)+ \frac{\lambda^2(1-v)}{2}+\lambda|z|(1-\sqrt{v})+z\theta\left(1-\frac{1}{\sqrt{v}}\right)\\
 &+\log(4/\sqrt{v}) + (1-v)\log \left( \frac{1-\eta}{\eta }\right)+\log(1+1/\lambda).
 \label{eq:bound1}
 \end{align}

\subsubsection{ Case (ii): $1/R_1 \geq[\eta /(1-\eta)]^{v}$} \label{sec:caseii}
Using similar arguments as before in the Case (i), we have
$$
\mathcal A_2\equiv \left\{\frac{1}{R_1}>\left(\frac{\eta}{1-\eta}\right)^{v}\right\}=\{ t_1>x_{\eta }\}\subset\{ t_1>h^L_{\eta}\}=\{|z+\theta|<\lambda-h^L_{\eta }\}.
$$
This regime mirrors the scenario when the observed data is below the detection threshold. In this case, the slab risk plays a minor role  and we obtain a simplified expression
$$
\log \frac{N_{\theta,1}(z)}{N_{\theta,v}(z)}<\log\left(\frac{1+\frac{\eta}{1-\eta}\lambda R_1}{1+\frac{\eta}{1-\eta}\frac{ \sqrt{v}}{2}\lambda R_2}\right)<\log\left[1+\lambda\left(\frac{\eta}{1-\eta}\right)^{1-v}\right].
$$
Next, we have (using $(a+b)^2\leq 2a^2+2b^2$) and because $v/r=1/(r+1)=(1-v)$
$$
\frac{\theta^2}{2r}\leq \frac{ |z+\theta|^2+|z|^2}{ r}\leq \frac{2\lambda^2 +|z|^2}{r}+  4(1-v) \log\left(\frac{1-\eta}{\eta }\right).  
$$
This implies an upper bound $g(z,\theta,v)\leq g_2(z,\theta,v)$ where
\begin{align}
g_2(z,\theta,v)&=   \frac{2\lambda^2 +|z|^2}{r}+  4(1-v) \log\left(\frac{1-\eta}{\eta}\right)+\log(1+\lambda). \label{eq:bound2}
\end{align}

\subsubsection{Combining the Cases}\label{sec:cases_combined}
We combine the two bounds in \eqref{eq:bound1} and \eqref{eq:bound2} to obtain
\begin{equation}
g(z,\theta,v) \leq g_1(z,\theta,v)\mathbb I(z\in \mathcal A_1)+g_2(z,\theta,v)\mathbb I(z\in \mathcal A_2).\label{eq:bound_paste}
\end{equation}
Here, we are pasting two risk bounds depending on the two scenarios for $|z+\theta|$ as opposed to pasting two estimators \cite{muk1}. 
Next, we define 
$$
f(\theta,z,v,\eta)= 2z\theta\left(1-\frac{1}{\sqrt{v}}\right) \mathbb I(z\in \mathcal A_1)
$$
which occurs in the bound $g_1(z,\theta,v)$. It turns out that, on the event $\mathcal A_1$, this term averages out to a negative value, i.e. we have  
$E\, \mathbb I(z\in  \mathcal A_1) f(\theta,z,v,\eta)<0$. Indeed, with $c_\eta=\lambda-h^U_{\eta }$
$$
E\, \mathbb I(z\in \mathcal A_1)(\theta z)=\int_{z>-\theta+c_{\eta}}\theta z\phi(z)+\int_{z<-\theta-c_{\eta}}\theta z\phi(z)
=\theta [\phi(-\theta+c_{\eta})-  \phi(-\theta-c_{\eta})]>0
$$
which yields $E\,\mathbb I(z\in \mathcal A_1)f(\theta,z,v,\eta)<0$ because $1-1/\sqrt{v}<0$.

Combining \eqref{eq:bound1}, \eqref{eq:bound2} and \eqref{eq:term3}  and using the fact that $|z|$ has a folded normal distribution with   $E|z|=\sqrt{2/\pi}$
and $E|z|^2=1$
\begin{align}
E g(z,\theta,v) 
\leq& 
 \log \left(4\sqrt{2\pi/v}\right)  +\frac{4}{\lambda^2 } + \frac{\lambda^2 }{2}+\frac{2\lambda^2 +1}{r} +\lambda (1-\sqrt{v})\sqrt{2/\pi} \\
 &  + 5(1-v)\log \left( \frac{1-\eta}{\eta }\right)+\log(2\lambda+\lambda^2+1).
\label{eq:bound_total}
 \end{align}
 With fixed $r\in(0,\infty)$, fixed $\lambda>0$ and $(1-\eta)/\eta=n/s_n$, we can see that the upper bound is of the order $(1-v)\log (n/s_n)$.
 What is  concerning is the case when $r$ is very small which may take a large $n$ for the term $(1-v)\log (n/s_n)$ to dominate. It might be worthwhile to consider
 tuning $\lambda$ according to $r$ to make sure the the bound is valid also for smaller $n$.
 
 We now consider two scenarios: (a) $r>1$ when the variance of the future observation $y$ is larger than that of $x$, and (b) $0<r<1$ when the training data is noisier. 
 The case (a) $r>1$ implies $v>1/2$ and for calibrations $1/[(1-v)\log (n/s_n)]\lesssim\lambda^2\lesssim (1- v) \log (n/s_n)$  (which includes setting $\lambda$ equal to a fixed constant that does not depend on $n$) 
 imply
 $
 \rho(\theta,\hat p)\lesssim (1-v)\log(n/s_n).
 $
 For instance, we can choose $\lambda^2=C_r^*(1-v)$ for $C_r^*>0$ such that $C_r^*>2/[5(1-v)]$. Then
\begin{equation}
\rho(\theta,\hat p) \leq 5(1-v)\log (n/s_n)+\wt C_r^*,  \label{eq:rho}
\end{equation}
where
\begin{equation}
\wt C_r^*=\log(8\sqrt{\pi})+10+2C_r^*(1-v)+C_r^*\sqrt{2/\pi}+\log (2\sqrt{C_r^*(1-v)}+C_r^*(1-v)).\label{eq:remainder1}
\end{equation}
 In the case (b) $0<r<1$ we have $0<v<1/2$.
We can choose $\lambda^2=C_r v$ for $C_r>\frac{2}{v(1/2+4)}$. This yields
\begin{align*}
\frac{4}{\lambda^2 }+\frac{\lambda^2 }{2}+\frac{2\lambda^2 +1}{r} +\lambda (1-\sqrt{v})\sqrt{2/\pi} &< 9 +\frac{\lambda^2(1/2+4)+2}{ r} +\sqrt{C_r v}(1-\sqrt{v})\sqrt{2/\pi}\\
&<9+ 9\frac{C_r}{r+1} +
\sqrt{C_r v}(1-\sqrt{v})\sqrt{2/\pi}
\end{align*}
and
 $$
 \log \left(4\sqrt{2\pi/v}\right)+\log(2\lambda+\lambda^2+1)< \log (8\sqrt{2\pi})+\log (2\sqrt{C_r}+C_r)
  $$
 and 
$
\rho(\theta,\hat p) \leq 5(1-v)\log (n/s_n)+\wt C_r,  
$
 where
\begin{equation}
 \wt C_r=\log (8\sqrt{2\pi})+9+ \frac{C_r}{r+1}(1/2+4)+
\sqrt{C_r }\sqrt{1/(8\pi)}+\log (2\sqrt{C_r}+C_r)\label{eq:remainder2}.
\end{equation}
\subsubsection{The case when $\theta=0$}
This step is analogous to Section 2 in \cite{muk1}. Because $N_{\theta,v}(z)>1$ and from  Jensen's inequality and the fact that $E\e^{cZ}=\e^{-c^2/2}$ 
$$
E\log N_{\theta,v}(Z )\leq \log E N_{\theta,v}(Z )=1+\frac{\eta}{1-\eta}\int \exp(\mu\theta/v)\pi_1(\mu\C\lambda)\d\mu
$$ 
we have for $\theta=0$
\begin{equation}
\rho(0,\hat p)<\log\left(1+\frac{\eta}{1-\eta}\right)<\frac{\eta}{1-\eta}.\label{eq:rho0}
\end{equation}

\medskip

\noindent The statement  \eqref{eq:rate_optimal} of Theorem \ref{thm:dirac_laplace} follows from the fact that, for separable (independent product) priors, 
$$
\rho_n(\theta,\hat p)=\sum_{i=1}^n\rho(\theta_i,\hat p)\leq (n-s_n)\rho(0,\hat p)+s_n\sup_{\theta\in \R}\rho(\theta,\hat p).
$$
Plugging  the inequalities \eqref{eq:rho} and \eqref{eq:rho0} into the expression above, we obtain the desired statement.






\clearpage

\section*{Appendix}
The Appendix contains proofs of all remaining theorems as well as auxiliary lemmata.


\section{Proof of Theorem \ref{lemma:risk_lasso_upper}}\label{sec:proof_lemma:risk_lasso_upper}
Similarly as in the proof of Theorem \ref{thm:dirac_laplace}, we separate the cases when $\theta\neq0$ and when $\theta=0$.
 \subsection{The case when $\theta\neq 0$.}\label{sec:signal_lasso}
We find an upper bound on $\rho(\theta,\hat p)$ for the case when $\theta\neq 0$ using  Lemma \ref{lemma:lasso_risk_decompose}.
Using the prior \eqref{eq:lasso_prior}, we have
$$
\log N_{\theta,v}^{LASSO}(z)=\log \left[\frac{\lambda}{2}(I_1^v+I_2^v)\right]
$$
where $I_1^v$ and $I_2^v$ were defined in \eqref{eq:Is}. Recall also the definition of $\mu_1$ and $\mu_2$ in \eqref{eq:mus}.
To find and upper and lower bound to $\log N_{\theta,v}(z)$,
we consider two plausible scenarios (a) $\mu_1>-\mu_2$ which is equivalent to $I_1^v>I_2^v$ and to $\{z>-\theta/\sqrt{v}\}$, and (b)  $\mu_1\leq -\mu_2$ which is equivalent to $I_2^v\geq I_1^v$ and  to $\{z\leq -\theta/\sqrt{v}\}$.

\subsubsection{Upper bound on $E\log N_{\theta,v}^{LASSO}(Z)$}
We consider  the two cases (a) $I_1^v>I_2^v$ and (b) $I_1^v\leq I_2^v$ and write 
\begin{equation}
E \log N_{\theta,v}^{LASSO}(Z)\leq \log(\lambda\sqrt{2\pi v})+ T_1^v(\theta)+ T_2^v(\theta)\label{eq:LASSO_upper}
\end{equation}
where
\begin{equation}
T_1^v(\theta)=E \,\mathbb{I} (z>-\theta/\sqrt{v})\mu_1^2/(2v)\quad\text{and}\quad
T_2^v(\theta)=E\, \mathbb{I}( z\leq-\theta/\sqrt{v}) \mu_2^2/(2v).\label{eq:Ts}
\end{equation}
We find that 
\begin{equation}
T_1^v(\theta)+T_2^v(\theta)= E \frac{(|z+\theta/\sqrt{v}|-\lambda\sqrt{v})^2}{2}=\frac{\theta^2}{2v}+\frac{1+\lambda^2v}{2}- \lambda E |z\sqrt{v}+\theta|.\label{eq:Tsum}
\end{equation}

\subsubsection{Lower bound on $E\log N_{\theta,v}^{LASSO}(Z)$}
Considering again the two cases (a) $I_1^v>I_2^v$ and (b) $I_1^v\leq I_2^v$ we find that
\begin{equation}\label{eq:LASSO_lower}
E \log N_{\theta,v}^{LASSO}(z)>\log(\lambda\sqrt{\pi v/2})+ T_1^v(\theta)+ T_2^v(\theta)+T^3_v(\theta)
\end{equation}
where $T_1^v(\theta)$ and $T_1^v(\theta)$ were defined earlier in \eqref{eq:Ts} and where (because in the case (a) $\mu_1>-\lambda v$ and in the case (b) $-\mu_2>-\lambda v$)
\begin{align}
T^3_v(\theta)=&[{P(\mu_1>0)+ P(\mu_2<0)}]\log 1/2+P\left( |z+\frac{\theta}{\sqrt{v}}|<\lambda\sqrt{v}\right)\log \Phi(-\lambda\sqrt{v}).\label{eq:T3}
\end{align}
Recall the definition of the Gaussian Mills ratio $R(x)=(1-\Phi(x))/\phi(x)$.
Then  
$$
T^3_v(\theta)>\log 1/2+P\left( |z+\frac{\theta}{\sqrt{v}}|<\lambda\sqrt{v}\right)\left(\log \phi(\lambda\sqrt{v})+\log R(\lambda\sqrt{v})\right).
$$
Next, using the lower bound on the Gaussian Mills ratio in Lemma \ref{lem:mills}  
we have
\begin{equation}
T^3_v(\theta)>\log 1/2+P\left( |z+\frac{\theta}{\sqrt{v}}|<\lambda\sqrt{v}\right)\left(-\log\sqrt{2\pi}-\frac{\lambda^2v}{2}-\log\frac{\sqrt{\lambda^2v+4}+\lambda\sqrt{v}}{2} \right).\label{eq:lb_T3}
\end{equation}

\subsubsection{Combining the bounds}\label{sec:upper_lasso_bound}
Combining the upper bound for $E \log N_{\theta,1}^{LASSO}(z)$ and the lower bound for $E \log N_{\theta,v}^{LASSO}(z)$, Lemma \ref{lemma:lasso_risk_decompose} yields that for any $\theta\in \R$
\begin{align*}
\rho(\theta,\hat p)\leq&  
\log(2/\sqrt{v})+\frac{\lambda^2(1-v)}{2}+\lambda E|z\sqrt{v}+\theta|-\lambda E|z+\theta|- T^3_v(\theta).
\end{align*}
Next, we use the fact that  $|z\sqrt{v}+\theta|-|z+\theta|\leq |z|(1-\sqrt{v})$ and that $|z|$ has a folded normal distribution with a mean $\sqrt{2/\pi}$. From the bound
$$
-T^3_v(\theta)<\log (2\sqrt{2\pi})+ \frac{\lambda^2v}{2}+\log\frac{\sqrt{\lambda^2v+4}+\lambda\sqrt{v}}{2} 
$$
we obtain
$$
\rho(\theta,\hat p)\leq 
\log(\lambda 4\sqrt{2\pi/v} )+\frac{\lambda^2}{2}+\lambda\sqrt{2/\pi} +  \log\frac{\sqrt{1+4/(\lambda^2)}+1}{2}.
$$
Using Lemma \ref{lemma:aux_function} we obtain the desired upper bound on  $\rho(\theta,\hat p)$.

\subsection{The case when $\theta=0$.}\label{sec:noise_lasso}
By Jensen's inequality $E\log X\leq \log E X$ we find from the Fubini's theorem and from the fact that $E \exp(\mu Z/\sqrt{v})=\exp(\mu^2/(2v))$  
$$
E \log N_{0,1}^{LASSO}(Z)<\log E \int \exp \left\{\mu\ Z/\sqrt{v}  -\frac{\mu^2}{2v}\right\}\pi_1(\mu\C\lambda)\d\mu=0.
$$ 
This yields $\rho(0,\hat p)\leq -E\log N_{0,v}^{LASSO}(z)$. To find a lower bound for $E\log N_{0,v}^{LASSO}(z)$, we use the  notation introduced  $I_1^v$ and $I_2^v$ in \eqref{eq:Is} but now for the special case when $\theta=0$. 
Similarly as in Section \ref{sec:signal_lasso} we consider two cases $(a)$ when $z>0$ we have $I_1^v>I_2^v$ and   (b) when $z\leq0$ we have $I_1^v\leq I_2^v$. 
Next, in the case (a) (since $\theta=0$) we have $\mu_2=z\sqrt{v}+\lambda v>0$ and thereby we can use Lemma \ref{lem:mills} which yields
\begin{equation}\label{eq:mills_mu2}
\frac{2}{\sqrt{\mu_2^2/v+4}+\mu_2/\sqrt{v}}<I_2^v=\frac{1-\Phi(\mu_2/\sqrt{v})}{\phi(\mu_2/\sqrt{v})}<\frac{2}{\sqrt{\mu_2^2/v+2}+\mu_2/\sqrt{v}}.
\end{equation}
Similarly, in the case (b) we have $\mu_1=z\sqrt{v}-\lambda v<0$ and thereby
\begin{equation}\label{eq:mills_mu1}
 \frac{2}{\sqrt{\mu_1^2/v+4}-\mu_1/\sqrt{v}}<I_1^v=\frac{ 1- \Phi(-\mu_1/\sqrt{v})}{\phi(\mu_1/\sqrt{v})}< \frac{2}{\sqrt{\mu_1^2/v+2}-\mu_1/\sqrt{v}}.
\end{equation}
While here we only need the lower bounds, in the next Section \ref{sec:lemma:lb} we utilize also the upper bounds. 
This yields 
\begin{align}
E \log\left[\frac{\lambda}{2}(I_1^v+I_2^v)\right]>& \log(\lambda\sqrt{v})-E \log (\lambda\sqrt{v}+|z|)-E  \log\frac{\sqrt{\frac{4}{(\lambda \sqrt{v}+|z|)^2}+1}+1}{2}\label{eq:lb2}
\end{align}
We use  Jensen's inequality $E\log (\lambda\sqrt{v}+|z|) \leq \log (\lambda\sqrt{v}+\sqrt{2/\pi})$ and Lemma \ref{lemma:aux_function} which yields
$$
E  \log\frac{\sqrt{\frac{4}{(\lambda \sqrt{v}+|z|)^2}+1}+1}{2}<\log\frac{\sqrt{\frac{4}{ \lambda^2 v}+1}+1}{2}<\frac{4}{\lambda^2v}.
$$
Altogether, we find 
$$
\rho(0,\hat p)<-E \log\left[\frac{\lambda}{2}(I_1^v+I_2^v)\right]<  \log \left(1+\frac{\sqrt{2}}{\lambda\sqrt{\pi v}}\right)+ \frac{4}{\lambda^2 v}.
$$

\section{Lower Bound for the LASSO}\label{sec:lemma:lb}

We know that for separable priors (such as the Laplace product prior \eqref{eq:lasso_prior}) we have
$$ 
(n-s_n)\rho(0,\hat p)< \rho_n(\theta,\hat p)=\sum_{i=1}^n\rho(\theta_i,\hat p)=(n-s_n)\rho(0,\hat p)+\sum_{i:\theta_i\neq 0}\rho(\theta_i,\hat p).
$$
We focus on the lower part of these inequalities  and  obtain a lower bound for 
$$
\rho(0,\hat p)= E \log N_{1,0}^{LASSO}(z)- E \log N_{v,0}^{LASSO}(z).
$$

\begin{lemma}\label{lemma:lb_lasso}
Consider the Laplace prior \eqref{eq:lasso_prior} with $\lambda>0$.  For $v=1/(1+1/r)$ the univariate Bayesian LASSO predictive distribution $\hat p$ satisfies  for any $a>0$  
$$
\rho(0,\hat p)>[1-\Phi(a)]\log\left[1+
{\left(\frac{1}{\sqrt{v}}-1\right)\frac{a}{\lambda+a}}\right]-\frac{4}{\lambda^2}-{\e^{-\lambda^2v/2}}\left(\frac{1}{2}+\log(\lambda\sqrt{2\pi v})+ \frac{2}{\lambda\sqrt{2\pi v}}\right).
$$
\end{lemma}
\proof

First, we recall the lower bound for $E\log N_{1,0}^{LASSO}(z)$ in \eqref{eq:lb2} obtained in Section \ref{sec:noise_lasso}. 
$$
E \log N_{1,0}^{LASSO}(Z)>- E \log \left(1+\frac{|z|}{\lambda}\right)+E  \log\frac{2}{\sqrt{\frac{4}{(\lambda+|z|)^2}+1}+1}.
$$
Now, we obtain an upper bound for $E \log N_{v,0}^{LASSO}(z)=E\log[\frac{\lambda}{2}(I_1^v+I_2^v)]$ using similar ideas as in Section \ref{sec:noise_lasso}. Recall the notation $I_1^v$ and $I_2^v$ in \eqref{eq:Is}  and $\mu_1$ and $\mu_2$ in \eqref{eq:mus}. These quantities now tacitly assume  $\theta=0$.
We again consider two cases (a) $z> 0$ and (b) $z\leq0$ but we split them further depending on whether $|z|>\lambda\sqrt{v}$ or $|z|\leq\lambda\sqrt{v}$. In the case (a) the term $I_1^v$ dominates $I_2^v$ and  when $\mu_1\leq 0$ (i.e. $z\leq \lambda\sqrt{v}$) we can use the upper part in the Mills ratio bounds \eqref{eq:mills_mu1}. Similarly, in the case (b) when $\mu_2<0$ (i.e. $z>-\lambda\sqrt{v}$) we can use the upper part in the Mills ratio bound \eqref{eq:mills_mu2}. This yields  
\begin{align*}
E  \mathbb I(|z|&\leq \lambda\sqrt{v})\log \left[\frac{\lambda}{2}(I_1^v+I_2^v)\right]\\
&\leq  -E\,\mathbb{I}(|z|\leq \lambda\sqrt{v}) \log \left(1+\frac{|z|}{\lambda\sqrt{v}}\right)+E\,\mathbb{I}(|z|\leq \lambda\sqrt{v})   \log\frac{2}{\sqrt{\frac{2}{(\lambda\sqrt{v}+|z|)^2}+1}+1}\\
&\leq -E \log \left(1+\frac{|z|}{\lambda\sqrt{v}}\right) +E\,\mathbb{I}(|z|> \lambda\sqrt{v}) \log \left(1+\frac{|z|}{\lambda\sqrt{v}}\right)
+  \log\frac{2}{\sqrt{\frac{1}{2\lambda^2v}+1}+1}.
\end{align*}
Next we find that (using the fact that $|z|$ has a folded normal distribution with a density $2/\sqrt{2\pi}\e^{-z^2/2}$)
$$
E\,\mathbb{I}(|z|> \lambda\sqrt{v}) \log \left(1+\frac{|z|}{\lambda\sqrt{v}}\right)<\frac{2}{\sqrt{2\pi}}\int_{\lambda\sqrt{v}}^\infty\frac{z}{\lambda\sqrt{v}}\e^{-z^2/2}\d z=\frac{2}{\sqrt{2\pi}}\frac{\e^{-\lambda^2v/2}}{\lambda\sqrt{v}}
$$
and thereby
$$
E\,\mathbb I(|z|\leq \lambda\sqrt{v})\log N_{v,0}^{LASSO}(z)<-E \log \left(1+\frac{|z|}{\lambda\sqrt{v}}\right)+\frac{2}{\sqrt{2\pi}}\frac{\e^{-\lambda^2v/2}}{\lambda\sqrt{v}}.
$$
For the remaining scenario when $|z|>\lambda\sqrt{v}$ we bound (using the Gaussian tail bound $(1-\Phi(x))\leq \e^{-x^2/2}$ for $x>0$)
\begin{align*}
E\,\mathbb I(|z|> \lambda\sqrt{v})\log  N_{0,v}^{LASSO}(z) \leq& \log(\lambda \sqrt{ 2\pi v})(1-\Phi(\lambda\sqrt{v}))\\
&+  E [\,\mathbb{I}(z>\lambda\sqrt{v})\mu_1^2/(2v)+ \mathbb{I}(z<-\lambda\sqrt{v})\mu_2^2/(2v)] \\
\leq&  \log(\lambda \sqrt{2\pi v})\e^{-\lambda^2v/2}+ E  (|z|-\lambda\sqrt{v})^2\mathbb I(|z|\geq \lambda\sqrt{v})/2.
\end{align*}
We note that
$$
E  (|z|-\lambda\sqrt{v})^2\mathbb I(|z|\geq \lambda\sqrt{v})=\frac{2}{\sqrt{2\pi}}\int_{\lambda\sqrt{v}}^\infty(z-\lambda\sqrt{v})^2\e^{-z^2/2}<\frac{2\e^{-\lambda^2v/2}}{\sqrt{2\pi}}\int_0^\infty y^2\e^{-y^2/2}={\e^{-\lambda^2v/2}}.
$$
Putting the bounds together, we have 
$$
\rho(0,\hat p)>E \log \left(\frac{1+\frac{  |z|}{\lambda\sqrt{v}}}{1+\frac{|z|}{\lambda}}\right)-{\e^{-\lambda^2v/2}}\left(\frac{1}{2}+\log(\lambda\sqrt{2\pi v})+\frac{2}{\lambda\sqrt{2\pi v}}\right)
+E  \log\frac{2}{\sqrt{\frac{4}{(\lambda+|z|)^2}+1}+1}.
$$
We use   Lemma \ref{lemma:aux_function} to find that
$$
E  \log\frac{\sqrt{\frac{4}{(\lambda +|z|)^2}+1}+1}{2}<\log\frac{\sqrt{\frac{4}{ \lambda^2 }+1}+1}{2}<\frac{\sqrt{\frac{4}{ \lambda^2 }+1}-1}{2}<\frac{4}{\lambda^2}.
$$
Next,   for any $a>0$
$$
E \log \left(1+\frac{\left(\frac{1}{\sqrt{v}}-1\right)\frac{  |z|}{\lambda}}{1+\frac{|z|}{\lambda}}\right)> P(|z|\geq a)\log\left(1+
{\left(\frac{1}{\sqrt{v}}-1\right)\frac{a}{\lambda+a}}\right).
$$
Altogether, we have for any $a>0$
$$
\rho(0,\hat p)>P(|z|\geq a)\log\left[1+
{\left(\frac{1}{\sqrt{v}}-1\right)\frac{a}{\lambda+a}}\right]-\frac{4}{\lambda^2}-{\e^{-\lambda^2v/2}}\left(\frac{1}{2}+\log(\lambda\sqrt{2\pi v})+ \frac{2}{\lambda\sqrt{2\pi v}}\right).
$$

\section{Proof of Theorem \ref{thm:ssl}}\label{sec:proof_thm:ssl}

Recall the Spike-and-Slab LASSO prior
$$
\pi(\theta)=\eta\pi_1(\theta)+(1-\eta)\pi_0(\theta)
$$
where $\pi_0(\mu)=\lambda_0/2\e^{-\lambda_0|\mu|}$ and $\pi_1(\mu)=\lambda_1/2\e^{-\lambda_1|\mu|}$ for $\lambda_1<\lambda_0$.
We also recall the definitions \eqref{eq:Is} and \eqref{eq:mus} but we now explicitly state their dependence on $\lambda$. We define
$$
I_1^v(\lambda)=\sqrt{v}\frac{\Phi(\mu^v_1(\lambda)/\sqrt{v})}{\phi(\mu^v_1(\lambda)/\sqrt{v})} \quad\text{and}\quad I_2^v(\lambda)=\sqrt{v}\frac{\Phi(-\mu^v_2(\lambda)/\sqrt{v})}{\phi(-\mu^v_2(\lambda)/\sqrt{v})}
$$
where
$$
\mu^v_1(\lambda)=z\sqrt{v}+\theta-v\lambda \quad\text{and}\quad  \mu^v_2(\lambda)=z\sqrt{v}+\theta+v\lambda.
$$
We denote the rescaled ratio of two marginal likelihoods $\frac{\lambda_1}{\lambda_0}\frac{m_0(x)}{m_1(x)}$ for $x=z+\theta$ by
$$
R_{\lambda_0,\lambda_1}(z)=\frac{I_1^1(\lambda_0)+I_2^1(\lambda_0)}{I_1^1(\lambda_1)+I_2^1(\lambda_1)}.
$$
We can use Lemma \ref{lemma:decompose} to decompose the prediction risk
\begin{equation}
\rho(\theta,\hat p)= \rho(\theta,\hat p_1)+ E \log  N^{SS}_{\theta,1}(z)
- E \log N^{SS}_{\theta,v}(z)\label{eq:risk_SSL}
\end{equation}
where
$$
  N_{\theta,v}^{SS}(z)=1+\frac{1-\eta}{\eta}\frac{\lambda_0}{\lambda_1} R_{\lambda_0,\lambda_1}(z).
$$
Alternatively, we could write
\begin{equation}
\rho(\theta,\hat p)= \rho(\theta,\hat p_0)+ E \log  \wt N^{SS}_{\theta,1}(z)
- E \log \wt N^{SS}_{\theta,v}(z)\label{eq:risk_SSL2}
\end{equation}
where
$$
 \wt N_{\theta,v}^{SS}(z)=1+\frac{\eta}{1-\eta}\frac{\lambda_1}{\lambda_0}\frac{1}{ R_{\lambda_0,\lambda_1}(z)}.
$$
We will utilize both of these characterizations. The decomposition \eqref{eq:risk_SSL} is useful when the observed data $Y$ is large in magnitude, because we would expect the slab risk to be the dominant term (according to Lemma \ref{lemma:simple_bound_risk}). The opposite is true when $Y$ is small. While the upper bound on the risk in Lemma \ref{lemma:simple_bound_risk} uses an average mixing weight $\Delta_\eta(Y)$,  we pause with averaging over the distribution $\pi(Y\C\theta)$ and bound the Kullback-Leibler loss in two regimes depending on the magnitude of $Y$. Below, we show that $(n-s_n)\rho(0,\hat p)\lesssim s_n$ and $s_n\sup_{\theta\in\R}\rho(\theta,\hat p)\lesssim (1-v)s_n\log (n/s_n)$ which will conclude the proof of the theorem.
Throughout this section, we denote the observed data $Y$ simply with $x$ and thereby assume $x\sim N(\theta,1)$.

\subsection{The case when $\theta=0$.}\label{sec:SSL_case0}
We utilize the expression \eqref{eq:risk_SSL2} and (using the risk expression at $\theta=0$ for Bayesian LASSO from Section \ref{sec:noise_lasso}) we find that 
\begin{align}
\rho(0,\hat p)\leq \frac{\sqrt{2}}{\lambda_0\sqrt{\pi v}}+\frac{4}{\lambda_0^2v}+ \log\left(1+\frac{\eta}{1-\eta}E_{x\C\theta=0}\frac{m_1(x)}{m_0(x)}\right)\label{eq:upper_bound_SSL_zerp}
\end{align}
In order to bound the expectation $E_{x\C\theta=0}\frac{m_1(x)}{m_0(x)}$, we consider $4$ possible cases.
\begin{itemize}
\item[(1)]When $x=z>\lambda_0>0$, we have $I_1^1(\lambda)>I_2^1(\lambda)$ and because $\mu_1^1(\lambda_0)=z-\lambda_0>0$ (and thereby $\Phi(\mu_1^1(\lambda_0))>1/2$) we have
$$
\frac{m_1(x)}{m_0(x)}=2\frac{\lambda_1}{\lambda_0}\times \frac{I_1^1(\lambda_1)}{I_1^1(\lambda_0)}<4\frac{\lambda_1}{\lambda_0}\times\frac{\phi(\mu_1^1(\lambda_0))}{\phi(\mu_1^1(\lambda_1))}=
4 \frac{\lambda_1}{\lambda_0}\e^{x(\lambda_0-\lambda_1)-\lambda_0^2/2+\lambda_1^2/2}.
$$
Then
$$
\int_{\lambda_0}^\infty \phi(x)\frac{m_1(x)}{m_0(x)}\d x< \frac{4\lambda_1\e^{-\lambda_0^2/2+\lambda_1^2/2+(\lambda_0-\lambda_1)^2/2}}{\lambda_0}\int_{\lambda_0}^\infty\frac{\e^{-[x-(\lambda_0-\lambda_1)]^2/2}}{\sqrt{2\pi}}\d x<
\frac{4\lambda_1}{\lambda_0}\e^{-\lambda_0\lambda_1+\lambda_1^2}.
$$
\item[(2)] When $\lambda_0>x=z>0$, we can use Lemma \ref{lem:mills} (because $\mu_1^1(\lambda_0)=z-\lambda_0<0$) to find
$$
\frac{m_1(x)}{m_0(x)}<2\frac{\lambda_1}{\lambda_0}\times \frac{1}{\phi(\mu_1^1(\lambda_1))}\times \frac{ \sqrt{(\lambda_0-x)^2+4}+\lambda_0-x }{2}.
$$
This yields
$$
\frac{m_1(x)}{m_0(x)}<2\frac{\lambda_1}{\lambda_0}\times \sqrt{2\pi}\e^{(x-\lambda_1)^2/2}\times \frac{(1+\sqrt{2}) \max\{2,(\lambda_0-x)\}}{2}
$$
and
$$
\int_{0}^{\lambda_0} \phi(x)\frac{m_1(x)}{m_0(x)}\d x < \frac{ (1+\sqrt{2}) \max\{2,\lambda_0\}}{\lambda_0}\e^{\lambda_1^2/2}(1-\e^{-\lambda_0\lambda_1}).
$$
The next two cases are mirror images of the previous too.
\item[(3)] When $-\lambda_0<x=z\leq 0$ we have $I_1^1(\lambda)\leq I_2^1(\lambda)$ and (using Lemma \ref{lem:mills})
$$
\frac{m_1(x)}{m_0(x)}=2\frac{\lambda_1}{\lambda_0}\times \frac{1}{\phi(\mu_2^1(\lambda_1))}\times \frac{ \sqrt{(\lambda_0+x)^2+4}+\lambda_0+x }{2}.
$$
This yields
$$
\frac{m_1(x)}{m_0(x)}=2\frac{\lambda_1}{\lambda_0}\times \sqrt{2\pi}\e^{(x+\lambda_1)^2/2}\times \frac{(1+\sqrt{2}) \max\{2,(\lambda_0+x)\}}{2}
$$
and
$$
\int_{-\lambda_0}^{0} \phi(x)\frac{m_1(x)}{m_0(x)}\d x\leq \frac{ (1+\sqrt{2}) \max\{2,\lambda_0\}}{\lambda_0}\e^{\lambda_1^2/2}(1-\e^{-\lambda_0\lambda_1}).
$$
\item[(4)] When  $x=z\leq -\lambda_0<-\lambda_1$ we have $-\mu_2^1(\lambda_0)=-z-\lambda_0>0$ and thereby
\begin{align*}
\frac{m_1(x)}{m_0(x)}&=4\frac{\lambda_1}{\lambda_0} \times \frac{\phi(\mu_2^1(\lambda_0))}{\phi(\mu_2^1(\lambda_1))} = 4 \frac{\lambda_1}{\lambda_0}\e^{-x(\lambda_0-\lambda_1)-\lambda_0^2/2+\lambda_1^2/2}.
\end{align*}
and
$$
\int_{-\infty}^{-\lambda_0} \phi(x)\frac{m_1(x)}{m_0(x)}\d x< \frac{4\lambda_1\e^{-\lambda_0^2/2+\lambda_1^2/2+(\lambda_0-\lambda_1)^2/2}}{\lambda_0}\int_{-\infty}^{-\lambda_0} \frac{\e^{-[z+(\lambda_0-\lambda_1)]^2/2}}{\sqrt{2\pi}}\d z=
\frac{4\lambda_1}{\lambda_0}\e^{-\lambda_0\lambda_1+\lambda_1^2}.
$$
 \end{itemize}

Putting all the pieces together, we find that  keeping $\lambda_1=\sqrt{v}C_v$ for some $C_v>0$  we have for some $C_1>0$
$$
E_{x\C\theta=0}\frac{m_1(x)}{m_0(x)}<C_1.
$$
With  $\eta/(1-\eta)=n/s_n$  and $\lambda_0\sqrt{v} =n/s_n$  we find that 
$\rho(0,\hat p)\lesssim s_n/n$.
\subsection{The case when $\theta\neq 0$}

We consider two cases (similarly as in the proof of Theorem \ref{thm:dirac_laplace}) for some $A>0$:

{\bf Case i}: On the event $A_\eta (\theta,A,d)^c\equiv\{z:  R_{\lambda_0,\lambda_1}(z)> A(s_n/n)^{d}\}$ we have
$$
\log \wt N^{SS}_{\theta,v}(z)\leq \log \left[1+\left(\frac{\eta}{1-\eta}\right)\frac{\lambda_1}{\lambda_0}(n/s_n)^{d}/A \right].
$$

{\bf Case ii}: On the event  $A_\eta(\theta,A,d)\equiv\{z:   R_{\lambda_0,\lambda_1}(z)\leq A(s_n/n)^{d} \}$ we have
$$
\log   N^{SS}_{\theta,v}(z)\leq \log \left[1+\left(\frac{1-\eta}{\eta}\right)\frac{\lambda_0}{\lambda_1}A(n/s_n)^{-d}\right].
$$

Recall  (from \eqref{eq:KL_lasso})  that under the Laplace prior with a parameter $\lambda$, the KL loss equals (using the expression $x+y/r=z/\sqrt{v}+\theta/v$)
\begin{equation}
K_\lambda(\theta,z)=\theta^2/(2r)+ \log[\lambda/2(I_1^1(\lambda)+I_2^1(\lambda))]- E_{y\C\theta} \log[\lambda/2(I_1^v(\lambda)+I_2^v(\lambda))].\label{eq:KL_lasso}
\end{equation}
This means 
\begin{align}
\rho(\theta,\hat p)\leq& P[A_\eta(\theta,A,d)^c] \log \left[1+\left(\frac{\eta}{1-\eta}\right)\frac{\lambda_1}{  \lambda_0}(n/s_n)^{ d}/A\right] + \int_{z\in A_\eta^c (\theta,A,d)} K_{\lambda_0}(\theta,z)\phi(z)\d z\nonumber \\
&+ P[A_\eta (\theta,A,d)]\log \left[1+\left(\frac{1-\eta}{\eta}\right)\frac{  \lambda_0}{\lambda_1}A(n/s_n)^{-d}\right]+
\int_{z\in A_\eta(\theta,A,d)}K_{\lambda_1}(\theta, z)\phi(z)\d z\label{eq:essential}.
\end{align}
Now we focus on the properties of the set $A_\eta^c (\theta,A,d)$.
\begin {lemma}\label{lem:emptyset}
For $\lambda_0/\lambda_1=n/s_n$ and $\lambda_1>0$ is a fixed constant. When $0<r<1$, there exists $A>0$ such that $A_\eta(\theta,A,2v)^c=\emptyset$. When $r\in [1,\infty)$, we have
$$
A_{\eta}(\theta,A,2v)^c\subset \{x: |x|<\wt x_c\},
$$
where $\wt x_c=\lambda_1+\sqrt{4v\log [n/s_n]}$ and  $A=[\sqrt{2\pi}(1+\sqrt{2})]^{-1}$. 
\end{lemma}
\begin{proof}
The ratio $R_{\lambda_0,\lambda_1}(z)$ as a function of $x=z+\theta$ has a maximal value at $x=0$ where (using the Mills ratio notation $R(x)=(1-\Phi(x))/\phi(x)$ and Lemma \ref{lem:mills})
$$
  R_{\lambda_0,\lambda_1}(z)\leq R_{\lambda_0,\lambda_1}(-\theta)= 
\frac{R(\lambda_0)}{R(\lambda_1)}<\frac{\lambda_1}{\lambda_0}\frac{\sqrt{1+4/\lambda_1^2}+1}{\sqrt{1+2/\lambda_0^2}+1}<
\frac{s_n}{n}\frac{\sqrt{1+4/\lambda_1^2}+1}{2}.
$$
This means that for $A=\frac{\sqrt{1+4/\lambda_1^2}+1}{2}$ we have $ R_{\lambda_0,\lambda_1}(z)< As_n/n$ which is strictly smaller than $A(s_n/n)^{2v}$ when $0<v<1/2$ (i.e. when $0<r<1$). Now we look into the case when $d=2v\geq 1$. Because the ratio $m_0(x)/m_1(x)$ is symmetrical around zero and monotone decreasing on $(0,\infty)$, we have
$$
A_{\eta}(\theta,A,d)^c=\left\{x=z+\theta:\frac{m_0(x)}{m_1(x)}>\frac{\lambda_0}{\lambda_1} A(s_n/n)^d \right\}=\{x: |x|\leq x_c\},
$$
where
$$
\frac{m_0(x_c)}{m_1(x_c)}=\lambda_0/\lambda_1 A(s_n/n)^d.
$$
We now find an upper bound to $\frac{m_0(x)}{m_1(x)}$. We first consider the case when $\lambda_1<x<\lambda_0-2$. 
On this interval (similarly as in the case (2) in Section \ref{sec:SSL_case0}) we obtain  
$$
\frac{m_0(x)}{m_1(x)}\leq \frac{\lambda_0}{\lambda_1 }\frac{2 \phi(\mu_1^1(\lambda_1))}{\sqrt{(\lambda_0-x)^2+2}+\lambda_0-x}\leq
  \frac{\lambda_0}{\lambda_1 }
  \frac{2\phi(\mu_1^1(\lambda_1))}{\min\{2,\lambda_0-x\}(1+\sqrt{2})}.
$$
Setting this upper bound equal to $A(s_n/n)^d$ yields
$$
A_{\eta}(\theta,A,d)^c\subset \{x: |x|<\wt x_c\},
$$
where $\wt x_c=\lambda_1+\sqrt{2d\log [n/s_n]}$ and  $A=[\sqrt{2\pi}(1+\sqrt{2})]^{-1}$. 
When $0<x<\lambda_1$, we have 
$$
\frac{m_0(x)}{m_1(x)}> \frac{\lambda_0}{2\lambda_1 }\frac{\sqrt{(\lambda_1-x)^2+2}+\lambda_1-x}{\sqrt{(\lambda_0-x)^2+4}+\lambda_0-x}.
$$ 
for $\lambda_0=n/s_n$. For $d\geq 1$ and suitable $A>0$ we will have $\{-\lambda_1,\lambda_1\}\subset A_{\eta}(\theta,A,d)^c$. Because $\wt x_c<\lambda_0$ when $\lambda_0=n/s_n$, we conclude that $A_{\eta}(\theta,A,d)^c\subset \{x: |x|<\wt x_c\}.$
\end{proof} 
Now we continue with the proof of Theorem {\ref{thm:ssl}}.
\subsubsection{When $r\in (0,1)$}
 Going back to \eqref{eq:essential}, we find  that for $A=\frac{\sqrt{1+4/\lambda_1^2}+1}{2}$ and $\lambda_1=\sqrt{v}C_v$ for some $C_v>0$ and $\lambda_0\sqrt{v}=n/s_n$ and $\eta/(1-\eta)=1$ we have for $d=2v$
$$
\rho(\theta,\hat p)\leq \log \left[1+A (n/s_n)^{2- 2v}/C_v\right]+\rho(\theta,\hat p_1),
$$
where $\hat p_1$ is the predictive distribution under the slab Laplace prior.  From the proof of Theorem  \ref{lemma:risk_lasso_upper} in Section \ref{sec:upper_lasso_bound}, we know that
$$
\rho(\theta,\hat p)\leq \log \left[1+ A(n/s_n)^{2- 2v}/C_v \right]+ \log(\lambda_14\sqrt{2\pi/v})+\lambda_1^2/2+\lambda_1\sqrt{2/\pi}+4/\lambda_1^2.
$$
 Keeping $\lambda_1=\sqrt{v}C_v$ for some $C_v>0$, the first term is the dominant term  and $\rho(\theta,\hat p)\lesssim (1-v)\log(n/s_n)$.

 \subsubsection{When $r\in [1,\infty)$}

Using \eqref{eq:essential}, we need to make sure that the event $A_{\eta}(\theta,A,d)^c$ for $d=2v$ is small enough when the signal (and  $|x|$) is large so that the spike predictive distribution gets silenced.  We have from Lemma \ref{lem:emptyset}  when $|\theta|>c\sqrt{\log (n/s_n)}$ for some $c > 2d$ (because  $|z+\theta|>|\theta|-|z|$)
$$
P(A_{\eta}(\theta,A,d)^c)\leq P(|z|>|\theta|-\wt x_c)\leq 2 \e^{-[(c-\sqrt{2d})\sqrt{\log (n/s_n)}-\lambda_1]^2/2}
\leq 2\e^{\lambda_1^2/4}\e^{-(c-\sqrt{2d})^2\log (n/s_n)}.
$$
With  $(c-\sqrt{2d})^2\geq2$ we have  $P(A_{\eta}(\theta,A,d)^c)\lesssim (s_n/n)^2$.  Using similar steps as in the proof of Theorem \ref{lemma:risk_lasso_upper} in Section \ref{sec:upper_lasso_bound} we find
\begin{align}
K_{\lambda }(z)&\leq \frac{\theta^2}{2r}+\log(2/\sqrt{v})+\frac{(|z+\theta|-\lambda )^2}{2}-\frac{(|z+\theta/\sqrt{v}|-\lambda \sqrt{v})^2}{2}-\log\Phi(-\lambda \sqrt{v})\nonumber\\
& \leq \frac{\theta^2}{2r}+\log(\lambda4\sqrt{2\pi/v})+\frac{(|z+\theta|-\lambda )^2}{2}-\frac{(|z+\theta/\sqrt{v}|-\lambda \sqrt{v})^2}{2} +\lambda^2v/2+ \frac{1}{\lambda^2}.\label{eq:K}
\end{align}
On the set $A_{\eta}(\theta,A,d)^c$  we have $|z+\theta|\leq \wt x_c$ and 
$$
K_{\lambda_0 }(z)\leq  \frac{\theta^2}{2r}+\log(\lambda_04\sqrt{2\pi/v})+{ \wt x_c^2}+(1+v)\lambda_0^2/2 + \frac{1}{\lambda_0^2}.
$$
With $\lambda_0=n/s_n$, the dominant term among the last three terms above is $\lambda_0^2(1+v)$. Above, we have shown that when the signal is strong enough we have $P(A_{\eta}(\theta,A,d)^c)\lesssim 1/\lambda_0^2=(s_n/n)^2$. This means
$$
\int_{A_{\eta}(\theta,A,d)^c} K_{\lambda_0 }(z)\phi(z)\d z\lesssim \theta^2/(2r)P(A_{\eta}(\theta,A,d)^c)+O(1).
$$
Combined with $\eqref{eq:essential}$ the term $\frac{\theta^2}{2r}P(A_{\eta}(\theta,A,d)^c$ can be combined with the term $\frac{\theta^2}{2r}P(A_{\eta}(\theta,A,d))$ contained inside the slab part. Altogether, we obtain an upper bound that is of the order $(1-v)\log (n/s_n)$.

\section {Proofs of Lemmata}

\subsection{Proof of Lemma \ref{lemma:lasso_risk_decompose}}\label{sec:lemma:lasso_risk_decompose}
We have the following risk definition
\begin{equation}
\rho(\theta,\hat p)=\int \pi(Y\C\theta)\int \pi(\wt Y\C\theta )\log [\pi(\wt Y\C\theta)/\hat p(\wt Y\C Y)]\d \wt Y\d Y\label{eq:rho_lasso}.
\end{equation}
For the marginal likelihood $m(Y)=\int \pi(Y\C\mu)\pi_1(\mu\C\lambda)\d\mu$, we have
$$
\hat p(\wt Y\C Y)=\frac{\int \pi(\wt Y\C\mu)\pi(\mu\C Y)\d\mu}{m(Y)}=\frac{\e^{-\wt Y^2/(2r)}\int\exp \{\mu(\wt Y/r+Y)-\mu^2/2(1+1/v)\}\pi_1(\mu\C\lambda)\d\mu}{\int \exp\{\mu Y-\mu^2/2\}\pi_1(\mu\C\lambda)\d\mu }
$$
and
\begin{equation}
\frac{\pi(\wt Y\C\theta)}{\hat p(\wt Y\C Y)}=\frac{\exp(\wt Y\theta/r-\theta^2/2r)\int \exp\{\mu Y-\mu^2/2\}\pi_1(\mu\C\lambda)\d\mu }{
\int\exp \{\mu(\wt Y/r+Y)-\mu^2/2(1+1/r)\}\pi_1(\mu\C\lambda)\d\mu}.\label{eq:ratio_inside_KL}
\end{equation}
Then
\begin{align}
\log\frac{\pi(\wt Y\C\theta)}{\hat p(\wt Y\C Y)}=& \wt Y\theta/r-\theta^2/2r+ \log \int \exp\{\mu Y-\mu^2/2\}\pi_1(\mu\C\lambda)\d\mu\nonumber\\
&-
\log \int\exp \{\mu(\wt Y/r+Y)-\mu^2/2(1+1/r)\}\pi_1(\mu\C\lambda)\d\mu.\label{eq:KL_lasso}
\end{align}
The expectation of the first term with respect to  $\wt Y\sim N(\theta,r)$ is $\theta^2/(2r)$.
Since $Y\sim N(\theta,1)$ and $\wt Y\sim N(\theta,r)$, the expectation of the other two terms can be taken with respect to $Y+\wt Y/r\sim N(\theta/v, 1/v)$ where
$v=1/(1+1/r)$. This is the same as taking an expectation with respect to $Z/\sqrt{v}+\theta/v$. This concludes the proof.

\subsection{Proof of Lemma \ref{lemma:decompose}}\label{sec:proof_lemma:decompose}
The Kullback-Leibler loss can be written as
\begin{align*}
L(\theta,\hat p)&=\int \pi(\wt Y\C \theta)\log\frac{\pi(\wt Y\C\theta)}{\Delta_\eta(Y)\hat p_1(\wt Y\C Y)+ (1-\Delta_\eta(Y))\hat p_0(\wt Y\C Y)}\d \wt Y\\
&=L(\theta,\hat p_1)-\log[\Delta_\eta(Y)]-\int \pi(\wt Y\C\theta)\log\left[1+\frac{1-\Delta_\eta(Y)}{\Delta_\eta(Y)}\frac{\hat p_0(\wt Y\C Y)}{\hat p_1(\wt Y\C Y)}\right]\d \wt Y.
\end{align*}
Next, note that 
$$
-\log\left[{\Delta_\eta(Y)}\right]=\log\left(1+\frac{(1-\eta)}{\eta} \frac{m_0(Y)}{m_1(Y)}\right)
$$
and 
$$
\frac{1-\Delta_\eta(Y)}{\Delta_\eta(Y)}\frac{\hat p_0(\wt Y\C Y)}{\hat p_1(\wt Y\C Y)}=\frac{(1-\eta)}{\eta} \frac{m_0(Y)}{m_1(Y)}\frac{\hat p_0(\wt Y\C Y)}{\hat p_1(\wt Y\C Y)}.
$$
Next,
$$
\frac{m_0(Y)}{m_1(Y)}\frac{\hat p_0(\wt Y\C Y)}{\hat p_1(\wt Y\C Y)}=\frac{\int \exp(\mu(Y+\wt Y/r)-\mu^2/2(1+1/r))\pi_0(\mu)\d\mu }{
\int \exp(\mu(Y+\wt Y/r)-\mu^2/2(1+1/r))\pi_1(\mu)\d\mu}
$$
To obtain $\rho(\theta,\hat p)$ we take an expectation over $Y\sim N(\theta,1)$ since $Y+\wt Y/r\sim N(\theta(1+1/r), 1+1/r)$
which is the same as taking an expectation with respect to $\theta/v+Z/\sqrt{v}$ for $Z\sim N(0,1)$ and $v=1/(1+1/r)$.
Noting that
$
-E_{Y\C\theta}\log\Delta_\eta(Y)=E_z \log    N_{\theta,1}^{SS}(z)
$
we obtain the desired expression.

\subsection{Proof of Lemma \ref{lemma:risk_hierarchical}}\label{sec:proof_lemma:risk_hierarchical}
Because, given $\eta$, the Kullback-Leibler loss is separable  
$$
L(\theta ,\hat p(\cdot\C   Y,\eta))=\sum_{i=1}^nL(\theta_i ,\hat p(\cdot\C  Y_i,\eta)),
$$
taking the expectation of both sides of \eqref{eq:bound_KL_sep} over the distribution $\pi( Y\C\theta)$ yields
\begin{equation}
\rho(\theta,\hat p)\leq   (n-s_n)E_{ Y\C\theta} E_{\eta\C  Y} L(0,\hat p(\cdot\C  Y_i,\eta))   +   E_{ Y\C\theta} E_{\eta\C  Y}  \sum_{i:\theta_i\neq 0}  L(\theta_i,\hat p(\cdot\C  Y_i,\eta)).\label{eq:risk_bound_double}
\end{equation}
We use a similar expression as in the proof of Lemma   \ref{lemma:decompose} in Section \ref{sec:proof_lemma:decompose}. We first focus on the case when $\theta_i\neq 0$.
From now on, we will be using simpler notation $x=Y$ and $y=\wt Y$. We will also denote with $\theta$ the unknown true parameter value and with $\mu$ the random vector used to estimate $\theta$.
$$
\wt g(x,y,\theta,\eta)=\log\frac{\pi(y\C\theta)}{\hat p(y\C x)}
$$
Then
\begin{equation}
\wt g(x,y,\theta,\eta)=\log \frac{\phi((y-\theta)/\sqrt{r})}{\phi(y/\sqrt{r})}-\log[1-\Delta_\eta(x)]-\log \left[1+\frac{\Delta_\eta(x)}{1-\Delta_\eta(x)}\frac{\hat p_1(y\C x)}{\hat p_0(y\C x)}\right].\label{eq:gtilde}
\end{equation}
which means  
$$
L(\theta_i ,\hat p(\cdot\C  x_i,\eta)) =E_{y_i\C\theta_i}\wt g(x_i,y_i,\theta_i,\eta).
$$ 
We also denote
$$
\frac{1}{1-{\Delta_\eta(x)}}=1+\frac{\eta}{ 1-\eta } \frac{m_1(x)}{m_0(x)}=1+\frac{\lambda}{2}\frac{\eta}{1-\eta}\wt R_1(x)
$$
and 
$$
1+\frac{\Delta_\eta(x)}{1-\Delta_\eta(x)}\frac{\hat p_1(y\C x)}{\hat p_0(y\C x)}=1+\frac{\lambda}{2}\frac{\eta}{1-\eta}\wt R_v(x,y)
$$
where $\wt R_1(x)=\wt R_{v=1}(x,y)$ where for $v=1/(1+1/r)$
$$
\wt R_{v}(x,y)=\int \exp\left[\mu\left(x+y\frac{1-v}{v}\right)-\frac{\mu^2}{2v}-\lambda|\mu|\right]\d\mu.
$$
Similarly as in the proof of Theorem \ref{thm:dirac_laplace}, we consider two cases (i) and (ii) depending on the magnitude $|x|$. Note that 
unlike in the proof of Theorem \ref{thm:dirac_laplace}, here we are explicitly using notation involving $z$ and $y$ as opposed to $z$.  
Section \ref{sec:casei} and \ref{sec:caseii} show upper bounds on the risk which separate parameter $\eta$ from $z$ (and inherently also $x$ and $y$).
We can use the same ideas as in Section \ref{sec:casei}, \ref{sec:caseii} and \ref{sec:cases_combined} to find that  
$$
\log\frac{1+\frac{\lambda}{2}\frac{\eta}{1-\eta}\wt R_1(x)}{1+\frac{\lambda}{2}\frac{\eta}{1-\eta}\wt R_v(x,y)}<C(\lambda,v,x,y)+5(1-v)\log\left(\frac{1-\eta}{\eta}\right).
$$
Then because the term $C(\lambda,v,x,y)$ does not depend on $\eta$, we can write
$$
 E_{ Y\C\theta} E_{\eta\C  Y}E_{y_i\C\theta_i}\wt g(x_i,y_i,\theta_i,\eta)\leq  E_z  \wt C(\lambda,v,z)+     5(1-v) E_{\eta\C  Y}\log\left(\frac{1-\eta}{\eta}\right)
$$
where the term $\wt C(\lambda,v,z)$ contains aspects of the bounds \eqref{eq:bound1} and \eqref{eq:bound2} that do not involve $\eta$.
After taking the expectation over $Z\sim N(0,1)$ we arrive at a version of the  \eqref{eq:bound_total}, only with $\log[(1-\eta)/\eta]$ replaced by its conditional expectation.
It also follows from \eqref{eq:bound_total} that
\begin{equation}
C(\lambda,v)\equiv E_z \wt C(\lambda,v,z)= \log \left(4\sqrt{\frac{2\pi}{v}}\right)  +\frac{1}{\lambda^2 } + \frac{\lambda^2 }{2}+\frac{2\lambda^2 +1}{r} +\lambda (1-\sqrt{v})\sqrt{\frac{2}{\pi}} +\log(2\lambda+\lambda^2+1).
.\label{eq:Clambda}
\end{equation}
When $\theta_i=0$, we have from \eqref{eq:gtilde}
$$
\wt g(x,y,\theta,\eta)<-\log[1-\Delta_\eta(x)]
$$
and using Jensen's inequality $E\log X<\log EX$ we find that
$$
 E_{ Y\C\theta}E_{\eta\C Y} \wt g(x,y,\theta,\eta)<   \log\left[1+E_{ Y\C\theta}E\left(\frac{\eta}{1-\eta}\C Y\right)\frac{m_1(x_i)}{\phi(x)} \right]. 
$$
The conditional expectation $E\left(\frac{\eta}{1-\eta}\C Y\right)$ will be (for large $n$) very similar to the conditional expectation $E\left(\frac{\eta}{1-\eta}\C Y_{\backslash i}\right)$,
where $ Y_{\backslash i}$ denotes the sub-vector of $ Y$ after excluding the $i^{th}$ coordinate. For $\lambda>2$, we can use the sandwich relation \eqref{eq:sandwich} to find that
$$
\frac{E\left(\frac{\eta}{1-\eta}\C Y\right)}{E\left(\frac{\eta}{1-\eta}\C Y_{\backslash i}\right)}\leq 
\frac{a+E[s_n(\mu)\C  Y]+1}{a+E[s_n(\mu_{\backslash i})\C  Y_{\backslash i}]}\leq \frac{a+E[s_n(\mu)\C  Y]+1}{a+E[s_n(\mu)\C  Y]-1}<1+\frac{2}{a-1}.
$$
Then we can write
$$
 E_{ Y\C\theta}E\left(\frac{\eta}{1-\eta}\C Y\right)\frac{m_1(x_i)}{\phi(x_i)}\leq  E_{ Y_{\backslash i}\C\theta}E\left(\frac{\eta}{1-\eta}\C Y_{\backslash i}\right)\left(1+\frac{2}{a-1}\right)\int m_1(x_i) \d x_i.
$$
The desired statement is obtained after noting that  (because $E\e^{c x}=\e^{-c^2/2}$ for $X\sim N(0,1)$)
$$
\int m_1(x_i) \d x_i=\int_\mu {\e^{-\mu^2/2}}E\e^{\mu x_i}\pi_1(\mu)\d\mu=1.
$$
 
\subsection{Proof of Lemma \ref{lemma:posterior_odds} }\label{sec:proof_lemma:posterior_odds}

Since the parameter $\eta$ is hierarchically separated from the data by $\mu\sim \pi(\mu)$, we can write
$$
\pi(\eta \C  Y)={\int\pi(\eta\C\mu)\pi(\mu\C  Y)}\d\mu.
$$
Then
$$
E\left(\frac{\eta}{1-\eta}\C Y\right)={\int_\eta \frac{\eta}{1-\eta}\,\pi(\eta)\int_{\mu}\frac{\pi(\mu\C\eta)\pi(\mu\C  Y)}{\pi(\mu)}\d\mu\d\eta}.
$$
 We will now work conditionally on $\mu$. For $\mu=(\mu_1,\dots,\mu_n)'$ and $\pi_1(\mu)=\prod_{i=1}^n\pi_1(\mu_i)=(\lambda/2)^n\e^{-\lambda|\mu|_1}$, we write
$$
\pi(\mu\C\eta)=\eta^n\, \pi_1(\mu)\prod_{i=1}^n\left[1+\frac{1-\eta}{\eta}\frac{\pi_0(\mu_i)}{\pi_1(\mu_i)}\right]\,\,\text{and}\,\,\pi(\bm \mu)=\int_\eta\pi(\mu\C\eta)\pi(\eta).
$$
For $S_n(\mu)=\{i: \mu_i\neq0\}$ with size $s_n(\mu)=|S_n(\mu)|$ we can write for the point-mass spike $\pi_0(\mu)=\delta_0(\mu)$
$$
\pi(\mu\C\eta)=\eta^{n} \pi_1(\mu)\prod_{i=1}^n\left[1+\frac{1-\eta}{\eta}\frac{\pi_0(\mu_i)}{\pi_1(\mu_i)}\right]=
\eta^n\pi_1(\mu)\left[1+\frac{1-\eta}{\eta}\frac{1}{\pi_1(0)}\right]^{n-s_n(\mu)}.
$$
Then
\begin{equation}
E\left(\frac{\eta}{1-\eta}\C Y\right)=E_{ \theta\C Y}E\left(\frac{\eta}{1-\eta}\C \theta\right)= \int_{\mu}\pi(\mu\C Y)\frac{\int_\eta \frac{\eta^{n+1}}{1-\eta}\pi(\eta)
\left[1+\frac{1-\eta}{\eta}\frac{2}{\lambda}\right]^{n-s_n(\mu)}}{\int_{\eta}\eta^n\pi(\eta)\left[1+\frac{1-\eta}{\eta}\frac{2}{\lambda}\right]^{n-s_n(\mu)}}.\label{eq:odds1}
\end{equation}
Next, we represent the posterior expectation above as a functional of a  ratio of two Gauss hypergeometric functions.
An Euler representation of the Gauss hypergeometric function writes as 
\begin{equation}
F_2(a',b',c';z)=\frac{\Gamma(c')}{\Gamma(b')\Gamma(c'-b')}\int_0^1\eta^{b'-1}(1-\eta)^{c'-b'-1}(1-\eta z)^{-a'}\d\eta \label{eq:gaus}.
\end{equation}
We now bound the posterior expectation of the prior odds $\eta/(1-\eta)$, given $\mu$. Using simple notation $s=s_n(\mu)$ and $z=\lambda/2-1$, we can write
\begin{align*}
E\left(\frac{\eta}{1-\eta}\C\mu\right)&=\frac{\int_\eta \eta^{a+s}(1-\eta)^{b-2}
\left[1+\eta\left(\frac{\lambda}{2}-1\right)\right]^{n-s}}{\int_{\eta} \eta^{a-1+s}(1-\eta)^{b-1}
\left[1+\eta\left(\frac{\lambda}{2}-1\right)\right]^{n-s}}\\
&=\frac{\Gamma(b-1)}{\Gamma(b)} \frac{\Gamma(a+s+1) }{\Gamma(a+s) }\frac{F_2(s-n,a+s+1,b+a+s;-z)}{F_2(s-n,a+s,b+a+s;-z)}.
\end{align*}
In Lemma \ref{lemma:gauss_ratio} we show that the ratio of two Gauss hypergeometric functions with a shifted second parameter is monotone increasing in $z$. This result is an extension of Theorem 1 in \cite{karp} who considered a different integer shift on the second and third argument. We use the notation $f_\delta(a',b',c';z)$  
in \eqref{eq:ratio_gauss} and use 
Lemma \ref{lemma:gauss_ratio} 
to find that (when $z=\lambda/2-1<0$)
\begin{align*}
E\left(\frac{\eta}{1-\eta}\C\mu\right)&\leq \frac{a+s+1}{b} f_1(s-n,a+s,b+a+s;0)\\
&=\frac{B(a+s+1,b-1)}{B(a+s,b-1)}\frac{B(a+s,b-1)}{B(a+s,b)}=\frac{a+s}{b-1}<\frac{a+s+1}{b-1}.
\end{align*}
In the regime when $0<\lambda/2-1$ with $\lambda\rightarrow\infty$, we have (by a repeated application of the l'Hospital rule)
$\lim_{z\rightarrow\infty}f_1(s-n,a+s,b+a+s;-z)\rightarrow 1$ and
\begin{equation}
\frac{a+s}{b-1}<E\left(\frac{\eta}{1-\eta}\C\mu\right)\leq \frac{a+s+1}{b}<\frac{a+s+1}{b-1}\label{eq:sandwich}
\end{equation}
Similarly, we have
\begin{align*}
E\left(\frac{1-\eta}{\eta}\C\mu\right)&=\frac{\int_\eta \eta^{a-2+s}(1-\eta)^{b}
\left[1+\eta\left(\frac{\lambda}{2}-1\right)\right]^{n-s}}{\int_{\eta} \eta^{a-1+s}(1-\eta)^{b-1}
\left[1+\eta\left(\frac{\lambda}{2}-1\right)\right]^{n-s}}\\
&=\frac{\Gamma(a+s-1)\Gamma(b+1)}{\Gamma(a+s)\Gamma(b)} \frac{F_2(s-n,a+s-1,b+a+s;-z)}{F_2(s-n,a+s,b+a+s;-z)}.
\end{align*}
Applying Lemma \ref{lemma:gauss_ratio} we find that the ratio above is monotone {\em decreasing} in $z=\lambda/2-1$ for $z>-1$.
Thereby
\begin{align*}
E\left(\frac{1-\eta}{\eta}\C\mu\right)&\leq \frac{\Gamma(a+s-1)\Gamma(b+1)}{\Gamma(a+s)\Gamma(b)} \frac{1}{f_1(s-n,a+s-1,b+a+s,-1)}\\
&=\frac{B(a+s-1,b+n-s+1)}{B(a+s-1,b+n-s)}\frac{B(a+s-1,b+n-s)}{B(a+s,b+n-s)}=\frac{ b+n-s }{ a+s-1 }.
\end{align*}
\qedhere

\subsection{Proof of Lemma \ref{lem:overshoot}}\label{sec:proof_lem:overshoot}

We start by noting that the prior $Beta(2,n+1)$ has a strict exponential decrease property (property (2.2) in \cite{castillo_vdv} $\pi(s_n(\mu)=s)\leq \pi(s_n(\mu)=s-1)C_{\pi}$ for some $0<C_{\pi}<1$ which can be shown as in their Example 2.2). We can thereby apply  Theorem 2.1 in \cite{castillo_vdv}. 
Denoting $B_n(M)=\{\mu\in\R^n:   \|\mu\|_0\leq M\,s_n\}$,   this theorem yields that for some suitable $M>0$, the posterior concentrates on sparse vectors in the sense that 
$P_{ Y\C\theta}\Pi(B_n^c\C  Y)=o(1)$ for any $\theta \in\Theta_n(s_n)$. Lemma \ref{lemma:posterior_odds} then yields  that for $a=2$ and $b=n+1$
$$
\sup_{\theta\in \Theta_n(s_n)}E_{ Y\C \theta}E\left(\frac{\eta}{1-\eta}\C Y\right)\leq \frac{3+M\,s_n}{n}+(3/n+1) P_{ Y\C\theta}\Pi(B_n^c\C  Y)=
Ms_n/n+o(1).\qedhere
$$

\subsection{Proof of Lemma \ref{lem:no_miss}}\label{sec:proof_lem:no_miss}
We rewrite the model $ Y\sim N(\theta , I)$ as
$$
Y_i=\theta_i +\epsilon_i \quad \text{for}\quad \epsilon_i\iid N(0,1) \quad (1\leq i\leq n)
$$
and define an event
$$
\mathcal A_n=\left\{\epsilon_i:\max\limits_{1\leq i\leq n}|\epsilon_i|\leq 2\sqrt{ \log n}\right\}.
$$
This event has a large probability in the sense that $P[\mathcal A_n^c]\leq 2/n$ (see Lemma 4 in \cite{castillo_vdv2}).
Then for $\theta\in \Theta_n(s_n,\wt M)$, take any $j$ such that $\theta_j\neq 0$. We will denote with $S$ a variable indexing all possible subsets of $\{1,\dots, n\}$. Denote with $\mathcal S_j=\{S\subseteq\{1,\dots, n\}: j\notin S\}$ the set of subsets of $\{1,\dots, n\}$ that do not include $j$.
Take $S\in \mathcal S_j$  and denote with $S^+=S\cup\{j\}$ an enlarged ``model"  obtained by including $j$. Then
$$
\frac{\Pi(S\C Y)}{\Pi(S^+\C Y)}=\frac{\Pi(S)}{\Pi(S^+)}\frac{\pi( Y\C S)}{\pi( Y\C S^+)}.
$$
The one-dimensional marginal likelihood under the Laplace prior satisfies
$$
m_1(x)=\phi(x)\lambda/2\left[I_1(x )+I_2(x )\right],
$$
where
$$
I_1(Y_i)=\frac{\Phi(Y_i-\lambda)}{\phi(Y_i-\lambda)}\quad\text{and}\quad  I_2(Y_i)=\frac{\Phi(-(Y_i+\lambda))}{\phi(Y_i+\lambda)}.
$$
This yields that, given $S$, the marginal likelihood is an independent product
$$
\pi( Y\C S) =  \prod_{i=1}^n\phi(Y_i)\prod_{i\in S}\frac{\lambda}{2}\left[I_1(Y_i)+I_2(Y_i)\right].
$$
 and 
$$
\frac{\pi( Y\C S)}{\pi( Y\C S^+)}=\frac{2}{\lambda}\frac{1}{I_1(Y_j)+I_2(Y_j)}.
$$
When $Y_j>0$ we have $I_1(Y_j)>I_2(Y_j)$ and $I_1(Y_j)\leq I_2(Y_j)$ when $Y_j\leq 0$. This yields
$$
\frac{\pi( Y\C S)}{\pi( Y\C S^+)}\leq \frac{2}{\lambda}\frac{\e^{-(|Y_j|-\lambda)^2/2}}{\Phi(-\lambda)}.
$$
Because  of Lemma \ref{lem:mills} and Lemma \ref{lemma:aux_function} we have for $\lambda>1$
$$
\frac{1}{\Phi(-\lambda)}<\frac{\left(\sqrt{1+4/\lambda^2}+1\right)\phi(\lambda)\lambda}{2}< 2 \phi(\lambda)\lambda.
$$
On the event $\mathcal A_n$ we obtain (using the inequality $(a+b)^2>a^2/2-b^2$) 
$$
(|Y_j|-\lambda)^2=(|\theta_j+\epsilon_j|-\lambda)^2>|\theta_j+\epsilon_j|^2/2-\lambda^2 >\theta_j^2/4-2\log n-\lambda^2>(\wt M^2/4-2)\log n-\lambda^2. 
$$
This yields for $c= \wt M^2/4-2$
$$
\frac{\pi( Y\C S)}{\pi( Y\C S^+)}\leq 4  \phi(\lambda)\e^{-c\log n +\lambda^2 }=\frac{4\sqrt{2\pi}}{n^{c}}\e^{\lambda^2/2}.
$$
Under the hierarchical prior $\eta\sim Beta(a,b)$, we have for $s=|S|$
$$
\Pi(S)=\frac{B(a+s,n-s+b)}{B(a,b)}=\frac{\Gamma(s+a)\Gamma(n-s+b)}{\Gamma(a+b+n)}\frac{\Gamma(a+b)}{\Gamma(a)\Gamma(b)}.
$$
The prior ratio for  $b=n+1$ satisfies
$$
\frac{\Pi(S )}{\Pi(S^+ )}=\frac{\Gamma(s+a)\Gamma(n-s+b)}{\Gamma(s+a+1)\Gamma(n-s-1+b)}=\frac{n-s+b}{s+a+1}\leq  2n. 
$$
Because the mapping $S\rightarrow S^+$ is one-to-one, we have for $C=8\sqrt{2\pi}$ 
\begin{align*}
\Pi(\mathcal S_j\C  Y)&=\sum_{S:j\notin S}\frac{\Pi(S\C  Y)}{\Pi(S^+\C Y)}\Pi(S^+\C Y)<\frac{C\e^{\lambda^2/2}}{n^{c-1}}\sum_{S:j\notin S}\Pi(S^+\C  Y)\\
&=\frac{C \e^{\lambda^2/2}}{n^{c-1}}\sum_{S^+:\exists S\,s.t.\, S^+=S\cup j}\Pi(S^+\C  Y)\leq \frac{C \e^{\lambda^2/2}}{n^{c-1}}.
\end{align*}
This means that the posterior probability of missing any signal  satisfies for $\lambda$ such that $\lambda^2\leq 2 d\log n$ for some $d>0$ and for $c>2+d$  
$$
\Pi\left(\exists j: |\theta_j|\neq 0\,\,\,\text{and}\,\,\,j\notin S\C  Y\right)\leq \sum_{j:|\theta_j|\neq 0}\Pi(\mathcal S_j\C Y)\leq \frac{s_n C \e^{\lambda^2/2}}{n^{c-1}}=o(1).
$$
This means that for any $\theta \in \Theta_n(s_n,\wt M)$, on the event $\mathcal A_n$, the posterior {\em does not undershoot}  $s_n=\|\theta \|_0$. In other words
$$
\sup\limits_{\theta \in \Theta_n(s_n,\wt M)}\Pi( s_n(\mu)< s_n\C  Y)\leq \sup\limits_{\theta \in \Theta_n(s_n,\wt M)}\Pi( s_n(\mu)< s_n\C  Y)\mathbb I (\mathcal A_n)+o(1)=o(1)
$$ 
and using Lemma \ref{lemma:posterior_odds} we have for $a=2$ and $b=n+1$
$$
 E_{ Y\C\theta}E \left[\log\left(\frac{1-\eta}{\eta }\right)\C Y\right]\leq P (\mathcal A_n)\log \left(\frac{2n}{s_n}\right) + \log(2n+1) P (\mathcal A_n^c)\lesssim \log(n/s_n).
$$

\section{Auxiliary Results for Sparse Normal Means}

\begin{lemma}\label{lemma:gauss_ratio}
For the Gauss hypergeometric function $F_2(a',b',c';z)$ defined in \eqref{eq:gaus},   the ratio   
\begin{equation}
f_\delta(a',b',c';z)=\frac{F_2(a',b'+\delta,c';-z)}{F_2(a',b',c';-z)}\label{eq:ratio_gauss}
\end{equation}
for $\delta>0$ is monotone increasing when $a'<0$  for any $z>-1$.
\end{lemma}
\proof 
We will prove this analogously as in Theorem 1 of \cite{karp}.
We denote with $A_\delta=\frac{\Gamma(c')}{\Gamma(b'+\delta)\Gamma(c'-b'-\delta)}$ and write
$$
f_\delta(a',b',c';z)=\frac{A_\delta \int  [\eta/(1-\eta)]^\delta q(\theta,z) \d\theta}{A_0 \int  q(\theta,z) \d\theta}
$$
where $q(\eta,z)=\eta^{b'-1}(1-\eta)^{c'-b'-1}(1+\eta z)^{-a'}$. By differentiating the ratio $f_\delta(z)$ with respect to $z$,
we find that the function $f_\delta(z)$ is monotone increasing for $a'<0$ if
$$
\int_0^1 h(\eta)\times q(\eta,z)\d\eta \int_0^1 f(\eta)\times q(\eta,z)\d\eta<\int_0^1 h(\eta)\times f(\eta)\times q(\eta,z)\d\eta\int_0^1 q(\eta,z)\d\eta
$$
where $h(\eta)=(\frac{\eta}{1-\eta})^\delta$ and $f(\eta)=\frac{\eta}{1+\eta z}$.
The function $q(\eta,z)$  is positive, while the functions $h(\eta)$ and $f(\eta)$ are monotone increasing for fixed $z>-1$ and $0 < \eta < 1$. 
Hence, the above inequality is an instance of the Chebyshev inequality  (\cite{mitrinov} Chapter IX, formula (1.1)). \qedhere

\begin{lemma}\label{lemma:aux_function}
For any $x>0$, we have for $0<c $
$$
\log\frac{\sqrt{1+c/x^2}+1}{2}<\frac{\sqrt{1+c/x^2}-1}{2}<c/x^2.
$$
\end{lemma}
\proof This follows from the fact that $\log(1+x)<x$ and $\sqrt{1+x}-1<x$ for $x>0$.
\begin{lemma}(Mills Ratio Bounds)\label{lem:mills}
For the Gaussian Mills ratio $R(x)=(1-\Phi(x))/\phi(x)$ we have for any $x>0$
\begin{equation}
\frac{2}{\sqrt{x^2+4}+x}< \frac{1-\Phi(x)}{\phi(x)}<\frac{2}{\sqrt{x^2+2}+x}.
\end{equation}
\end{lemma}
\proof This result shown is in \cite{birnbaum}.

\section{Auxiliary Results for Sparse Regression}
We will be using the following notation $\|X\|=\max_{1\leq j\leq p}\|X^j\|_2,$ where $X^j$ is the $j^{th}$ column of the matrix $X$. For a vector $\beta\in\R^p$ and a set of indices $S\subseteq\{1,\dots, p\}$
we denote with $\beta_S=\{\beta_i:i\in S\}$ the active subset. Similarly, with $S_\beta=\{i:\beta_i\neq 0\}$ we denote the support of the vector $\beta\in\R^p$.
Lastly, with $X_S$ we denote the sub-matrix $[X^j: j\in S]$ of columns that belong to $S$.

\begin{definition} (Definition 2.1 in \cite{castillo_vdv2})\label{def:compatible}
The compatibility number of a model $S\subset\{1,\dots, p\}$ is given by
$$
\phi(S)=\inf\left\{ \frac{\|X\beta\|_2|S|^{1/2}}{\|X\|\|\beta_S\|_1}: \|\beta_{S^c}\|_1\leq 7\|\beta_S\|_1,\beta_S\neq 0\right\}.
$$
\end{definition}

\begin{definition}(Definition 2.3 in \cite{castillo_vdv2})\label{def:smallest_scaled}
The smallest scaled singular value of dimension $s$ is defined as 
$$
\wt\phi(s)=\inf\left\{ \frac{\|X\beta\|_2}{\|X\|\|\beta\|_2}:0\neq |S_\beta|\leq s\right\}.
$$
\end{definition}

\begin{theorem}(Consistent Model Selection)\label{thm:consist}
Assume the prior  \eqref{eq:prior_dim} with $A_4>1$ and $a<A_4-1$ where $s_0\leq p^a$. For any $c_0>0$ and $s_n\lambda\sqrt{\log p}/\|X\|\rightarrow 0$  there exists a set $\mathcal D_n\subset\R^n$ and $a_1>0$ such that for any  $Y\in \mathcal D_n$  
$$
\sup_{\beta_0\in \wt\Theta(s_n,M):\phi(S_0)\geq c_0, \wt \psi(S_0)\geq c_0} \P(\beta:S_\beta\neq S_{0}\C Y)\lesssim \frac{1}{p^{a_1}}
$$
where $P(\mathcal D_n^c)\lesssim p^{-c_1}$ for some $c_1>0$.
\end{theorem}
\proof This follows from the proof of Corollary 1 in \cite{castillo_vdv2}.

\begin{lemma}\label{lemma:lb}
Assume that $X_0$ has been normalized such that $\|X_0\|=n$. Given  the Laplace prior $\pi\equiv\pi_{\lambda,S_0}(\beta)=(\lambda/2)^{s}\e^{-\lambda\|\beta\|_1}$ supported on $S_0$ with $s=|S_0|$, we have 
$$
\log \Lambda_{n,\beta_0,\pi}(Y,X)\geq -\lambda\|\beta_0\|_1-1/2-\lambda/n+s\log(\lambda/n)-\log s!
$$
almost surely.
\end{lemma}
 \proof Analogous to Lemma 2 in \cite{castillo_vdv2}.

\begin{lemma}\label{lemma:ub}
 Assume the Laplace prior $\pi\equiv\pi_{\lambda,S_0}(\beta)=(\lambda/2)^{s}\e^{-\lambda\|\beta\|_1}$ supported on $S_0$ of size $s=|S_0|$. Under the Assumption \ref{ass:design} we have for $\beta_0\in\wt\Theta(s_n,M)$
 $$
E_{Y\sim N(X\beta_0, I)}\log \Lambda_{n,\beta_0,\pi}(Y,X)\leq  \frac{s}{2}\left[1+2\lambda^2\, n(\log p)^d + C_0\sqrt{s/n \log p}+\log\left(\frac{8\pi}{n}\right)\right]-\lambda\|\beta_0\|_1. 
$$
where $C_0$ depends on $\wt\psi(S_0), \phi(S_0)$ and $b$. 
\end{lemma}
\proof
For a vector $\beta\in\R^{s }$, we denote  with $sign(\beta)$ an  $(s \times 1)$ vector of $sign(\beta_i)$. Next,  for $\pi\equiv\pi_{\lambda,S_0}(\beta)$ and
$
\Lambda_{n,\beta_0,\pi}(Y,X)\equiv  \int_{\beta}\Delta_{n,\beta,\beta_0}(Y,X)\pi_{S_0,\lambda}(\beta)\d\beta
$
with $\Delta_{n,\beta,\beta_0}(Y,X)$ defined in \eqref{eq:delta}
we can write
\begin{equation}
\Lambda_{n,\beta_0,\pi}(Y,X)=\e^{-1/2\|X\beta_0\|_2^2-\beta_0'X'(Y-X\beta_0)}(\lambda/2)^s\times \mathcal I,\label{eq:Lambda}
\end{equation}
where, using a shorthand notation $X_0=X_{S_0}$ and $\beta=\beta_{S_0}$,
\begin{equation}
 \mathcal I=\int\e^{-1/2\|X_0\beta\|_2^2+\beta'[X_0'Y-\lambda\,sign(\beta)]}\d\beta.\label{eq:I}
\end{equation}
For $\beta\in \R^s$, there are $2^s$ patterns $\{\pm 1\}^s$ of $sign(\beta)$ and we denote this set with $\Xi$. We write $\R^s=\cup_{\kappa\in \Xi} \mathcal O_{\kappa}$ where $\mathcal O_{\kappa}$ corresponds to an orthant that corresponds to the sign pattern indexed by $\kappa$.  For each $\kappa\in \Xi$ we define a shrinkage estimator
$$
\mu^\kappa=(X_0'X_0)^{-1}(X_0'Y- \lambda\bm 1_\kappa)
$$
where $\bm 1_\kappa$ corresponds to the $\pm1$ pattern    based on the orthant $\kappa$.
We denote with $\phi(x;\mu,\Sigma)$ the density of an $s$-variate normal distribution with mean $\mu$ and variance matrix $\Sigma$.
Then we decompose the integral  \eqref{eq:I} into
 \begin{equation}
 \mathcal I=\sum_{\kappa\in\, \Xi} J_\kappa,\quad\text{where}\quad J_{\kappa}=\frac{\int_{\mathcal O_\kappa} \phi(\beta;\mu^\kappa, (X_0'X_0)^{-1})\d\beta}{\phi(\mu^\kappa; 0, (X_0'X_0)^{-1})}\label{eq:J}.
 \end{equation}
Next, we denote with $A_{\kappa^*}=\{Y:  J_{\kappa^*}=\max_{\kappa}J_{\kappa}\}$ and write
\begin{align*}
E_{Y\sim N(X\beta_0, I)}\log \mathcal I  &\leq E_{Y\sim N(X\beta_0, I)} \sum_{\kappa^*\in\, \Xi}\mathbb I(Y\in A_{\kappa^*})\log (2^s J_{\kappa^*}).
\end{align*}
 Denote with $\mu=(X_0'X_0)^{-1}X_0'Y$ the MLE estimator for the true model $S_0$, then we have
 \begin{align*}
 J_\kappa\leq \frac{1}{\phi(\mu^{\kappa}; 0, (X_0'X_0)^{-1})} &=\frac{\e^{\mu'(X_0'X_0)\mu/2- \lambda\bm 1_{\kappa }'\mu+\lambda^2\bm 1_{\kappa }'(X_0'X_0) \bm 1_{\kappa }/2}}{(2\pi)^{-s/2}|X_0'X_0|^{1/2}}.
 \end{align*}
Then since 
 $
E \mu'(X_0'X_0)\mu=  tr(X_0(X_0'X_0)^{-1}X_0' E YY')
=\mathrm{rank}(X_0)/2+\|X \beta_0\|_2^2/2
$
we have
\begin{align}
E_{Y\sim N(X\beta_0, I)}\log I  \leq&  s/2\log(8\pi)-s/2\log n+s/2+\|X\beta_0\|_2^2/2+   \lambda^2\times ns(\log p)^d\nonumber \\
&- E_{Y\sim N(X\beta_0, I)} \sum_{\kappa^*\in\, \Xi}\mathbb I(Y\in A_{\kappa^*}) \lambda\bm 1_{\kappa^* }'\mu,\label{eq:theterm}
\end{align}
where we used the fact that  $\bm 1_{\kappa }'(X_0'X_0) \bm 1_{\kappa }\leq\lambda_{max}(X_0'X_0)s\leq ns(\log p)^d$ under Assumption \ref{ass:design}. 
Next, recall the definition of   $\kappa^*=\arg\max\limits_{\kappa\in\Xi} J_\kappa$ which is  a random variable in $Y$.
Now, we inspect the occurrence of an event $\{\bm 1_{\kappa^*} =\mathrm{sign}(\mu)\}$. 
We can rewrite $J_\kappa$ as
$$ 
 J_\kappa=\e^{\mu'(X_0'X_0)\mu/2} \int_{\mathcal O_\kappa}\e^{-(\beta-\mu)'(X_0'X_0)(\beta-\mu)/2-\lambda\|\beta\|_1}\d\beta.
$$ 
Since the function $\e^{-\lambda|\beta|}$ is symmetrical around the origin, its integral for each orthant $\mathcal O_{\kappa}$ is the same and equals $1/\lambda^s$. The orthant intergals change after reweighting by the Gaussian kernel $\e^{-(\beta-\mu)'(X_0'X_0)(\beta-\mu)/2}$. Filtering the Gaussian "likelihood" $N(\mu,(X_0'X_0)^{-1})$ through the Laplace prior has the effect of squashing the Gaussian distribution towards  (not across) coordinate axes and, depending on the magnitude of $\lambda$, creating ridgelines at the coordinate axes.
In other words, multiplying the Gaussian likelihood by a Laplace density does not change the ordering of Gaussian orthant probabilities.
The orthant $\mathcal O_{\kappa^*}$ which has the highest Gaussian orthant probability for the distribution $N(\mu,(X_0'X_0)^{-1})$ will also be the one for which $J_\kappa$ is the largest.
In addition, the orthant $\mathcal O_{\kappa^*}$ will very often be the one containing the mode $\mu$.
This will {\em always} be the case when $X_0'X_0=I_s$. It is reasonable to expect that when the correlation among the columns in $X_0$ is not too strong and/or when the signal $\beta_0$ is strong enough, this event will happen with overwhelmingly large probability. Below, we conclude that  Assumption \ref{ass:beta_min} and \ref{ass:design} are sufficient conditions for this to happen.

Denote  the ellipse $
\mathcal E_{\mu}(\chi^2_{s,1/2})=\{\beta\in\R^s:(\beta-\mu)'X_0'X_0(\beta-\mu)\leq \chi^2_{s,1/2}\}$ where $\chi^2_{s,1/2}$ is the median  of the Chi-squared distribution with $s$ degrees of freedom. This ellipse  contains $1/2$ of the Gaussian mass $N(\mu,(X_0'X_0)^{-1})$. If $\mathcal E_{\mu}(\chi^2_{s,1/2})\cap\mathcal O_\kappa=\mathcal E_{\mu}(\chi^2_{s,1/2})$, then  orthant  $\mathcal O_{\kappa}$ has the largest orthant probability and thereby also the integral $J_\kappa$, i.e. $\kappa=\kappa^*$. 
This happens, for instance, when 
\begin{equation}
|\mu_{j}|>c\equiv 2\sqrt{\chi^2_{s,1/2}/\lambda_{min}(X_0'X_0)}\label{eq:c}
\end{equation}
because  the axes of the ellipse are of the  length $c_j=2 \sqrt{\wt\lambda_j\times \chi^2_{s,1/2}}$, where $\wt\lambda_j$ is the $j^{th}$ eigenvalue   of $(X_0'X_0)^{-1}$. This means that 
\begin{equation}
\{|\mu_j|>c \}\subset \{sign(\mu_j)=1_{\kappa^*  j}\}\quad\forall j\in \{1,\dots, s\}\label{eq:embed}.
\end{equation}
 On the other hand, when $\mathcal E_{\mu}(\chi^2_{s,1/2})$  crosses the $j^{th}$ coordinate axis, then $sign(\mu_j)$ may or may not correspond to $1_{\kappa^*  j}$. 

 Next, we slice $\R^n$ into $2^s$ mutually exclusive sets depending on whether or not $|\mu_j|\geq c $.    Define by $I_m$ the set of all subsets  of $\{1,\dots,s \}$ of size $m$. 
 Then we have
$$
\R^n=\bigcup _{m=0}^s\bigcup _{I\in I_m }A(I),
$$
where  $A(I)=\{Y: |\mu_j|> c \,\,\text{for}\,\, j\in I \,\,\text{and}\,\, |\mu_j|\leq c \,\,\text{for}\,\,j\notin I\}$.
Due to \eqref{eq:embed},  for each $I\in I_m$ we have on the event $ \{Y\in A_{\kappa^* }\}\cap A(I)$  
$$
\bm 1_{\kappa^* }'\mu\geq  \sum_{j\in I }|\mu_j| - \sum_{j\notin I }|\mu_j|> \sum_{j\in I }|\beta_{0i}|-\sum_{j\in I }|\varepsilon^*_{i}|-  \sum_{j\notin I }c,
$$
where we have used the fact that $|\mu_j|=|\beta_{0j}+\varepsilon_j^*|>|\beta_{0j}|-|\varepsilon_j^*|$ for $\varepsilon^*=(X_0'X_0)^{-1}X_0'\varepsilon$.
Going back to \eqref{eq:theterm}, the term
$$
E_{Y\sim N(X\beta_0, I)}\sum_{\kappa^*\in\,\Xi}\mathbb I(Y\in A_{\kappa^*})  \bm 1_{\kappa^* }'\mu=  E_{Y\sim N(X\beta_0, I)}\sum_{m=0}^s\sum_{I\in I_m} \sum_{\kappa^*\in\,\Xi}\mathbb I(Y\in A_{\kappa^*}\cap A(I))  \bm 1_{\kappa^* }'\mu 
$$ can be lower-bounded by
\begin{equation}
 \sum_{m=0}^s\sum_{I\in I_m}\sum_{\kappa^*\in\Xi}P[Y\in A_{\kappa^*}\cap A(I)]\left( \sum_{j\in I }|\beta_{0j}| - \sum_{i\notin I }c\right)
-E\|\varepsilon^*\|_1> \sum_{j=1}^s |\beta_{0i}| P(|\mu_j|>c) -s\, c- E\|\varepsilon^*\|_1.
\end{equation}
Each $|\varepsilon_j^*|$ has a folded normal distribution with an expectation $\sqrt{2\sigma_j^2/\pi}$ where 
$\sigma_j^2$ is the $j^{th}$ diagonal element of 
$(X_0'X_0)^{-1}$ and satisfies $\sigma_j^2\leq \lambda_{max}[(X_0'X_0)^{-1}]=1/\lambda_{min}(X_0'X_0)$. This yields
$$
- E_{Y\sim N(X\beta_0, I)}\sum_{\kappa^*\in\Xi}\mathbb I(Y\in A_{\kappa^*})  \bm 1_{\kappa^* }'\mu\leq s\, c+ s\sqrt{2/(\pi\lambda_{min}(X_0'X_0))}-\|\beta_0\|_1+\sum_{j=1}^s |\beta_{0i}| P(|\mu_j|\leq c).
$$
  Now, we inspect  $P(|\mu_j|\leq c)$.  Because $|\mu_j|>|\beta_{0j}+\varepsilon^*_j|>|\beta_{0j}|-|\varepsilon^*_j|$ we have from Assumption \ref{ass:beta_min}  and  \ref{ass:design} that 
 $|\beta_{0j}|>\beta_{min}>c/(1-b)$ and thereby $|\beta_{0j}|-c>b|\beta_{0j}|$. Then   using a Gaussian tail bound we obtain
$$
P(|\mu_j|\leq c)\leq P(|\varepsilon^*_j|\geq |\beta_{0j}|-c)\leq 2\e^{-b^2\beta_{0j} ^2/(2\sigma_j^2)}.
$$
Then because $|x|\exp(-x^2/2)\leq 1$ we have
$$
|\beta_{0j}|P(|\mu_j|\leq c)\leq  |\beta_{0j}|\e^{-b^2\beta_{0j} ^2/(2\sigma_j^2)}\leq  \sigma_j/b \leq \frac{1}{b\sqrt{\lambda_{min}(X_0'X_0)}}.
$$

It is known \cite{johnson_kotz} that  $s-1<\chi^2_{s,1/2}< s$
which, from our assumptions, yields   for $M_0=M/[\wt\psi(S_0)^2\phi(S_0)]$ and $s\geq 2$
$$
\frac{1}{\sqrt{\lambda_{min}(X_0'X_0)}}<\frac{M_0(1-b)}{2\sqrt{\chi^2_{s,1/2}}}\sqrt{s/n\log p}<\frac{M_0(1-b)}{2}\sqrt{1/n\log p }.
$$
 A similar bound could also be obtained for $s=1$ using the approximation $\chi^2_{s,1/2}\approx s(1-2/(9s))^3$.
 This implies (using the definition of $c$ in \eqref{eq:c}) 
 \begin{align*}
 - E_{Y\sim N(X\beta_0, I)}\sum_{\kappa^*\in\Xi}\mathbb I(Y\in A_{\kappa^*})  \bm 1_{\kappa^* }'\mu&\leq s\left[ 2\sqrt{s}+\sqrt{2/\pi}+1/b\right]\frac{M_0(1-b)}{2}\sqrt{1/n\log p }-\|\beta_0\|_1\\
 &\leq C_0s\sqrt{s/n\log p}-\|\beta_0\|_1.
 \end{align*}
 Together with \eqref{eq:theterm} and \eqref{eq:Lambda}, this concludes the Lemma.

 \section{Proof of Theorem \ref{thm:final}}\label{sec:proof_thm:final}

We first focus on the set $\mathcal D_n$ from Theorem \ref{thm:consist}. Choosing $\lambda=\sqrt{n}/p$ (as in Example 9 in \cite{castillo_vdv2}) the assumptions in Theorem \ref{thm:consist} are satisfied for $a=1$ and  $A_4>2$ because $s\lambda\sqrt{\log p}/\sqrt{n}\leq n/p\sqrt{\log p}=o(1)$ where $s=|S_0|$. The set $\mathcal D_n$ is defined as an intersection of the set (6.12) in \cite{castillo_vdv2} and the set from the proof of Theorem 5 in \cite{castillo_vdv2}.
Define $\mathcal B=\{\beta\in\R^p:S_{\beta}= S_0\}$. 
Due to Pinsker's inequality we note that
\begin{equation}
\rho^{TV}_{n,m}(\beta_0, \hat p)^2 \leq 4 \times P(\mathcal D_n)+ 1/2\times \E_{Y\C\beta_0}\mathbb{I}(Y\in \mathcal D_n) KL(\pi(\wt Y\C \beta_0),\hat p(\wt Y\C Y)).\label{eq:pinsker}
\end{equation}
The second summand is the typical KL distance for a set of largely probable $Y$ and can be written as 
$$
1/2\times \E_{Y\C\beta_0}\E_{\wt Y\C\beta_0}\mathbb{I}(Y\in \mathcal D_n) \log \frac{\Lambda_{n,\beta_0,\pi}(Y,X)}{ \Lambda_{n+m,\beta_0,\pi}(Z,\bar X)}.
$$
Under the Spike-and-Slab prior with a Laplace slab, we can rewrite this term as follows
\begin{equation}
\E_{Y\C\beta_0}\mathbb{I}(Y\in \mathcal D_n) \E_{\wt Y\C\beta_0}\log \frac{ \int_{\mathcal B} \Delta_{n,\beta,\beta_0}(Y,X)\pi(\beta)\d\beta}{\int \Delta_{n+m,\beta,\beta_0}(Z,\bar X)\pi(\beta)\d\beta}+
\E_{Y\C\beta_0}\mathbb{I}(Y\in \mathcal D_n)\log \frac{1}{P(\mathcal B \C Y)}.\label{eq:two_terms}
\end{equation}
Next, we  write (using  Jensen's inequality $E\log X\leq \log E X$) using Theorem \ref{thm:consist}
\begin{align*}
\E_{Y\C\beta_0}\mathbb{I}(Y\in \mathcal D_n)\log \frac{1}{P(\mathcal B \C Y)}&=P(\mathcal D_n)+\E_{Y\C\beta_0}\mathbb{I}(Y\in \mathcal D_n)\frac{P(\mathcal B^c \C Y)}{1-P(\mathcal B^c \C Y)}\\
&<P(Y\in\mathcal D_n)(1+1/(p^{a_1}-1)). 
\end{align*}
The first term in \eqref{eq:two_terms}  can be upper bounded by the  KL risk  (restricted to $\mathcal D_n$) obtained under the oracle prior $\pi_{S_0,\lambda}(\beta)$
$$
 \rho^O_{n,m}(\beta_0, \hat p)\equiv \E_{Y\C\beta_0}\mathbb{I}(Y\in \mathcal D_n)\E_{\wt Y\C\beta_0}\log \frac{ \int \Delta_{n,\beta,\beta_0}(Y,X)\pi_{S_0 ,\lambda}(\beta)\d\beta}{ 
 \int \Delta_{n+m,\beta,\beta_0}(Z,\bar X)\pi_{S_0,\lambda}(\beta)\d\beta}
$$
The term $ \rho^O_{n,m}(\beta_0, \hat p)$ integrates over the high-probability $Y$'s inside $\mathcal D_n$.
For the log-numerator, we apply Lemma \ref{lemma:ub} which presents an upper bound for the (normalized) marginal likelihood under the Laplace prior.
From the proof of Lemma \ref{lemma:ub} it can be seen that  restricting the integration $\mathcal D_n$  {\em does not} increase the upper bound on the entire marginal likelihood. We can thus apply this upper bound when confined to $\mathcal D_n$.
For the log-denominator, we apply the lower bound Lemma \ref{lemma:lb} which holds uniformly (and almost surely) for all $Y$.
Altogether, this yields
\begin{align*}
\rho^O_{n,m}(\beta_0, \hat p)\leq&   \frac{s}{2}\left[1+2\wt C\lambda^2\, n\log^{d}p + C_0\sqrt{s/n \log p}+\log\left(\frac{8\pi}{n}\right)\right] \\
&+1/2+\lambda/(n+m)+s\log[(n+m)/\lambda]+\log s!. 
\end{align*}
With $\lambda=\sqrt{n}/p$ and  because $n/p\lesssim 1/\sqrt{\log (p)}$ we have $\lambda^2n\log^{d}p\lesssim \log^{d-1 } p$ we have
$$
\rho^O_{n,m}(\beta_0, \hat p)\lesssim s[\log^{d-1\vee 1} p\vee \log(1+m/n)]
$$
Combined with \eqref{eq:pinsker} we conclude that $\rho^{TV}_{n,m}(\beta_0, \hat p)^2\lesssim \eta_n$. Moreover, throughout the proof we showed that uniformly on $\mathcal D_n$, the KL divergence can be bounded by a multiple of $\eta_n$. The probability of the KL divergence exceeding a multiple of $\eta_n$ is thus smaller than the probability of the complement of set $\mathcal D_n$.



\begin{thebibliography}{}

\bibitem[\protect\citeauthoryear{Aitchison}{Aitchison}{1975}]{aitchison}
Aitchison, J. (1975).
\newblock Goodness of prediction fit.
\newblock {\em Biomerika\/}~{\em 62}, 547--554.

\bibitem[\protect\citeauthoryear{Bai, Ro\v{c}kov\'{a}, and George}{Bai
  et~al.}{2020}]{bai}
Bai, R., V.~Ro\v{c}kov\'{a}, and E.~George (2020).
\newblock {Spike-and-Slab Meets LASSO: A Review of the Spike-and-Slab LASSO}.
\newblock {\em ar{X}iv:2010.06451\/}, 1--30.

\bibitem[\protect\citeauthoryear{Bhattacharya, Pati, Pillai, and
  Dunson}{Bhattacharya et~al.}{2015}]{dirichlet_laplace}
Bhattacharya, A., D.~Pati, N.~Pillai, and D.~Dunson (2015).
\newblock Dirichlet-{L}aplace priors for optimal shrinkage.
\newblock {\em Journal of the American Statistical Association\/}~{\em 110},
  1479--1490.

\bibitem[\protect\citeauthoryear{Birnbaum}{Birnbaum}{1942}]{birnbaum}
Birnbaum (1942).
\newblock An inequality for {M}ill's ratio.
\newblock {\em Annals of Mathematical Statistics\/}~{\em 13}, 245--246.

\bibitem[\protect\citeauthoryear{Carvalho and Polson}{Carvalho and
  Polson}{2010}]{carvalho}
Carvalho, C. and N.~Polson (2010).
\newblock The horseshoe estimator for sparse signals.
\newblock {\em Biometrika\/}~{\em 97}, 465--480.

\bibitem[\protect\citeauthoryear{Castillo, Schmidt-Hieber, and van~der
  Vaart}{Castillo et~al.}{2015}]{castillo_vdv2}
Castillo, I., J.~Schmidt-Hieber, and A.~van~der Vaart (2015).
\newblock Bayesian linear regression with sparse priors.
\newblock {\em The Annals of Statistics\/}~{\em 43}, 1986--2018.

\bibitem[\protect\citeauthoryear{Castillo and van~der Vaart}{Castillo and
  van~der Vaart}{2012}]{castillo_vdv}
Castillo, I. and A.~van~der Vaart (2012).
\newblock Needles and straw in a haystack: Posterior concentration for possibly
  sparse sequences.
\newblock {\em The Annals of Statistics\/}~{\em 40}, 2069--2101.

\bibitem[\protect\citeauthoryear{Deshpande, Ro\v{c}kov\'{a}, and
  George}{Deshpande et~al.}{2019}]{deshpande}
Deshpande, S., Ro\v{c}kov\'{a}, and E.~George (2019).
\newblock Simultaneous variable and covariance selection with the multivariate
  spike-and-slab lasso.
\newblock {\em Journal of Computational and Graphical Statistics\/}~{\em 28},
  921--931.

\bibitem[\protect\citeauthoryear{George, Liang, and Xu}{George
  et~al.}{2006}]{george_liang_xu2}
George, E., F.~Liang, and X.~Xu (2006).
\newblock Improved minimax predictive densities under kullback-leibler loss.
\newblock {\em The Annals of Statistics\/}~{\em 34}, 78--91.

\bibitem[\protect\citeauthoryear{George, Liang, and Xu}{George
  et~al.}{2012}]{george_liang_xiu}
George, E., F.~Liang, and X.~Xu (2012).
\newblock From minimax shrinkage estimation to minimax shrinkage prediction.
\newblock {\em Statistical Science\/}~{\em 27}, 1102--1130.

\bibitem[\protect\citeauthoryear{George and Xu}{George and
  Xu}{2008}]{george_xu}
George, E. and X.~Xu (2008).
\newblock Predictive density estimation for multiple regression.
\newblock {\em Econometric Theory\/}~{\em 24}, 528--544.

\bibitem[\protect\citeauthoryear{Hans}{Hans}{2009}]{hans_biometrika}
Hans, C. (2009).
\newblock Bayesian {LASSO} regression.
\newblock {\em Biometrika\/}~{\em 96}, 835--845.

\bibitem[\protect\citeauthoryear{Hoffmann and Schmidt-Hieber}{Hoffmann and
  Schmidt-Hieber}{2015}]{hoffmann}
Hoffmann, M.~Rousseau, J. and J.~Schmidt-Hieber (2015).
\newblock On adaptive posterior concentration rates.
\newblock {\em The Annals of Statistics\/}~{\em 43}, 2259--2295.

\bibitem[\protect\citeauthoryear{Karp and Sitnik}{Karp and Sitnik}{2009}]{karp}
Karp, D. and S.~Sitnik (2009).
\newblock Inequalities and monotonicity of ratios for generalized
  hypergeometric function.
\newblock {\em Journal of Approximation Theory\/}~{\em 161}, 337--352.

\bibitem[\protect\citeauthoryear{Komaki}{Komaki}{2001}]{komaki}
Komaki, F. (2001).
\newblock A shrinkage predictive distribution for multivariate normal
  observables.
\newblock {\em Biometrika\/}~{\em 88}, 859--864.

\bibitem[\protect\citeauthoryear{Liang and Barron}{Liang and
  Barron}{2004}]{liang_barron}
Liang, F. and A.~Barron (2004).
\newblock Exact minimax strategies for predictive denstiy estimation, data
  compression, and model selection.
\newblock {\em IEEE Transactions of Information Theory\/}~{\em 50}, 2708--2723.

\bibitem[\protect\citeauthoryear{Mitchell and Beauchamp}{Mitchell and
  Beauchamp}{1988}]{mitchell_beauchamp}
Mitchell, T.~J. and J.~J. Beauchamp (1988).
\newblock {B}ayesian variable selection in linear regression.
\newblock {\em Journal of the American Statistical Association\/}~{\em 83},
  1023--1032.

\bibitem[\protect\citeauthoryear{Mitrinov\'{c}, Pecari\'{c}, and
  Fink}{Mitrinov\'{c} et~al.}{1993}]{mitrinov}
Mitrinov\'{c}, D., J.~Pecari\'{c}, and M.~Fink (1993).
\newblock {\em Classical and new inequalities in Analysis}.
\newblock Kluwer Academic Publishers.

\bibitem[\protect\citeauthoryear{Moran, Ro\v{c}kov\'{a}, and George}{Moran
  et~al.}{2021}]{moran_biclus}
Moran, G., Ro\v{c}kov\'{a}, and E.~George (2021).
\newblock {Spike-and-Slab LASSO} biclustering.
\newblock {\em Annals of Applied Statistics\/}~{\em 15}, 148--173.

\bibitem[\protect\citeauthoryear{Mukherjee and Johnstone}{Mukherjee and
  Johnstone}{2015}]{muk1}
Mukherjee, G. and I.~Johnstone (2015).
\newblock Exact minimax estimation of the predictive density in sparse gaussian
  models.
\newblock {\em The Annals of Statistics\/}~{\em 43}, 81--106.

\bibitem[\protect\citeauthoryear{Mukherjee and Johnstone}{Mukherjee and
  Johnstone}{2022}]{muk2}
Mukherjee, G. and I.~Johnstone (2022).
\newblock On minimax optimality of sparse {Bayes} predicitve density estimates.
\newblock {\em The Annals of Statistics\/}~{\em 50}, 81--106.

\bibitem[\protect\citeauthoryear{Nie and Ro\v{c}kov\'{a}}{Nie and
  Ro\v{c}kov\'{a}}{2022}]{nie}
Nie, L. and V.~Ro\v{c}kov\'{a} (2022).
\newblock Bayesian bootstrap {Spike-and-Slab LASSO}.
\newblock {\em Journal of the American Statistical Association (in
  press)\/}~{\em 1}, 1--50.

\bibitem[\protect\citeauthoryear{Park and Casella}{Park and
  Casella}{2008}]{casella}
Park, T. and G.~Casella (2008).
\newblock The {B}ayesian {LASSO}.
\newblock {\em Journal of the American Statistical Association\/}~{\em 103},
  681--686.

\bibitem[\protect\citeauthoryear{Ray and Szabo}{Ray and
  Szabo}{2022}]{ray_szabo}
Ray, K. and B.~Szabo (2022).
\newblock Variational {B}ayes for high-dimensional linear regression with
  sparse priors.
\newblock {\em Journal of the American Statistical Association\/}~{\em 117},
  1270--1281.

\bibitem[\protect\citeauthoryear{Ro\v{c}kov\'{a}}{Ro\v{c}kov\'{a}}{2018}]{rockova}
Ro\v{c}kov\'{a} (2018).
\newblock Bayesian estimation of sparse signals with a continuous
  spike-and-slab prior.
\newblock {\em The Annals of Statistics\/}~{\em 46}, 401--437.

\bibitem[\protect\citeauthoryear{Ro\v{c}kov\'{a} and George}{Ro\v{c}kov\'{a}
  and George}{2016}]{RG15}
Ro\v{c}kov\'{a}, V. and E.~George (2016).
\newblock Fast {B}ayesian factor analysis via automatic rotations to sparsity.
\newblock {\em Journal of the American Statistical Association\/}~{\em 111},
  1608--1622.

\bibitem[\protect\citeauthoryear{Ro\v{c}kov\'{a} and George}{Ro\v{c}kov\'{a}
  and George}{2018}]{RG14b}
Ro\v{c}kov\'{a}, V. and E.~George (2018).
\newblock The {Spike-and-Slab} {LASSO}.
\newblock {\em Journal of the American Statistical Association\/}~{\em 113},
  431--444.

\bibitem[\protect\citeauthoryear{Ro\v{c}kov\'{a} and Rousseau}{Ro\v{c}kov\'{a}
  and Rousseau}{2023}]{rockova_rousseau}
Ro\v{c}kov\'{a}, V. and J.~Rousseau (2023).
\newblock Ideal {B}ayesian spatial adaptation.
\newblock {\em Journal of the American Statistical Association (In Press)\/},
  1--80.

\bibitem[\protect\citeauthoryear{Tibshirani}{Tibshirani}{1994}]{tibshirani_lasso}
Tibshirani, R. (1994).
\newblock Regression shrinkage and selection via the {LASSO}.
\newblock {\em Journal of the Royal Statistical Society. Series B\/}~{\em 58},
  267--288.

\bibitem[\protect\citeauthoryear{van~der Pas, Kleijn, and van~der
  Vaart}{van~der Pas et~al.}{2014}]{van_der_pas}
van~der Pas, S., B.~Kleijn, and A.~van~der Vaart (2014).
\newblock The horseshoe estimator: Posterior concentration around nearly black
  vectors.
\newblock {\em Electronic Journal of Statistics\/}~{\em 8}, 2585--2618.

\end{thebibliography}

\end{document}